\documentclass[a4paper,reqno,11pt]{amsart}

\usepackage[american]{babel}

\usepackage[utf8]{inputenc}
\usepackage[T1]{fontenc}
\usepackage{microtype}
\usepackage[a4paper, scale={0.72,0.74}, marginratio={1:1}, footskip=7mm, headsep=10mm]{geometry}

\usepackage{indentfirst}

\usepackage{emptypage}  

\usepackage{perpage}

\usepackage{titlesec}
\titleformat*{\section}{\scshape\Large\bfseries}
\titleformat*{\subsection}{\scshape\large\bfseries}

\usepackage{enumitem}

\usepackage{fixltx2e}

\usepackage{color}
\newcommand\blue{\textcolor{blue}}

\usepackage[normalem]{ulem}

\usepackage{multicol}

\usepackage{hyperref}

\usepackage{booktabs}
\usepackage{graphicx}
\usepackage{tabularx}
\usepackage{caption,subcaption}

\usepackage{longtable}
\usepackage[vcentermath]{youngtab}
\usepackage{tikz,tikz-cd}
\usetikzlibrary{matrix, calc, arrows, decorations.markings, positioning, backgrounds}
\tikzset{>=stealth}

\usepackage{amsmath,amssymb,amsthm}
\usepackage{stmaryrd,mathtools,mathdots}
\usepackage{eufrak}
\usepackage{wasysym}
\allowdisplaybreaks[3]

\usepackage{accents}

\usepackage{bm,bbm}

\usepackage{braket}  

\numberwithin{equation}{section}

\usepackage{subdepth}

\makeatletter
\renewcommand{\@secnumfont}{\bfseries}
 
\renewcommand\section{\@startsection{section}{1}%
\z@{.7\linespacing\@plus\linespacing}{.5\linespacing}%
{\large\scshape\bfseries\centering}}

\renewcommand\subsection{\@startsection{subsection}{2}%
  \z@{.5\linespacing\@plus.7\linespacing}{-.5em}%
  {\bfseries\scshape}}
\makeatother

\newenvironment{myitemize}{%
\begin{list}{$\bullet$}%
 	{%
	\setlength{\itemsep}{0.4em}%
	\setlength{\topsep}{0.5em}%
	\setlength\leftmargin{2.45em}%
	\setlength\labelwidth{2.05em}%
	\setlength{\labelsep}{0.4em}%
	}%
	}%
{\end{list}}

\renewenvironment{itemize}{
\begin{myitemize}}%
{\end{myitemize}}

\MakePerPage{footnote}

\makeatletter
\newcommand*{\myfnsymbolsingle}[1]{%
  \ensuremath{%
    \ifcase#1% 0
    \or % 1
      \dagger
    \else % >= 2
      \@ctrerr  
    \fi
  }%   
}   
\makeatother

\usepackage{alphalph}
\newalphalph{\myfnsymbolmult}[mult]{\myfnsymbolsingle}{}

\makeatletter
\newcommand \Dotfill {\leavevmode \leaders \hb@xt@ 6pt{\hss .\hss }\hfill \kern \z@}
\makeatother

\makeatletter
\def\@tocline#1#2#3#4#5#6#7{\relax
  \ifnum #1>\c@tocdepth % then omit
  \else
    \par \addpenalty\@secpenalty\addvspace{#2}%
    \begingroup \hyphenpenalty\@M
    \@ifempty{#4}{%
      \@tempdima\csname r@tocindent\number#1\endcsname\relax
    }{%
      \@tempdima#4\relax
    }%
    \parindent\z@ \leftskip#3\relax \advance\leftskip\@tempdima\relax
    \rightskip\@pnumwidth plus4em \parfillskip-\@pnumwidth
    #5\leavevmode\hskip-\@tempdima
      \ifcase #1
       \or\or \hskip 1.65em \or \hskip 3.3em \else \hskip 4.95em \fi%
      #6\nobreak\relax
    \Dotfill
    \hbox to\@pnumwidth{\@tocpagenum{#7}}\par
    \nobreak
    \endgroup
  \fi}
\makeatother

\makeatletter
\def\l@section{\@tocline{1}{0pt}{1pc}{}{\scshape}}
\renewcommand{\tocsection}[3]{%
\indentlabel{\@ifnotempty{#2}{\ignorespaces#1 #2.\hskip 0.7em}}#3}
\def\l@subsection{\@tocline{2}{0pt}{1pc}{5pc}{}}

\def\l@subsubsection{\@tocline{3}{0pt}{1pc}{7pc}{}}

\makeatother

\theoremstyle{plain}
\newtheorem{theorem}{Theorem}[section]

\newtheorem{proposition}[theorem]{Proposition}
\newtheorem{corollary}[theorem]{Corollary}

\theoremstyle{definition}
\newtheorem{definition}[theorem]{Definition}
\newtheorem{remark}[theorem]{Remark}

\DeclarePairedDelimiter\abs{\lvert}{\rvert} % absolute value
\makeatletter
\let\oldabs\abs
\def\abs{\@ifstar{\oldabs}{\oldabs*}}

\DeclarePairedDelimiterX{\norm}[1]{\lVert}{\rVert}{#1} % norm
\let\oldnorm\norm
\def\norm{\@ifstar{\oldnorm}{\oldnorm*}}

\DeclarePairedDelimiterX{\ceil}[1]{\lceil}{\rceil}{#1} % ceiling
\let\oldceil\ceil
\def\ceil{\@ifstar{\oldceil}{\oldceil*}}

\DeclarePairedDelimiterX{\floor}[1]{\lfloor}{\rfloor}{#1} % floor
\let\oldfloor\floor
\def\floor{\@ifstar{\oldfloor}{\oldfloor*}}
\makeatother

\newcommand{\bbC}{{\ensuremath{\mathbb C}} }

\newcommand{\bbP}{{\ensuremath{\mathbb P}} }

\newcommand{\cA}{{\ensuremath{\mathcal A}} }
\newcommand{\cB}{{\ensuremath{\mathcal B}} }

\newcommand{\gl}{\lambda}

\newcommand\sff{\mathsf f}

\DeclareMathOperator{\Tr}{Tr} % trace
\DeclareMathOperator{\sign}{sgn} % sign function
\DeclareMathOperator{\shape}{sh} % shape
\DeclareMathOperator{\type}{type} % type

\newcommand{\diff}{\mathop{}\!\mathrm{d}}
\newcommand{\dd}{\mathop{}\!\mathrm{d}}
\renewcommand{\P}{\mathbb{P}} % probability
\newcommand{\1}{\mathbbm{1}} % indicator function
\renewcommand{\i}{\mathrm{i}} % imaginary unit
\DeclareMathOperator{\e}{\mathrm{e}} % Euler number
\newcommand{\RSK}{\mathsf{RSK}} % RSK
\DeclareMathOperator{\GL}{GL} % group GL
\DeclareMathOperator{\SO}{SO} % group SO
\DeclareMathOperator{\Sp}{Sp} % group Sp
\DeclareMathOperator{\schur}{s} % Schur fn
\renewcommand{\sp}{\operatorname{sp}} % sp Schur fn
\DeclareMathOperator{\so}{so} % so Schur fn
\DeclareMathOperator{\Pf}{Pf} % Pfaffian
\DeclareMathOperator{\Ai}{Ai} % Airy function
\newcommand{\CB}{\mathrm{CB}}
\newcommand{\BC}{\mathrm{BC}}
\newcommand{\trans}[1]{\mathrm{s}^{\scriptscriptstyle\CB}_{#1}}
\DeclareMathOperator{\atyp}{\mathfrak{a}}
\newcommand{\DB}{\mathrm{DB}}
\newcommand{\transDB}[1]{\mathrm{s}^{\scriptscriptstyle\DB}_{#1}}
\DeclareMathOperator{\oddrows}{oddrows}

\renewcommand{\emptyset}{\varnothing}
\newcommand{\N}{\mathbb{N}} % natural numbers
\newcommand{\Z}{\mathbb{Z}} % integer numbers
\newcommand{\R}{\mathbb{R}} % real numbers
\newcommand{\C}{\mathbb{C}} % complex numbers
\renewcommand{\epsilon}{\varepsilon}
\renewcommand{\rho}{\varrho}
\renewcommand{\phi}{\varphi}

\DeclareMathSymbol{\widehatsym}{\mathord}{largesymbols}{"62}

\renewcommand{\tilde}{\widetilde} % wider tilde

\makeatletter
\newcommand{\doubletilde}[1]{{%
  \mathpalette\double@tilde{#1}%
}}
\newcommand{\double@tilde}[2]{%
  \sbox\z@{$\m@th#1\tilde{#2}$}%
  \ht\z@=.9\ht\z@
  \tilde{\box\z@}%
}
\makeatother

\newenvironment{myenumerate}{%
\renewcommand{\theenumi}{(\roman{enumi})}%
\renewcommand{\labelenumi}{\theenumi}%
\begin{list}{\labelenumi}
	{%
	\setlength{\itemsep}{0.4em}%
	\setlength{\topsep}{0.5em}%
	\setlength\leftmargin{2.45em}%
	\setlength\labelwidth{2.05em}%
	\setlength{\labelsep}{0.4em}%
	\usecounter{enumi}%
	}%
	}%
{\end{list}
}
\renewenvironment{enumerate}{
\begin{myenumerate}}%
{\end{myenumerate}}

  {\left\lbrace\begin{array}{@{}l@{}}}%
  {\end{array}\right.}

\makeatletter
\newsavebox{\mybox}\newsavebox{\mysim}
\newcommand{\distras}[1]{%
  \savebox{\mybox}{\hbox{\kern3pt$\scriptstyle#1$\kern3pt}}%
  \savebox{\mysim}{\hbox{$\sim$}}%
  \mathbin{\overset{#1}{\kern\z@\resizebox{\wd\mybox}{\ht\mysim}{$\sim$}}}%
}
\makeatother

\newcommand\symSq{%
\begin{tikzpicture}[scale=0.17]
\draw (0,0) -- (1,0) -- (1,1) -- (0,1) -- cycle;
\draw (1,0) -- (0,1);
\end{tikzpicture}
}
\newcommand\antisymSq{%
\begin{tikzpicture}[scale=0.17]
\draw (0,0) -- (1,0) -- (1,1) -- (0,1) -- cycle;
\draw (0,0) -- (1,1);
\end{tikzpicture}
}
\newcommand\doubleSymSq{%
\begin{tikzpicture}[scale=0.17]
\draw (0,0) -- (1,0) -- (1,1) -- (0,1) -- cycle;
\draw (0,0) -- (1,1);
\draw (1,0) -- (0,1);
\end{tikzpicture}
}
\newcommand\pointToLine{%
\begin{tikzpicture}[scale=0.17]
\draw (0,0) -- (0,1) -- (1,1) -- cycle;
\end{tikzpicture}
}

\newcommand\pointToLineSym{%
\begin{tikzpicture}[scale=0.17]
\draw (0,0) -- (0,1) -- (1,1) -- cycle;
\draw (0,1) -- (0.5,0.5);
\end{tikzpicture}
}

\DeclareRobustCommand{\antisymSqProtect}{\texorpdfstring{\antisymSq}{replacement text}}

\DeclareRobustCommand{\symSqProtect}{\texorpdfstring{\symSq}{replacement text}}

\DeclareRobustCommand{\doubleSymSqProtect}{\texorpdfstring{\doubleSymSq}{replacement text}}

\newcommand{\GT}{\mathrm{GT}} % Gelfand-Tsetlin pattern
\newcommand{\spGT}{\mathrm{spP}} % sp Gelfand-Tsetlin pattern
\newcommand{\oGT}{\mathrm{oP}} % orth Gelfand-Tsetlin pattern
\newcommand{\soGT}{\mathrm{soP}} % split orth Gelfand-Tsetlin pattern

\begin{document}

\title[Transition between characters of classical groups]{Transition between characters of classical groups, \\ decomposition of Gelfand-Tsetlin patterns \\ and last passage percolation}

\author[E.~Bisi]{Elia Bisi}
\address{Elia Bisi\\
Institut f\"ur Stochastik und Wirtschaftsmathematik\\
Technische Universit\"at Wien\\
E 105-07\\
Wiedner Hauptstrasse 8-10\\
1040 Wien\\
Austria}
\email{elia.bisi@tuwien.ac.at}

\author[N.~Zygouras]{Nikos Zygouras}
\address{Nikos Zygouras \\
Mathematics Institute\\
Zeeman building\\
University of Warwick\\
Coventry CV4 7AL\\
UK}
\email{n.zygouras@warwick.ac.uk}

\thanks{Elia Bisi was supported by ERC Advanced Grant \emph{IntRanSt} - 669306.
Nikos Zygouras was supported by EPSRC via grant EP/R024456/1.}

\keywords{Symmetric functions, Schur polynomials, symplectic characters, orthogonal characters, interpolating Schur polynomials, Weyl character formula, last passage percolation, RSK correspondence, Tracy-Widom distributions}
\subjclass[2010]{Primary: 05E05, 60Cxx, 05E10, 82B23}

\begin{abstract}
We study the combinatorial structure of the irreducible characters of the classical groups $\GL_n(\bbC)$, $\SO_{2n+1}(\bbC)$, $\Sp_{2n}(\bbC)$, $\SO_{2n}(\bbC)$ and the ``non-classical'' odd symplectic 
group $\Sp_{2n+1}(\C)$, finding new connections to the probabilistic model of Last Passage Percolation (LPP).
Perturbing the expressions of these characters as generating functions of Gelfand-Tsetlin patterns, we produce two families of symmetric polynomials that interpolate between characters of $\Sp_{2n}(\bbC)$ and $\SO_{2n+1}(\bbC)$ and between characters of $\SO_{2n}(\bbC)$ and $\SO_{2n+1}(\bbC)$.
We identify the first family as a one-parameter specialization of Koornwinder polynomials, for which we thus provide a novel combinatorial structure; on the other hand, the second family appears to be new.
We next develop a method of Gelfand-Tsetlin pattern decomposition to establish identities between all these polynomials that, in the case of irreducible characters, can be viewed as branching rules.
Through these formulas we connect orthogonal and symplectic characters, and more generally the interpolating polynomials, to LPP models with various symmetries, thus going beyond the link with classical Schur polynomials originally found by Baik and Rains (Duke Math.\ J., 2001).
Taking the scaling limit of the LPP models, we finally provide an explanation of why the Tracy-Widom GOE and GSE distributions from random matrix theory admit formulations in terms of both Fredholm determinants and Fredholm Pfaffians.
\end{abstract}

\maketitle

\tableofcontents

\addtocounter{section}{0}

\section{Introduction}
\label{sec:intro}

Characters of irreducible polynomial representations of complex classical groups, also known as Schur polynomials, are symmetric (Laurent) polynomials in variables $x=(x_1,\dots,x_n)$ indexed by partitions, half-partitions, signed partitions or signed half-partitions\footnote{For the precise definitions of (signed) (half-)partitions, see the beginning of Section~\ref{sec:patterns}.}.
They are usually classified according to the \emph{type} of the associated Lie algebras and root systems~\cite{FH91}: \begin{itemize}
\item
type A: characters $\schur^{(n)}_{\lambda}(x)$ of the \emph{general linear group} $\GL_n(\C)$, i.e.\ standard Schur polynomials, indexed by a partition $\lambda = (\lambda_1,\dots,\lambda_n)$.
They are symmetric in their variables $x_1,\dots,x_n$.
\item
type B: characters $\so^{(2n+1)}_{\lambda}(x)$ of the \emph{odd orthogonal group} $\SO_{2n+1}(\C)$, indexed by a partition or half-partition $\lambda$.
They are invariant under permutation of their variables and inversion of any of them, i.e.\ $x_i\mapsto x_i^{-1}$.
\item
type C: characters $\sp^{(2n)}_{\lambda}(x)$ of the \emph{symplectic group} $\Sp_{2n}(\C)$, indexed by a partition $\lambda$.
They have the same invariance properties as the characters of type B.
\item
type D: characters $\so^{(2n)}_{\lambda_\epsilon}(x)$ of the \emph{even orthogonal group} $\SO_{2n}(\C)$, indexed by a signed partition or signed half-partition $\lambda_{\epsilon} = (\lambda_1,\dots,\lambda_{n-1}, \epsilon\lambda_n)$ with sign $\epsilon\in\{+,-\}$.
They are symmetric and invariant under inversion of an \emph{even number} of their variables.
\end{itemize}
All these polynomials are often expressed in terms of their \emph{Weyl character formula}~\cite{FH91}, which reads as a ratio of determinants.
Their determinantal structure also emerges through the so-called \emph{Jacobi-Trudi identities} and \emph{Giambelli identities}~\cite{Mac95, FK97}.
However, we will mainly work with the combinatorial interpretation of the characters as generating functions of \emph{Gelfand-Tsetlin patterns} and other similar patterns composed of interlacing partitions~\cite{Pro94,Mac95}.
Further commonly used combinatorial definitions involve Young tableaux instead of patterns~\cite{Sun90b}.
We will be also interested in the characters $\sp^{(2n+1)}_{\lambda}(x)$ of the \emph{odd symplectic group} $\Sp_{2n+1}(\C)$; such a group, introduced by Proctor~\cite{Pro88}, is not counted among the classical groups, but its characters are also given as generating functions of patterns and naturally fit our framework.
For a review of all the aforementioned characters and details about our notation, see Section~\ref{sec:patterns}.

\vskip 1.5mm

Perturbing the pattern representation of the symplectic and orthogonal characters, we will derive two families of interpolating symmetric polynomials, indexed by a partition or half-partition $\lambda$ (see Section~\ref{sec:transition}).
This has been motivated by the study of the last passage percolation model, as will be discussed later in the introduction.

The polynomials of the first family, which we will here refer to as \emph{$\CB$-interpolating Schur polynomials} and denote by $\trans{\lambda}(x; \beta)$, will be defined as weighted generating functions of \emph{split orthogonal patterns} via a tuning parameter $\beta$.
They interpolate between characters of type C and B, in the sense that $\trans{\lambda}(x; 0)=\sp^{(2n)}_{\lambda}(x)$ and $\trans{\lambda}(x; 1) = \so^{(2n+1)}_{\lambda}(x)$.
Via a combinatorial bijection between certain classes of split orthogonal and symplectic patterns (see Subsection~\ref{subsec:bijectionPatterns}), we will establish the ``Weyl character formula''
\begin{equation}
\label{eq:BCWeyl_intro}
\trans{\lambda}(x;\beta)
= \frac{\underset{1\leq i,j\leq n}{\det} \left(
x_j^{\lambda_i + n-i+1} - x_j^{-(\lambda_i + n-i+1)}
+ \beta \left[x_j^{\lambda_i + n-i} - x_j^{-(\lambda_i + n-i)}\right]
\right)}
{\underset{1\leq i,j\leq n}{\det} \left( x_j^{n-i+1} - x_j^{-(n-i+1)} \right)} \, ,
\end{equation}
valid for any partition $\lambda$ (a similar expression holds when $\lambda$ is a half-partition, see Theorem~\ref{thm:transitionWeyl}).
A consequence of such a determinantal expression is that this family can be identified with a one-parameter specialization of Koornwinder polynomials (see Subsection \ref{subsec:CBtransition}), thus providing a previously unknown combinatorial structure for the latter.

We will also introduce a family of \emph{$\DB$-interpolating Schur polynomials} $\transDB{\lambda}(x; \alpha)$, defining them as weighted generating functions of \emph{orthogonal patterns} via a tuning parameter $\alpha$.
These symmetric polynomials interpolate between characters of type D and B, in the sense that $\transDB{\lambda}(x; 0)=\so^{(2n)}_{\lambda}(x)$ and $\transDB{\lambda}(x; 1) = \so^{(2n+1)}_{\lambda}(x)$.
To the best of our knowledge this
family of polynomials is new.

\vskip 1.5mm

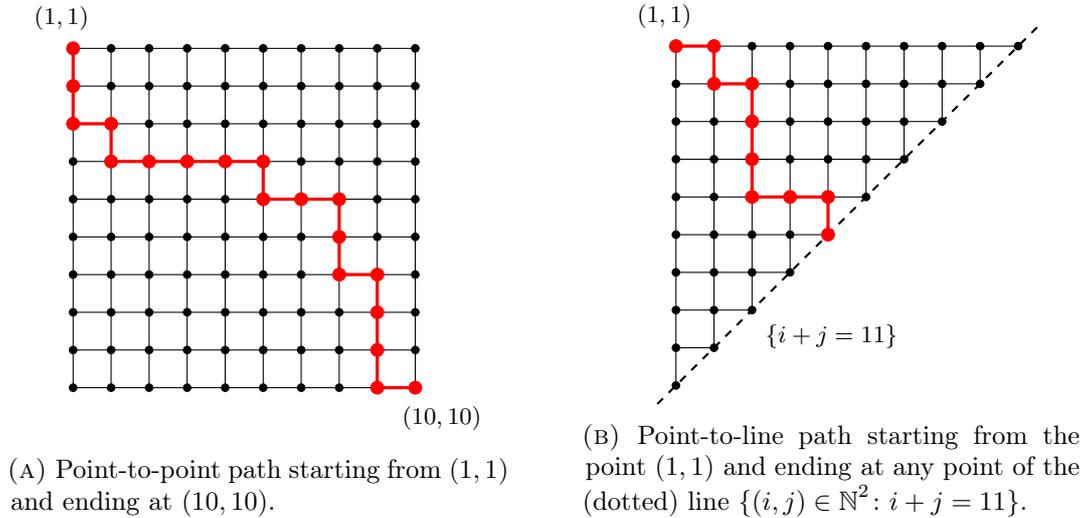
\begin{figure}
\centering
\begin{subfigure}[b]{.5\linewidth}
\centering
\captionsetup{width=.86\textwidth}%
\begin{tikzpicture}[scale=0.5]
\draw (1,-1) grid (10,-10);
\foreach \i in {1,...,10}{
\foreach \j in {1,...,10}{
		\node[draw,circle,inner sep=1pt,fill] at (\j,-\i) {};
	}
}

\draw[very thick,color=red,-] (1,-1) -- (1,-2) -- (1,-3) -- (2,-3) -- (2,-4) -- (3,-4) -- (4,-4) -- (5,-4) -- (6,-4) -- (6,-5) -- (7,-5) -- (8,-5) -- (8,-6) -- (8,-7) -- (9,-7) -- (9,-8) -- (9,-9) -- (9,-10) -- (10,-10);
\foreach \x in {(1,-1), (1,-2), (1,-3), (2,-3), (2,-4), (3,-4), (4,-4), (5,-4), (6,-4), (6,-5), (7,-5), (8,-5), (8,-6), (8,-7), (9,-7), (9,-8), (9,-9), (9,-10), (10,-10)}{
	\node[draw,circle,inner sep=1.7pt,fill,red] at \x {};
	}

\node at (0.7,-0.2) {\footnotesize $(1,1)$};
\node at (10.7,-10.8) {\footnotesize $(10,10)$};
\end{tikzpicture}
\subcaption{Point-to-point path starting from $(1,1)$ and ending at $(10,10)$.}
\label{subfig:pointToPointPath}
\end{subfigure}%
\begin{subfigure}[b]{.5\linewidth}
\centering
\captionsetup{width=.86\textwidth}%
\begin{tikzpicture}[scale=0.5]
\draw[thick,dashed] (10.5,-.5) -- (.5,-10.5);
\foreach \i in {1,...,10}{
	\draw (1,-\i) grid (11-\i,-\i);
	\draw (\i,-1) grid (\i,-11+\i);
	
		\foreach \j in {\i,...,10}{
		\node[draw,circle,inner sep=1pt,fill] at (11-\j,-\i) {};
	}
}

\draw[very thick,color=red,-] (1,-1) -- (2,-1) -- (2,-2) -- (3,-2) -- (3,-3) -- (3,-4) -- (3,-5) -- (4,-5) -- (5,-5) -- (5,-6);
\foreach \x in {(1,-1),(2,-1),(2,-2),(3,-2),(3,-3),(3,-4),(3,-5),(4,-5),(5,-5),(5,-6)}{
	\node[draw,circle,inner sep=1.7pt,fill,red] at \x {};
	}
\node at (0.7,-0.2) {\footnotesize $(1,1)$};
\node at (5.1,-8.7) {\footnotesize $\{i+j=11\}$};
\end{tikzpicture}
\subcaption{Point-to-line path starting from the point $(1,1)$ and ending at any point of the (dotted) line $\{(i,j)\in\N^2\colon i+j=11\}$.}
\label{subfig:pointToLinePath}
\end{subfigure}%
\caption{Directed paths in $\N^2$, highlighted in red, corresponding to the two geometries specified.
The picture is rotated by $90^{\circ}$ clockwise with respect to the Cartesian coordinate system, to adapt it to the usual matrix/array indexing.}
\label{fig:directedPaths}
\end{figure}

The first purpose of the interpolating Schur polynomials is to provide a new and unifying perspective of the intensively studied probabilistic model of (directed) last passage percolation.
To briefly introduce the model, let us denote by $\N$ the set of strictly positive integers.
Given a field $\{W_{i,j}\}$ of non-negative random variables on $\N^2$, usually called \emph{weights} or \emph{waiting times}, the \emph{Last Passage Percolation} (LPP) time is defined as
\begin{equation}
\label{eq:LPP}
L := \max_{\pi\in\Pi} \sum_{(i,j)\in \pi} W_{i,j} \, ,
\end{equation}
where $\Pi$ is a given set of directed paths.
Here, by \emph{directed path} we mean any finite sequence $\pi=((i_1,j_1),(i_2,j_2),\dots)$ of points of $\N^2$ such that $(i_k,j_k) - (i_{k-1},j_{k-1})$ is either $(1,0)$ or $(0,1)$ for $k>1$, as shown in Figure~\ref{fig:directedPaths}.
In particular, the \emph{point-to-point} LPP time, which we denote by $L(m,n)$, is taken on the set of all directed paths starting from $(1,1)$ and ending at a given $(m,n)\in\N^2$.

It has been known since the late 1990s that certain LPP models can be studied using standard Schur polynomials (of type A).
The point-to-point model with geometrically distributed weights was the first one to be solved exactly~\cite{Joh00}, just after the related Ulam's problem of the longest increasing subsequence of random permutations~\cite{BDJ99}.
Considering an array $\{W_{i,j}\colon 1\leq i,j\leq n\}$ of independent non-negative integer weights distributed as
\begin{align*}
\bbP(W_{i,j}=w_{i,j}) = (1-p_i q_j) (p_i q_j)^{w_{i,j}} \, ,
\end{align*}
for parameters $p_i,q_i\in (0,1)$,
and applying the Robinson-Schensted-Knuth ($\RSK$) correspondence~\cite{Knu70} and its properties, one obtains
\begin{align*}
\bbP\left( L(n,n) \leq u \right)
= \prod_{1\leq i,j\leq n}(1-p_iq_j) \sum_{\lambda_1 \leq u}
\schur_\lambda(p_1,\dots,p_n) \cdot \schur_\lambda(q_1,\dots,q_n)
\end{align*}
for $u\in\Z_{\geq 0}$, where the sum is over partitions $\lambda$ bounded above by $u$.
Using this exactly solvable structure, Johansson~\cite{Joh00} established the celebrated $n^{1/3}$ fluctuation scaling and derived the Tracy-Widom GUE limiting distribution\footnote{Such a distribution has been introduced in~\cite{TW94} to describe the fluctuations of the maximum eigenvalue of an asymptotically large random matrix from the Gaussian Unitary Ensemble (GUE).} from random matrix theory.
See e.g.~\cite{BDS16} for more details.

At the same time, Baik and Rains~\cite{BR01a} (see also~\cite{Rai00, Fer04, FR07}) considered point-to-point LPP problems on the square lattice $\{(i,j) \colon 1\leq i,j\leq N\}$ with various symmetries: about the antidiagonal $\{(i,j)\colon i+j=N+1\}$, about the diagonal $\{(i,j)\colon i=j\}$ or both the diagonal and the antidiagonal.
In other words, in these models some of the weights are independent and geometrically distributed, whereas others are determined by the symmetry constraints.
We will denote the respective last passage times, for geometrically distributed weights with certain choice of parameters $p_i$'s, $\alpha$, and $\beta$ that will be specified\footnote{More precisely, the $p_i$'s are the parameters of the geometric distributions, whereas the parameters $\beta$ and $\alpha$ modulate the intensity of the weights on the antidiagonal and diagonal, respectively (in the presence of antidiagonal and diagonal symmetry, respectively).} in Section~\ref{sec:results}, as $L^{\antisymSq}_\beta(N,N)$, $L^{\symSq}_{\alpha}(N,N)$, and $L^{\doubleSymSq}_{\alpha,\beta}(N,N)$.
For the sake of simplicity, let us assume for the moment that $N=2n$.
Via the use of the classical $\RSK$ on square matrices with symmetries, it was shown~\cite{BR01a} that such symmetric LPP models also admit exact expressions in terms of Schur polynomials of type A:
\begin{align}
\label{eq:antisymSchur_intro}
\P\left(L^{\antisymSq}_{\beta}(2n,2n) \leq 2u\right) 
&\propto \sum_{\mu_1 \leq 2u} \beta^{\oddrows\mu} \cdot \schur^{(2n)}_{\mu} (p_1,\dots,p_{2n}) \, , \\
\label{eq:symSchur_intro}
\P\left(L^{\symSq}_{\alpha}(2n,2n) \leq 2u\right)
&\propto \sum_{\mu_1 \leq 2u}
\alpha^{\oddrows\mu'} \cdot
\schur^{(2n)}_{\mu}(p_1,\dots,p_{2n}) \, , \\
\label{eq:doublesymSchur_intro}
\P\left(L^{\doubleSymSq}_{\alpha, 0}(2n,2n) \leq 2u\right)
&\propto \sum_{\mu_1 \leq u} \schur^{(n)}_{\mu}(p_1,\dots,p_n) \cdot \schur^{(n+1)}_{\mu}(p_1,\dots,p_n,\alpha) \, ,
\end{align}
where $\oddrows\mu$ counts the number of odd rows of $\mu$, while $\mu'$ stands for the conjugate partition of $\mu$.
The symbol $\propto$ above denotes equality up to a multiplicative constant that does \emph{not} depend on $u$.

In this work we will give a new perspective for the above symmetric models and derive formulas that involve characters of types \emph{other} than A, including the interpolating Schur polynomials introduced earlier.
Our analysis will be based on a modified point of view: this time, in the presence of antidiagonal symmetry we will work with the alternative formulation of the LPP problem in terms of point-to-line paths (see Figure~\ref{subfig:pointToLinePath}) and apply the $\RSK$ correspondence as a bijection between \emph{triangular arrays}, instead of square matrices (see Section~\ref{sec:RSK}).
In the resulting arrays, we will recognize precisely the patterns that generate $\CB$-interpolating Schur polynomials, thus arriving at the following identities (see Theorems~\ref{thm:antisymLPP} and~\ref{thm:doubleSymLPP}):
\begin{align}
\label{eq:antisymTrans_intro}
\P\left(L^{\antisymSq}_\beta(2n,2n) \leq 2u\right)
&\propto \left[ \prod_{i=1}^{2n} p_i \right]^u \sum_{\lambda_1 \leq u} \trans{\lambda} (p_1,\dots, p_n;\beta) \cdot \trans{\lambda} (p_{n+1}, \dots, p_{2n};\beta) \, , \\
\label{eq:doublesymTrans_intro}
\P\left(L^{\doubleSymSq}_{\alpha, \beta}(2n,2n) \leq 2u\right)
&\propto \left[ \prod_{i=1}^n p_i \right]^u\sum_{\lambda_1 \leq u} \alpha^{\oddrows(u^n-\lambda)'} \cdot \trans{\lambda}(p_1,\dots,p_n ;\beta) \, .
\end{align}
Here, $u^{n}$ is the partition with $n$ parts equal to $u$, whereas $u^n-\lambda$ denotes the complement of $\lambda$ with respect to $u^n$ (see beginning of Section~\ref{sec:patterns} for more precise definitions).
In Theorem~\ref{thm:antisymLPP} we will also obtain other exact expressions for $L^{\antisymSq}_\beta$ in terms of $\CB$-interpolating Schur polynomials.
Observe that, in fact, \eqref{eq:doublesymTrans_intro} is more general than~\eqref{eq:doublesymSchur_intro} as it does not require $\beta=0$.
Furthermore, even though via a more indirect method, we will derive an identity for the diagonally symmetric LPP model in terms of even orthogonal characters:
\begin{align}
\label{eq:symSo_intro}
\begin{split}
&\P\left(L^{\symSq}_{\alpha}(2n,2n) \leq 2u\right)
\propto  \left[ \alpha \prod_{i=1}^{2n} p_i \right]^u \sum_{\lambda_1 \leq u}
 \so^{(2n+2)}_{(u , \lambda_{\delta})}(p_1^{-1},\dots,p_n^{-1},\alpha^{-1}) \cdot \so^{(2n)}_{\lambda_{\delta}}(p_{n+1}^{-1},\dots,p_{2n}^{-1}) \, .
 \end{split}
\end{align}
In Theorem~\ref{thm:symLPP} we will also obtain other exact expressions for this model in terms of even orthogonal characters as well as $\DB$-interpolating Schur polynomials.

\vskip 1.5mm

Besides the fact that our analysis leads to, apparently, unnoticed links between exactly solvable probabilistic models and fundamental algebraic structures, our identities have a significance in terms of asymptotic analysis.
Indeed, they structurally explain the \emph{duality} between Pfaffian and determinant formulations of certain universal random matrix distributions that appear as a scaling limit of LPP models.

To see this, notice first that~\eqref{eq:antisymSchur_intro}, \eqref{eq:symSchur_intro} and~\eqref{eq:doublesymTrans_intro} are \emph{bounded Littlewood identities}, i.e.\ (weighted) sums, over bounded partitions, of Schur polynomials indexed by the given partition.
On the other hand, \eqref{eq:doublesymSchur_intro}, \eqref{eq:antisymTrans_intro} and~\eqref{eq:symSo_intro} are \emph{bounded Cauchy identities}, i.e.\ sums, over bounded partitions, of products of \emph{two} Schur polynomials indexed by (essentially) the same partition.
Therefore, whenever Baik and Rains' formulas are of Littlewood type, ours are of Cauchy type, and vice versa.
Now, as mentioned earlier, all the characters of the classical groups, and remarkably also $\CB$-interpolating polynomials, can be expressed as determinantal functions via formulas of Weyl character type.
Thanks to the well-known Andr\'eief's identity and de Bruijn identity (see Section~\ref{sec:duality}) that express integrals/sums of determinantal functions as either determinants or Pfaffians, one can easily see that Littlewood identities lead to Pfaffian measures, whereas Cauchy identities lead to determinantal measures.
Therefore, whenever a formula of Baik and Rains leads to a Pfaffian measure, ours leads to a determinantal measure, and vice versa.

The duality between Pfaffian and determinantal measures that emerges from comparing Baik and Rains' formulas with ours \emph{at a finite $n$ level} also induces an analogous duality \emph{at the $n\to\infty$ asymptotic level}.
It is well known (see~\cite{Joh00, BR01b,BBCS18}) that certain LPP times, when suitably centred and normalized at the fluctuation scale $n^{1/3}$, converge to limiting distributions from random matrix theory, such as the fundamental Tracy-Widom GUE, GOE and GSE laws\footnote{Analogously to the GUE case, the GOE and GSE Tracy-Widom distributions have been introduced to describe the fluctuations of the maximum eigenvalue of an asymptotically large random matrix from the Gaussian Orthogonal Ensemble (GOE) and Gaussian Symplectic Ensemble (GSE), respectively~\cite{TW96}.}.
Therefore, from this standpoint, the finite-$n$ duality just described translates into a dual structure for such universal limiting distributions.
In Section~\ref{sec:duality} we will analyze two notable cases: the Tracy-Widom GOE (as a scaling limit of the antidiagonally symmetric LPP $L^{\antisymSq}_{0}(2n,2n)$) and the Tracy-Widom GSE (as a scaling limit of the diagonally symmetric LPP $L^{\symSq}_{0}(2n,2n)$)\footnote{The reason why in this context we consider the LPP models with $\beta=0$ and $\alpha=0$ is not only convenience: if the value of $\alpha$ is too high, the asymptotic behavior becomes diffusive, i.e.\ with fluctuation scale $n^{1/2}$ and Gaussian limiting distribution.
See~\cite{BBCS18} for more details.}.
Originally, these distributions were expressed in terms of infinite-dimensional Pfaffians, wider known as \emph{Fredholm Pfaffians}, as well as in terms of Painlev\'e functions~\cite{TW96,TW05}.
It was later found out~\cite{Sas05, FS05} that they also possess a representation in terms of infinite-dimensional determinants, i.e.\ \emph{Fredholm determinants}.
The equivalence between the two formulations was shown in~\cite{FS05} by means of sophisticated linear operator tricks.
On the other hand, the duality between Pfaffian and determinantal measures, that we establish, sheds light on the structural foundations of this duality.
Namely, Baik and Rains' formulas~\eqref{eq:antisymSchur_intro} and~\eqref{eq:symSchur_intro} for $L^{\antisymSq}_{0}(2n,2n)$ and $L^{\symSq}_{0}(2n,2n)$ lead to the Fredholm Pfaffian representations of the GOE and GSE Tracy-Widom distributions, respectively, while our dual formulas~\eqref{eq:antisymTrans_intro} and~\eqref{eq:symSo_intro} lead to the corresponding Fredholm determinant representations.
Notice that, even though we do not undertake this task here (as it would require a longer asymptotic analysis), it should be also possible to obtain a non-trivial Fredholm Pfaffian representation of the GUE Tracy-Widom distribution from our formula~\eqref{eq:doublesymTrans_intro} for $L^{\doubleSymSq}_{0,0}(2n,2n)$, dual to the Fredholm determinant representation that can be derived from~\eqref{eq:doublesymSchur_intro}. 

\vskip 1.5mm

Another purpose of this article is to generalize and unify identities between the characters of the classical groups through interpolating Schur polynomials.
The first kind of identities (see Section~\ref{sec:decomposition}) describe how a Schur polynomial of \emph{rectangular} or \emph{bi-rectangular} shape\footnote{A rectangular (signed) (half-)partition is of the form $u^{n}_{\epsilon} = (u,\dots,\epsilon u)$, where $\epsilon=\pm $.
A bi-rectangular partition (commonly known as a \emph{fat hook}, see e.g.~\cite{Ste01}) is of the form $(u^{n}, v^{m})=(u,\dots,u,v,\dots,v)$.}
can be expressed as a bounded Cauchy sum for Schur polynomials of the same type.
Given two sets of variables $x=(x_1,\dots,x_n)$ and $y=(y_1,\dots,y_m)$ and denoting $x^{-1}:=(x_1^{-1},\dots,x_n^{-1})$, we will prove:
\begin{align}
\label{eq:decompositionA_intro}
\schur^{(n+m)}_{(u^{n}, v^{m})}(x,y)
&= \left[\prod_{i=1}^n x_i\right]^{u}
\left[\prod_{i=1}^m y_i\right]^{v}
\sum_{\mu_1 \leq u-v}
\schur^{(n)}_{\mu} (x^{-1}) \cdot
\schur^{(m)}_{\mu} (y)
\intertext{and, assuming that $n\geq m$,}
\label{eq:decompositionCB_intro}
\trans{u^{n+m}}(x,y;\beta)
&= \sum_{\lambda_1 \leq u}
\trans{(u^{n-m}, \lambda)}(x;\beta) \cdot
\trans{\lambda}(y;\beta) \, ,
 \\
\label{eq:decompositionOddSp_intro}
 \sp^{(2n+2m+2)}_{u^{n+m+1}}(x,y,s)
&= s^{-u}
\sum_{\lambda_1 \leq u}
\sp^{(2n+1)}_{(u^{n-m}, \lambda)}(x;s) \cdot
\sp^{(2m+1)}_{\lambda}(y;s) \, , \\
\label{eq:decompositionD_intro}
\so^{(2n+2m)}_{u^{n+m}_{\epsilon}}(x,y)
&= \sum_{\lambda_1 \leq u}
\so^{(2n)}_{(u^{n-m} , \lambda_{\delta})}(x) \cdot
\so^{(2m)}_{\lambda_{\delta\epsilon}}(y) \, .
\end{align}
In representation theory, these are known as \emph{branching rules}.
In particular, \eqref{eq:decompositionA_intro} indicates how irreducible polynomial representations of $\GL_{n+m}(\C)$ associated to bi-rectangular partitions decompose when restricted to $\GL_n(\C)\otimes \GL_m(\C)$.
Notice that~\eqref{eq:decompositionCB_intro} specializes, for $\beta=0$ and $\beta=1$, to the corresponding identities for even symplectic characters and odd orthogonal characters, respectively: thus, for $\beta=0$ and $\beta=1$, \eqref{eq:decompositionCB_intro} describes how irreducible polynomial representations of $\Sp_{2(n+m)}(\C)$ and $\SO_{2(n+m)+1}(\C)$, associated to rectangular (half-)partitions, decompose when restricted to $\Sp_{2n}(\C)\otimes \Sp_{2m}(\C)$ and $\SO_{2n+1}(\C)\otimes \SO_{2m+1}(\C)$, respectively.
Finally, \eqref{eq:decompositionOddSp_intro} and~\eqref{eq:decompositionD_intro} describe how certain irreducible polynomial representations of $\Sp_{2(n+m+1)}(\C)$ and $\SO_{2(n+m)}(\C)$ decompose when restricted to $\Sp_{2n+1}(\C)\otimes \Sp_{2m+1}(\C)$ and $\SO_{2n}(\C) \otimes \SO_{2m}(\C)$, respectively.
The identities that involve characters of classical groups were first proved in~\cite{Oka98} using intricate determinantal calculus based on the Weyl character formulas and the so-called \emph{minor summation formulas} (a generalization of the Cauchy-Binet identity and the de Bruijn identity)\footnote{More specifically, Okada proved: \eqref{eq:decompositionA_intro} for $v=0$ and $n\geq m$; \eqref{eq:decompositionCB_intro} for $\beta=0,1$; and \eqref{eq:decompositionD_intro}.}.
Our contribution in this respect is to introduce a simple method of decomposition of Gelfand-Tsetlin (and analogous) patterns of rectangular or bi-rectangular shape and use it to provide new combinatorial bijective proofs of Okada's identities, thus sheding light on their structure.
Our method of pattern decomposition is also suitable to prove formula~\eqref{eq:decompositionCB_intro}, which is a generalization to $\CB$-interpolating Schur polynomials, and formula~\eqref{eq:decompositionOddSp_intro} for odd symplectic characters, which was not dealt with in~\cite{Oka98}\footnote{A recent preprint by Okada~\cite{Oka19}, which appeared on arXiv after the present work, contains also a proof of~\eqref{eq:decompositionOddSp_intro} that relies on the same techniques as~\cite{Oka98}.
Therein this odd symplectic character identity is attributed to previous unpublished work of Brent-Krattenthaler-Warnaar.}.
We also expect this method to be applicable in wider settings, for more general functions that possess a similar combinatorial -- but not necessarily determinantal -- structure.
As pointed out by an anonymous referee, \eqref{eq:decompositionA_intro} also admits an alternative, purely algebraic proof, which we outline in Subection~\ref{subsec:decompositionA}; in particular, it is essentially a version of a well-known decomposition formula for skew Schur functions.
It is not clear to us, but it would be interesting to investigate, if this type of argument extends to the other decomposition identities~\eqref{eq:decompositionCB_intro}-\eqref{eq:decompositionOddSp_intro}-\eqref{eq:decompositionD_intro}.

Another set of identities that we will prove express an interpolating Schur polynomial of rectangular shape as a Littlewood sum of standard Schur polynomials:
\begin{align}
\label{eq:CB=LittlewoodSchur}
\trans{u^{N}} (x_1,\dots,x_N;\beta)
&= \left[ \prod_{i=1}^N x_i \right]^{-u} \sum_{\mu_1 \leq 2u} \beta^{\oddrows\mu} \cdot \schur^{(N)}_{\mu} (x_1,\dots,x_N) \, , \\
\label{eq:DB=LittlewoodSchur}
\transDB{u^{N}}(x_1^{-1},\dots,x_N^{-1}; \alpha)
&= \left[\prod_{i=1}^{N} x_i \right]^{-u} 
\sum_{\mu_1 \leq 2u}
\alpha^{\oddrows\mu'} \cdot
\schur^{(N)}_{\mu}(x_1,\dots,x_N) \, ,
\end{align}
and, somewhat conversely, we will express a standard Schur polynomial of rectangular shape as a Littlewood sum of symplectic characters:
\begin{align}
\label{eq:schur=LittlewoodSp}
\schur^{(2n+1)}_{u^{n}}(x_1^{-1},\dots,x_n^{-1},x_1,\dots,x_n,\alpha)
&= \sum_{\lambda_1 \leq u}
\alpha^{\oddrows(u^n-\lambda)'} \cdot
\sp^{(2n)}_{\lambda}(x_1,\dots,x_n) \, .
\end{align}
The proofs we provide rely on certain identities established by Krattenthaler~\cite{Kra98} for characters of ``nearly rectangular'' shape (see Section~\ref{sec:nearlyRectangular}).
This set of identities generalizes and unifies, by means of interpolating Schur polynomials, scattered identities in the literature for characters of various types: the specializations to $\beta=0,1$ and $\alpha=0,1$ can be found in~\cite{Ste90a,Mac95,Oka98,Kra98}.
Notice also that the $\alpha=\beta=0$, $u\to\infty$ versions of~\eqref{eq:CB=LittlewoodSchur} and~\eqref{eq:DB=LittlewoodSchur} are classical (unbounded) Littlewood identities~\cite{Lit50}.

Let us mention that analogous identities at the level of elliptic functions and BC symmetric polynomials and in the form of Selberg type integrals (generalizing random matrix related integrals) were established by Rains~\cite{Rai05, Rai10, Rai12}.
However, it is not obvious whether \eqref{eq:CB=LittlewoodSchur}, \eqref{eq:DB=LittlewoodSchur} and \eqref{eq:schur=LittlewoodSp} specifically fall within Rains' theory.

\vskip 1.5mm

{\bf Organization of the article}.
In Section~\ref{sec:patterns} we review the characters of classical groups, expressed both as generating functions of Gelfand-Tsetlin (and analogous) patterns and as ratios of determinants via the Weyl character formulas.
In Section~\ref{sec:transition} we define $\CB$- and $\DB$-interpolating Schur polynomials and establish various properties, including a determinantal formula of Weyl character type for the $\CB$-interpolating polynomials.
In Section~\ref{sec:results} we present in detail and discuss our results that relate three symmetric last passage percolation models to the characters of various types and to the interpolating Schur polynomials; all these results are proved in the next sections.
In Section~\ref{sec:RSK} we show how the $\RSK$ correspondence applied to triangular arrays directly leads to new exact formulas for the aforementioned LPP models in terms of (interpolating) Schur polynomials of type other than A.
In Section~\ref{sec:decomposition} we develop a method of decomposition of Gelfand-Tsetlin and related patterns, that we use to prove decomposition formulas~\eqref{eq:decompositionA_intro}-\eqref{eq:decompositionD_intro} for (interpolating) Schur polynomials of rectangular shape.
In Section~\ref{sec:nearlyRectangular} we prove identities~\eqref{eq:CB=LittlewoodSchur}-\eqref{eq:schur=LittlewoodSp} for (interpolating) Schur polynomials of (bi-)rectangular shape.
Finally, in Section~\ref{sec:duality} we explain how our formulas, in the scaling limit, explain the duality between Fredholm determinant and Fredholm Pfaffian structures in certain universal random matrix distributions.

\section{Gelfand-Tsetlin patterns and characters}
\label{sec:patterns}
Let us start by recalling some terminology.
We call {\bf\emph{half-integer}} any number that is half of an \emph{odd} integer, or equivalently any number of the form $n + 1/2$ with $n\in\Z$.
We call {\bf\emph{(unsigned) real $n$-partition}} an $n$-tuple $\lambda=(\lambda_1,\dots,\lambda_n)$ of real numbers such that $\lambda_1 \geq \lambda_2 \geq \cdots \geq \lambda_n \geq 0$.
We call {\bf\emph{signed real $n$-partition}} an $n$-tuple $\lambda_{\epsilon}=(\lambda_1,\dots,\lambda_{n-1}, \epsilon \lambda_n)$, where $(\lambda_1,\dots,\lambda_n)$ is a real $n$-partition and $\epsilon$ is a sign.
Clearly, every real $n$-partition $\lambda$ is in particular a signed real $n$-partition with positive sign, i.e.\ $\lambda = \lambda_+$; on the other hand, we have $\lambda_+ = \lambda_-$ if and only if $\lambda_n =0$.
The {\bf\emph{parts}} of a signed real $n$-partition $\lambda_\epsilon$ are the $\lambda_i$'s and its {\bf\emph{length}} is the number of non-zero parts.
When the parts are taken to be all integers or all half-integers, we obtain a signed {\bf\emph{$n$-partition}} or {\bf\emph{$n$-half-partition}}, respectively; notice that a signed $n$-half-partition is always of length $n$, whereas a signed $n$-partition may have smaller length.
One can also view any signed real $n$-partition as an infinite real sequence by setting $\lambda_i := 0$ for all $i>n$.
Therefore, it makes sense to refer to a signed real partition without reference to its maximum length $n$; however, we remark that a signed real partition with negative sign and length $m\leq n$ is a signed real $n$-partition only when $m=n$ (e.g., $(3,3,-1,0,0,\dots)$ is a signed $3$-partition but not a signed $4$-partition).
Denoting by $\abs{\cdot}$ the $1$-norm of sequences, we have $\abs{\lambda_+} = \abs{\lambda_-} = \sum_{i\geq 1} \lambda_i$.
We denote by $\emptyset:=(0,0,\dots)$ the partition of length zero.

An integer partition $\lambda$ is usually depicted as a {\bf\emph{Young diagram}}, i.e.\ a collection of left-aligned square boxes containing $\lambda_i$ squares in row $i$ (counting from the top).

The {\bf\emph{conjugate partition}} of $\lambda$, denoted by $\lambda'$, is defined by setting $\lambda'_i$ to be the number of $j\geq 1$ such that $\lambda_j\geq i$.
For instance, the Young diagram of the partition $\lambda=(4,3)$ is ${\tiny\yng(4,3)}$ and its conjugate partition is $\lambda'=(2,2,2,1)$.

Given two signed real partitions $\mu_\delta$ and $\lambda_\epsilon$, we write $ \mu_{\delta} \subseteq \lambda_\epsilon$ if $\lambda_i \geq \mu_i$ for $i\geq 1$; for integer partitions, this graphically means inclusion of the corresponding Young diagrams.
If the stronger condition $\lambda_i \geq \mu_i \geq \lambda_{i+1}$ ($i\geq 1$) holds, then we say that $\mu_\delta$ {\bf\emph{upwards interlaces with} $\lambda_\epsilon$} and write $\mu_\delta \prec \lambda_\epsilon$.
Notice that, in these definitions, the signs $\epsilon$ and $\delta$ do not play any role.
A {\bf\emph{rectangular signed real $n$-partition}} $u^{n}_{\epsilon}$ is the $n$-tuple $(u,\dots,\epsilon u)$, where $u\in\R$, $u\geq 0$, and $\epsilon$ is a sign.
If $\lambda\subseteq u^{n}$, then we denote by
\begin{equation}
\label{eq:complementPartition}
u^n-\lambda:=(u-\lambda_n,\dots,u-\lambda_1)
\end{equation}
the {\bf\emph{complement partition}} of $\lambda$ with respect to $u^{n}$.
If $\lambda$ is an integer partition, then
\begin{equation}
\label{eq:oddRows}
\oddrows\lambda:=\sum_{i\geq 1} (\lambda_i \bmod 2)
\end{equation}
is the number of odd rows of the corresponding Young diagram (i.e.\ the number of odd parts of $\lambda$).
Notice that, if $\lambda\subseteq u^{n}$, then we have
\begin{equation}
\label{eq:oddPartsConjugate}
\oddrows\lambda' = \sum_{i= 1}^n (-1)^{i-1}\lambda_i \, ,
\qquad
\oddrows(u^n-\lambda)' = \sum_{i=1}^n (-1)^{n-i} (u-\lambda_i) \, .
\end{equation}

A {\bf\emph{Young tableau}} is a Young diagram filled with symbols from an ordered set.
A {\bf\emph{semi-standard Young tableau}} $T$ is a Young tableau with entries that strictly increase down columns and weakly increase along rows; the {\bf\emph{shape}} of $T$, denoted by $\shape(T)$, is the partition associated with the underlying Young diagram.

In the following subsections we review Schur polynomials of types A, B, C and D, including Proctor's ``odd symplectic'' ones, which we define both as generating functions of the corresponding Gelfand-Tsetlin (or analogous) patterns and via their Weyl character formulas.
For more details and for the equivalence of the two definitions, we refer to~\cite{Pro94}.
For the representation theoretic significance of these polynomials as irreducible characters of the associated groups, the reader may consult~\cite{FH91}.

\subsection{Gelfand-Tsetlin patterns and general linear characters}
\label{subsec:schur}

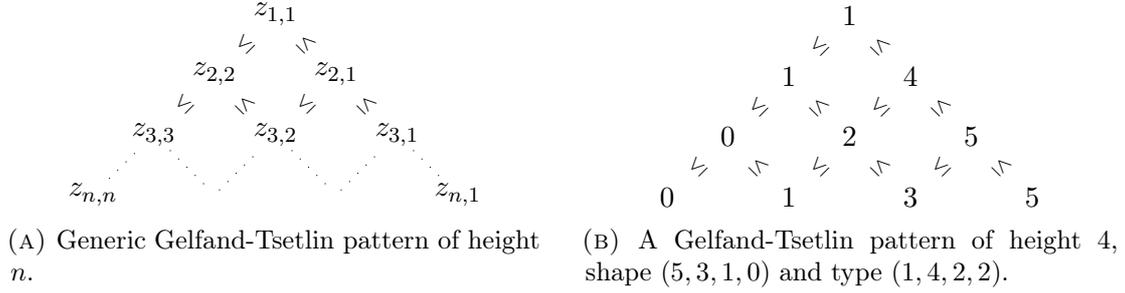
\begin{figure}
\centering
\begin{subfigure}[b]{.5\textwidth}
\centering
\captionsetup{width=.92\textwidth}%
\begin{tikzpicture}[scale=0.8, every path/.style={draw=none}, every node/.style={fill=white, scale=1, inner sep=1.5pt}, ineq1/.style={rotate=45, node contents={\tiny $\leq$}}, ineq2/.style={rotate=-45, node contents={\tiny $\leq$}}]

\node (z11) at (1,-1) {$z_{1,1}$};
\node (z21) at (2,-2) {$z_{2,1}$};
\node (z22) at (0,-2) {$z_{2,2}$};
\node (z31) at (3,-3) {$z_{3,1}$};
\node (z32) at (1,-3) {$z_{3,2}$};
\node (z33) at (-1,-3) {$z_{3,3}$};
\node (z41) at (4,-4) {$z_{n,1}$};
\node (z42) at (2,-4) {};
\node (z43) at (0,-4) {};
\node (z44) at (-2,-4) {$z_{n,n}$};

% arrows
\path (z22) -- node[ineq1]{} (z11) -- node[ineq2]{} (z21);
\path (z33) -- node[ineq1]{} (z22) -- node[ineq2]{} (z32) -- node[ineq1]{} (z21) -- node[ineq2]{} (z31);
\path[draw, loosely dotted] (z44) -- (z33) -- (z43) -- (z32) -- (z42) -- (z31) -- (z41);

\end{tikzpicture}
\subcaption{Generic Gelfand-Tsetlin pattern of height $n$.}
\label{subfig:GTpattern}
\end{subfigure}%
\begin{subfigure}[b]{.5\textwidth}
\centering
\captionsetup{width=.92\textwidth}%
\begin{tikzpicture}[scale=0.8, every path/.style={draw=none}, every node/.style={fill=white, scale=1, inner sep=1.5pt}, ineq1/.style={rotate=45, node contents={\tiny $\leq$}}, ineq2/.style={rotate=-45, node contents={\tiny $\leq$}}]

\node (z11) at (1,-1) {$1$};
\node (z21) at (2,-2) {$4$};
\node (z22) at (0,-2) {$1$};
\node (z31) at (3,-3) {$5$};
\node (z32) at (1,-3) {$2$};
\node (z33) at (-1,-3) {$0$};
\node (z41) at (4,-4) {$5$};
\node (z42) at (2,-4) {$3$};
\node (z43) at (0,-4) {$1$};
\node (z44) at (-2,-4) {$0$};

% arrows
\path (z22) -- node[ineq1]{} (z11) -- node[ineq2]{} (z21);
\path (z33) -- node[ineq1]{} (z22) -- node[ineq2]{} (z32) -- node[ineq1]{} (z21) -- node[ineq2]{} (z31);
\path (z44) -- node[ineq1]{} (z33) -- node[ineq2]{} (z43) -- node[ineq1]{} (z32) -- node[ineq2]{} (z42) -- node[ineq1]{} (z31) -- node[ineq2]{} (z41);

\end{tikzpicture}
\subcaption{A Gelfand-Tsetlin pattern of height $4$, shape $(5,3,1,0)$ and type $(1,4,2,2)$.}
\label{subfig:GTpatternEx}
\end{subfigure}
\caption{Gelfand-Tsetlin patterns.
The entries are non-negative integers and the inequalities illustrate the interlacing conditions.}
\label{fig:GTpatterns}
\end{figure}
A reparameterization of a certain kind of Young tableaux leads to the notion of Gelfand-Tsetlin patterns.
A {\bf\emph{Gelfand-Tsetlin pattern}} of height $n$ -- see Figure \ref{fig:GTpatterns} -- is a triangular array $z = (z_{i,j})_{ 1 \leq j\leq i \leq n}$ with non-negative integer entries that satisfy the {\bf \emph{interlacing conditions}}:
\begin{equation}
\label{eq:interlacing}
z_{i+1,j+1} \leq z_{i,j} \leq z_{i+1,j} \qquad\text{for all meaningful $i,j$} \, .
\end{equation}
We define the {\bf\emph{shape}} of $z$, denoted by $\shape(z)$, its bottom row $(z_{n,1}, \dots, z_{n,n})$.
We define the {\bf\emph{type}} of $z$ as the $n$-tuple $\type(z) \in \Z_{\geq 0}^{n}$ with entries
\begin{align*}
\type(z)_i := \sum_{j = 1}^i z_{i,j} - \sum_{j = 1}^{i-1} z_{i-1,j} \qquad \text{for } 1 \leq i \leq n \, ,
\end{align*}
where the convention (that we always adopt from now on) is that the empty sum equals zero.

A Gelfand-Tsetlin pattern $z$ of height $n$ and shape $\lambda$ can be equivalently viewed as an upwards interlacing sequence
\[ 
\Lambda
= \left(\ \emptyset = \lambda^{(0)} \prec \lambda^{(1)} \prec \dots \prec \lambda^{(n)} = \lambda \right) \, ,
\]
with $\lambda^{(i)}$ being an $i$-partition for $0\leq i\leq n$, by setting $\lambda^{(i)}_j := z_{i,j}$.

Moreover, we can map a semi-standard Young tableau $T$ in the alphabet $1<2<\dots <n$ to a Gelfand-Tsetlin pattern $z$ of height $n$ by setting $z_{i,j}$ to be the number of entries not greater than $i$ in row $j$ of $T$, for $1\leq j\leq i\leq n$.

Notice that, in the equivalence $z \leftrightarrow \Lambda \leftrightarrow T$, we have $\shape(z)= \lambda =\shape(T)$ and $\type(z)_i = |\lambda^{(i)}| - |\lambda^{(i-1)}| = \#\{ \text{$i$'s in } T\}$.

Schur polynomials of type A can be now defined as generating functions of Gelfand-Tsetlin patterns (or, equivalently, semi-standard Young tableaux).
Given an $n$-partition $\lambda$, let us denote by $\GT^{(n)}_{\lambda}$ the set of all Gelfand-Tsetlin patterns of height $n$ and shape $\lambda$.
\begin{definition}
\label{def:schur}
The {\bf\emph{Schur polynomial}} in $n$ variables $x=(x_1,\dots,x_n)$ indexed by an $n$-partition $\lambda$ is defined by
\begin{equation}
\label{eq:schur}
\schur^{(n)}_{\lambda}(x)
:= \sum_{z\in \GT^{(n)}_{\lambda}}
\prod_{i=1}^{n} x_i^{\type(z)_i}
\, .
\end{equation}
\end{definition}
Schur polynomials are characters of $\GL_n(\C)$ and as such they are invariant under the action of the associated Weyl group $S_n$: namely, they are invariant under permutation of the variables $x_1,\dots,x_n$.
Schur polynomials are \emph{determinantal}, in the sense that they can be expressed as ratios of determinants via the Weyl character formula:
\begin{equation}
\label{eq:schurWeyl}
\schur^{(n)}_{\lambda}(x)
= \frac{\underset{1\leq i,j\leq n}{\det} \big( x_j^{\lambda_i + n-i} \big)}
{\underset{1\leq i,j\leq n}{\det}\big( x_j^{n-i} \big)} \, ,
\end{equation}
where the denominator is the Vandermonde product $\prod_{1\leq i<j\leq n}(x_i-x_j)$.
An elementary proof of the equivalence of~\eqref{eq:schur} and~\eqref{eq:schurWeyl}, which does not resort to representation theoretic techniques, can be found in~\cite{Pro89a}.
The symmetry property is immediate from~\eqref{eq:schurWeyl}, while it is not obvious from their combinatorial definition~\eqref{eq:schur}.
Schur polynomials can also be expressed as single determinants of elementary or complete homogeneous symmetric polynomials via the so-called Jacobi-Trudi identities \cite{FK97}.

Finally, we mention a couple of properties of Schur polynomials that will turn out to be useful later on: denoting $\lambda+t:=(\lambda_1+t,\dots,\lambda_n + t)$, and recalling~\eqref{eq:complementPartition}, we have
\begin{align}
\label{eq:schurProp+}
\schur^{(n)}_{\lambda+t}(x_1,\dots,x_n)
&= \left[\prod_{i=1}^n x_i \right]^t \schur^{(n)}_{\lambda}(x_1,\dots,x_n) &&\text{for } t\geq 0 \, , \\
\label{eq:schurProp-}
\schur^{(n)}_{u^n-\lambda}(x_1,\dots,x_n)
&= \left[\prod_{i=1}^n x_i \right]^u \schur^{(n)}_{\lambda}(x_1^{-1},\dots,x_n^{-1}) &&\text{for } u\geq \lambda_1 \, .
\end{align}
These are obtained by using the changes of variables $z_{i,j} \mapsto z_{i,j} +t$ and $z_{i,j} \mapsto u-z_{i,i-j+1}$ in~\eqref{eq:schur}, respectively.

\subsection{Symplectic patterns and characters}
\label{subsec:sp}

\begin{figure}
\centering
\begin{subfigure}[b]{.5\textwidth}
\captionsetup{width=.92\textwidth}%
\centering
\begin{tikzpicture}[scale=0.8, every path/.style={draw=none}, every node/.style={fill=white, scale=1, inner sep=1.5pt}, ineq1/.style={rotate=45, node contents={\tiny $\leq$}}, ineq2/.style={rotate=-45, node contents={\tiny $\leq$}}]

\node (z11) at (1,-1) {$z_{1,1}$};
\node (z21) at (2,-2) {$z_{2,1}$};
\node (z22) at (0,-2) {};
\node (z31) at (3,-3) {$z_{3,1}$};
\node (z32) at (1,-3) {$z_{3,2}$};
\node (z41) at (4,-4) {$z_{4,1}$};
\node (z42) at (2,-4) {$z_{4,2}$};
\node (z43) at (0,-4) {};
\node (z51) at (5,-5) {};
\node (z52) at (3,-5) {};
\node (z53) at (1,-5) {};
\node (z61) at (6,-6) {$z_{N,1}$};
\node (z62) at (4,-6) {};
\node (z63) at (2,-6) {$z_{N,\lceil N/2 \rceil}$};
\node (z64) at (0,-6) {};

% arrows
\path (z11) -- node[ineq2]{} (z21);
\path (z32) -- node[ineq1]{} (z21) -- node[ineq2]{} (z31);
\path (z32) -- node[ineq2]{} (z42) -- node[ineq1]{} (z31) -- node[ineq2]{} (z41);
\path[draw, loosely dotted] (z53) -- (z42) -- (z52) -- (z41) -- (z51);
\path[draw, loosely dotted] (z53) -- (z63) -- (z52) -- (z62) -- (z51) -- (z61);

%\draw (0.3,-0.6) -- (0.3,-6.3);

\end{tikzpicture}
\subcaption{Generic symplectic pattern of height $N$. \\ \phantom{P} }
\label{subfig:spGTpattern}
\end{subfigure}%
\begin{subfigure}[b]{.5\textwidth}
\centering
\captionsetup{width=.92\textwidth}%
\begin{tikzpicture}[scale=0.8, every path/.style={draw=none}, every node/.style={fill=white, scale=1, inner sep=1.5pt}, ineq1/.style={rotate=45, node contents={\tiny $\leq$}}, ineq2/.style={rotate=-45, node contents={\tiny $\leq$}}]

\node (z11) at (1,-1) {$1$};
\node (z21) at (2,-2) {$2$};
\node (z22) at (0,-2) {};
\node (z31) at (3,-3) {$2$};
\node (z32) at (1,-3) {$0$};
\node (z41) at (4,-4) {$4$};
\node (z42) at (2,-4) {$1$};
\node (z43) at (0,-4) {};
\node (z51) at (5,-5) {$5$};
\node (z52) at (3,-5) {$3$};
\node (z53) at (1,-5) {$1$};
\node (z61) at (6,-6) {$5$};
\node (z62) at (4,-6) {$3$};
\node (z63) at (2,-6) {$2$};
\node (z64) at (0,-6) {};

% arrows
\path (z11) -- node[ineq2]{} (z21);
\path (z32) -- node[ineq1]{} (z21) -- node[ineq2]{} (z31);
\path (z32) -- node[ineq2]{} (z42) -- node[ineq1]{} (z31) -- node[ineq2]{} (z41);
\path (z53) -- node[ineq1]{} (z42) -- node[ineq2]{} (z52) -- node[ineq1]{} (z41) -- node[ineq2]{} (z51);
\path (z53) -- node[ineq2]{} (z63) -- node[ineq1]{} (z52) -- node[ineq2]{} (z62) -- node[ineq1]{} (z51) -- node[ineq2]{} (z61);

%\draw (0.4,-0.6) -- (0.4,-6.3);

\end{tikzpicture}
\subcaption{A symplectic pattern of height $6$, shape $(5,3,2)$, and type $(1,1,0,3,4,1)$.}
\label{subfig:spGTpatternEx}
\end{subfigure}%
\caption{Symplectic patterns.
The entries are non-negative integers and the inequalities illustrate the interlacing conditions.}
\label{fig:spGTpatterns}
\end{figure}
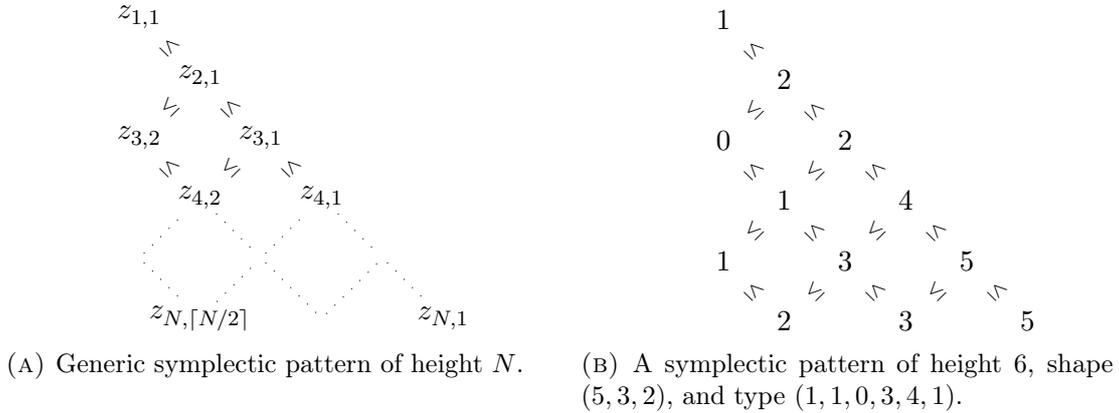

The even symplectic group $\Sp_{2n}(\C)$ is the group of all non-singular complex matrices of order $2n$ that preserve a non-degenerate skew-symmetric bilinear form.
It is a classical group with root system and Weyl group of type C.
Since there are no non-degenerate skew-symmetric bilinear forms on odd dimensional spaces, this definition makes sense only in the even case.
However, Proctor~\cite{Pro88} proposed an extended definition of symplectic group, requiring that matrices preserve a skew-symmetric bilinear form \emph{of maximal rank}.
This allows defining the \emph{odd symplectic group} $\Sp_{2n+1}(\C)$.
Odd symplectic groups are neither simple nor reductive, hence they are not counted among the classical groups.
They are, in various ways, related to root systems and Weyl groups of all three types B, C, and D -- see~\cite{Pro88} for details.

Besides the differences between even and odd symplectic groups, their characters are characterized by a very similar combinatorial definition as generating functions of certain patterns.
For this reason, we introduce them simultaneously in this subsection.

A {\bf\emph{symplectic pattern}} of height $N$ -- see Figure \ref{fig:spGTpatterns} -- is a ``half-triangular'' array $z=(z_{i,j})_{1\leq i\leq N, \, 1\leq j\leq \lceil i/2 \rceil}$ with non-negative integer entries that satisfy the interlacing conditions~\eqref{eq:interlacing}.
Its {\bf\emph{shape}} is the bottom row $\shape(z):=(z_{N,1},\dots,z_{N,\lceil N/2 \rceil})$ and its {\bf\emph{type}} is the $N$-tuple $\type(z) \in \Z_{\geq 0}^{N}$ defined by
\begin{equation}
\label{eq:spType}
\type(z)_i
:= \sum_{j = 1}^{\lceil i/2 \rceil} z_{i,j} - \sum_{j = 1}^{\lceil (i-1)/2 \rceil} z_{i-1,j}  \qquad\text{for } 1 \leq i \leq N \, .
\end{equation}

A symplectic pattern $z$ of height $N$ and shape $\lambda$ can be equivalently viewed as an upwards interlacing sequence
\[
\Lambda = \left(\ \emptyset = \lambda^{(0)} \prec \lambda^{(1)} \prec \dots \prec \lambda^{(N)} = \lambda \right) \, ,
\]
with $\lambda^{(i)}$ being an $\lceil i/2 \rceil$-partition for $0\leq i\leq N$, by setting $\lambda^{(i)}_j = z_{i,j}$.

Similarly to the standard Gelfand-Tsetlin case, symplectic patterns are in a bijective correspondence with the so-called ``symplectic tableaux''; the latter have been introduced by King~\cite{Kin76,KE83} in the even $N$ case, and by Proctor~\cite{Pro88} in the odd $N$ case.
Consider a semi-standard Young tableaux $T$ in the alphabet $1 < \overline{1} < 2 < \overline{2} < \dots < n < \overline{n}$ when $N=2n$ (respectively, in the alphabet $1 < \overline{1} < 2 < \overline{2} < \dots < n < \overline{n} < n+1$ when $N=2n+1$), and such that all entries in row $i$ are larger than or equal to $i$.
Setting $z_{2i-1,j}$ to be the number of entries not greater than $i$ in the $j$-th row of $T$ and $z_{2i,j}$ to be the number of entries not greater than $\overline{i}$ in the $j$-th row of $T$, we obtain a symplectic pattern $z$ of height $N$.
In the equivalence $z \leftrightarrow \Lambda \leftrightarrow T$, we have $\shape(z) = \lambda = \shape(T)$, $\mathrm{type}(z)_{2i-1} = |\lambda^{(2i-1)}| - |\lambda^{(2i-2)}| = \#\{ \text{$i$'s in } T\}$, and $\mathrm{type}(z)_{2i} = |\lambda^{(2i)}| - |\lambda^{(2i-1)}| = \#\{ \text{$\overline{i}$'s in } T\}$.

Symplectic Schur polynomials can be now defined as generating functions of symplectic patterns (or, equivalently, symplectic tableaux).
Given an $\lceil N/2 \rceil$-partition $\lambda$, let us denote by $\spGT^{(N)}_{\lambda}$ the set of all symplectic patterns of height $N$ and shape $\lambda$.
\begin{definition}
\label{def:spSchur}
The {\bf\emph{$(2n)$-symplectic Schur polynomial}} indexed by an $n$-partition $\lambda$ is the Laurent polynomial in variables $x=(x_1,\dots,x_n)$ defined by
\begin{equation}
\label{eq:spSchur}
\begin{split}
\sp^{(2n)}_{\lambda}(x)
:= \sum_{z\in \spGT^{(2n)}_{\lambda}}
\prod_{i=1}^{n} x_i^{\type(z)_{2i} - \type(z)_{2i-1}}
\, .
\end{split}
\end{equation}
\end{definition}
As characters of $\Sp_{2n}(\C)$, symplectic Schur polynomials are invariant under the action of the associated Weyl group $(\Z/2\Z)^n \rtimes S_n$ of type BC, i.e.\ they do not change if the variables $x_1,\dots,x_n$ are permuted or any of them is replaced by its multiplicative inverse.
One can deduce these properties from the Weyl character formula:
\begin{equation}
\label{eq:spWeyl}
\sp^{(2n)}_{\lambda}(x)
= \frac{\underset{1\leq i,j\leq n}{\det}\big( x_j^{\lambda_i + n-i+1} - x_j^{-(\lambda_i + n-i+1)} \big)}
{\underset{1\leq i,j\leq n}{\det}\big( x_j^{n-i+1} - x_j^{-(n-i+1)} \big)} \, .
\end{equation}
An ``elementary'' proof of the latter, not relying on any representation theory, can be found in~\cite{Pro93}.

\begin{definition}
\label{def:oddSpSchur}
The {\bf\emph{$(2n+1)$-symplectic Schur polynomial}} indexed by an $(n+1)$-partition $\lambda$ is the Laurent polynomial in variables $x=(x_1,\dots,x_n)$ and $y$ defined by
\begin{equation}
\label{eq:oddSpSchur}
\begin{split}
\sp^{(2n+1)}_{\lambda}(x;y)
:= \sum_{z\in \spGT^{(2n+1)}_{\lambda}}
\prod_{i=1}^{n} x_i^{\type(z)_{2i} - \type(z)_{2i-1}}
y^{\type(z)_{2n+1}}
\, .
\end{split}
\end{equation}
\end{definition}
Odd symplectic Schur polynomials are characters of $\Sp_{2n+1}(\C)$.
They are invariant under the action of the Weyl group of type BC \emph{on the $x$-variables only}.
Namely, they do not change if the variables $x_1,\dots,x_n$ are permuted or any of them is inverted; however, they have no invariance property with respect to the variable $y$.
A Weyl character formula for odd symplectic characaters was given in~\cite{Pro88} for the special case $y=1$.
A formula for general $y$ has recently appeared in~\cite{Oka19}:
\begin{equation}
\label{eq:oddSpWeyl}
\sp^{(2n+1)}_{\lambda}(x;y) = \frac{\underset{1\leq i,j\leq n+1}{\det}(A_{\lambda})}{\underset{1\leq i,j\leq n+1}{\det}(A_{\emptyset})} \, ,
\end{equation}
where, for any $(n+1)$-partition $\mu$, $A_{\mu}$ is the $(n+1)\times (n+1)$ matrix with $(i,j)$-entry
\[
\begin{cases}
\big(x_j^{\mu_i + n-i+2} - x_j^{-(\mu_i + n-i+2)} \big) - y^{-1} \big(x_j^{\mu_i + n-i+1} - x_j^{-(\mu_i + n-i+1)} \big) &\text{if } 1\leq j\leq n \, , \\
y^{\mu_i + n-i+1} &\text{if } j=n+1 \, .
\end{cases}
\]

Similarly to the standard Schur polynomials associated with $\GL_n(\C)$, the symplectic ones also have further determinantal expressions, such as the Jacobi-Trudi identities~\cite{FK97,Pro88}.

\subsection{Orthogonal patterns and characters}
\label{subsec:so}

\begin{figure}
\begin{subfigure}[b]{.5\textwidth}
\centering
\captionsetup{width=.92\textwidth}%
\begin{tikzpicture}[scale=0.8, every path/.style={draw=none}, every node/.style={fill=white, scale=1, inner sep=1.5pt}, ineq1/.style={rotate=45, node contents={\tiny $\leq$}}, ineq2/.style={rotate=-45, node contents={\tiny $\leq$}}]

\node (z11) at (1,-1) {$0$};
\node (z21) at (2,-2) {$2$};
\node (z22) at (0,-2) {};
\node (z31) at (3,-3) {$3$};
\node (z32) at (1,-3) {$-1$};
\node (z41) at (4,-4) {$5$};
\node (z42) at (2,-4) {$3$};
\node (z43) at (0,-4) {};
\node (z51) at (5,-5) {$5$};
\node (z52) at (3,-5) {$4$};
\node (z53) at (1,-5) {$-3$};
\node (z61) at (6,-6) {$6$};
\node (z62) at (4,-6) {$5$};
\node (z63) at (2,-6) {$4$};
\node (z64) at (0,-6) {};

% arrows
\path (z11) -- node[ineq2]{} (z21);
\path (z32) -- node[ineq1]{} (z21) -- node[ineq2]{} (z31);
\path (z32) -- node[ineq2]{} (z42) -- node[ineq1]{} (z31) -- node[ineq2]{} (z41);
\path (z53) -- node[ineq1]{} (z42) -- node[ineq2]{} (z52) -- node[ineq1]{} (z41) -- node[ineq2]{} (z51);
\path (z53) -- node[ineq2]{} (z63) -- node[ineq1]{} (z52) -- node[ineq2]{} (z62) -- node[ineq1]{} (z51) -- node[ineq2]{} (z61);

%\draw (0.3,-0.6) -- (0.3,-6.3);

\end{tikzpicture}
\subcaption{An orthogonal pattern of height $6$, shape $(6,5,4)$, and type $(0,2,2,4,4,3)$.}
\label{subfig:oOddGTpattern}
\end{subfigure}%
\begin{subfigure}[b]{.5\textwidth}
\centering
\captionsetup{width=.92\textwidth}%
\begin{tikzpicture}[scale=0.8, every path/.style={draw=none}, every node/.style={fill=white, scale=1, inner sep=1.5pt}, ineq1/.style={rotate=45, node contents={\tiny $\leq$}}, ineq2/.style={rotate=-45, node contents={\tiny $\leq$}}]

\node (z11) at (1,-1) {$-2.5$};
\node (z21) at (2,-2) {$2.5$};
\node (z22) at (0,-2) {};
\node (z31) at (3,-3) {$5.5$};
\node (z32) at (1,-3) {$0.5$};
\node (z41) at (4,-4) {$7.5$};
\node (z42) at (2,-4) {$1.5$};
\node (z43) at (0,-4) {};
\node (z51) at (5,-5) {$7.5$};
\node (z52) at (3,-5) {$7.5$};
\node (z53) at (1,-5) {$-0.5$};

% arrows
\path (z11) -- node[ineq2]{} (z21);
\path (z32) -- node[ineq1]{} (z21) -- node[ineq2]{} (z31);
\path (z32) -- node[ineq2]{} (z42) -- node[ineq1]{} (z31) -- node[ineq2]{} (z41);
\path (z53) -- node[ineq1]{} (z42) -- node[ineq2]{} (z52) -- node[ineq1]{} (z41) -- node[ineq2]{} (z51);

\end{tikzpicture}
\subcaption{An orthogonal pattern of height $5$, shape $(7.5,7.5,-0.5)$, and type $(2.5,0,3.5,3,6.5)$.}
\label{subfig:oEvenGTpattern}
\end{subfigure}%
\caption{Orthogonal patterns.
The generic orthogonal pattern has the same graphical representation as a symplectic pattern of the same height, see Figure~\ref{subfig:spGTpattern}.
However, the entries here are either all integers or all half-integers, and \emph{odd ends} are also allowed to be negative.
The interlacing conditions, illustrated by inequalities, hold in absolute value.}
\label{fig:oGTpatterns}
\end{figure}
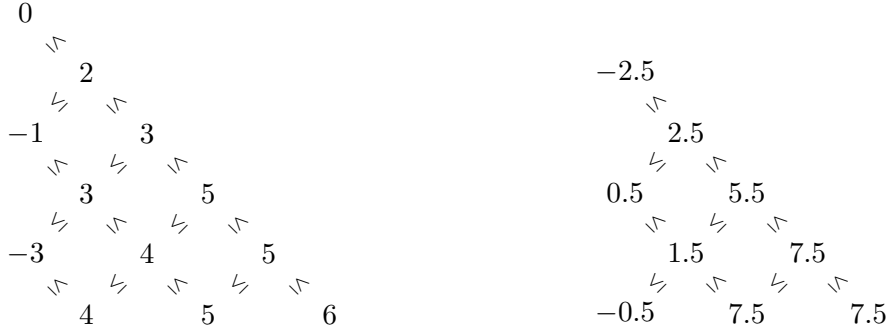

The (special) orthogonal group $\SO_N(\C)$ is the group of all complex orthogonal matrices of order $N$ with determinant $1$.
Usually, characters of the even and odd orthogonal groups (which are of type D and B respectively) are combinatorially defined using various types of tableaux -- see~\cite{KE83, KT87, KT90, Sun90a, Sun90b, Pro94}.
We will rather focus on two equivalent constructions of these polynomials as generating functions of patterns of two different kinds.

The first kind of pattern we deal with was first introduced by Gelfand and Tsetlin~\cite{GT50} and further studied by Proctor~\cite{Pro94}.
Let us define an {\bf\emph{orthogonal pattern}} of height $N$ -- see Figure~\ref{fig:oGTpatterns} -- to be a ``half-triangular'' array $z=(z_{i,j})_{1\leq i\leq N, \, 1\leq j\leq \lceil i/2 \rceil}$ that satisfies the following properties:
\begin{itemize}
\item
the entries are either all simultaneously integers or all simultaneously half-integers;
\item
the entries $z_{2i-1,i}$ for $1\leq i\leq \lceil N/2 \rceil$, which we call {\bf\emph{odd ends}}\footnote{This terminology is motivated by the fact that these entries are the last elements of odd rows of $z$.}, can be also \emph{negative}, whereas all other entries are non-negative;
\item
the interlacing conditions hold in absolute value\footnote{Here, $\abs{\cdot}$ denotes the absolute value of a number, whereas elsewhere in this work it denotes the $1$-norm of a tuple.}:
\begin{equation}
\label{eq:interlacingAbs}
\abs{z_{i+1,j+1}} \leq \abs{z_{i,j}} \leq \abs{z_{i+1,j}} \qquad\text{for all meaningful $i,j$} \, .
\end{equation}
\end{itemize}
The {\bf\emph{shape}} of $z$ is its bottom row $\shape(z):=(z_{N,1},\dots,z_{N,\lceil N/2 \rceil})$, which is an $n$-partition or $n$-half-partition, with or without sign according to whether $N=2n-1$ or $N=2n$.
We define the {\bf\emph{type}} of $z$ as in the previous subsections but considering the absolute values of the entries:
\begin{equation}
\label{eq:soType}
\type(z)_i
:= \sum_{j = 1}^{\lceil i/2 \rceil} \abs{z_{i,j}} - \sum_{j = 1}^{\lceil (i-1)/2 \rceil} \abs{z_{i-1,j}}  \qquad\text{for } 1 \leq i \leq N \, ,
\end{equation}
so that we have $\type(z) \in \left(\frac{1}{2} \Z_{\geq 0}\right)^{N}$.

For a given signed $n$-partition (respectively, signed $n$-half-partition) $\lambda_\epsilon$, an orthogonal pattern $z$ of height $2n-1$ and shape $\lambda_\epsilon$ can be equivalently viewed as an upwards interlacing sequence
\[
\Lambda = \left(\ \emptyset = \lambda^{(0)} \prec \lambda^{(1)}_{\epsilon_1} \prec \lambda^{(2)} \prec \dots \prec \lambda^{(2n-3)}_{\epsilon_{2n-3}} \prec \lambda^{(2n-2)} \prec \lambda^{(2n-1)}_{\epsilon_{2n-1}} = \lambda_{\epsilon} \right)
\]
such that:
\begin{itemize}
\item
$\lambda^{(2i-1)}_{\epsilon_{2i-1}}$ is a signed $i$-partition (respectively, a signed $i$-half-partition) for $1\leq i\leq n$;
\item
$\lambda^{(2i)}$ is an $i$-partition (respectively, an $i$-half-partition) for $0\leq i\leq n-1$.
\end{itemize}
An orthogonal pattern of height $2n$ can be viewed as an analogous upwards interlacing sequence that, this time, ends with an \emph{unsigned} $n$-(half-)partition.

Orthogonal patterns can be also shown to bijectively correspond to the so-called \emph{signed orthogonal tableaux} -- see~\cite[\S~8]{Pro93} for details.

As we will shortly see, both even and odd orthogonal Schur polynomials can be defined as generating functions of orthogonal patterns; however, the weight monomials differ from the ones used in the symplectic case.

Let us start with the even case.
Given a signed $n$-partition or signed $n$-half-partition $\lambda_\epsilon$, let us denote by $\oGT^{(2n-1)}_{\lambda_\epsilon}$ the set of all orthogonal patterns of height $2n-1$ and shape $\lambda_\epsilon$.
From now on, we set $\sign(a) := +1$ for $a\geq 0$ and $\sign(a) := -1$ for $a<0$.
\begin{definition}
\label{def:soEvenSchur}
The {\bf\emph{$(2n)$-orthogonal Schur polynomial}} indexed by a signed $n$-partition or signed $n$-half-partition $\lambda_\epsilon$ is the Laurent polynomial in variables $x=(x_1,\dots,x_n)$ defined by
\begin{equation}
\label{eq:soEvenSchur}
\begin{split}
\so^{(2n)}_{\lambda_\epsilon}(x)
:= \sum_{z\in \oGT^{(2n-1)}_{\lambda_\epsilon}}
x_1^{z_{1,1}}
\prod_{i=2}^{n} x_i^{\sign(z_{2i-3,i-1}) \sign(z_{2i-1,i}) [\type(z)_{2i-1} - \type(z)_{2i-2}]} \, .
\end{split}
\end{equation}
\end{definition}
Notice that the exponent of $x_1$ in~\eqref{eq:soEvenSchur} can be also expressed, as the exponents of the other variables, in terms of the type of $z$ and the sign of the odd ends: $z_{1,1} = \sign(z_{1,1}) \type(z)_{1}$.

As characters of $\SO_{2n}(\C)$, $(2n)$-orthogonal Schur polynomials are invariant under the action of the associated Weyl group $(\Z/2\Z)^{n-1} \rtimes S_n$ of type D.
Namely, they are invariant under permutation of the variables $x_1,\dots,x_n$ and multiplicative inversion of an \emph{even} number of them.
These properties can be deduced from the Weyl character formula:
\begin{equation}
\label{eq:soEvenWeyl}
\so^{(2n)}_{\lambda_{\epsilon}}(x) = \frac{\underset{1\leq i,j\leq n}{\det}\big( x_j^{\lambda_i + n-i} + x_j^{-(\lambda_i + n-i)} \big) + \epsilon \underset{1\leq i,j\leq n}{\det}\big( x_j^{\lambda_i + n-i} - x_j^{-(\lambda_i + n-i)} \big)}
{\underset{1\leq i,j\leq n}{\det}\big( x_j^{n-i} + x_j^{-(n-i)} \big)} \, .
\end{equation}

It is not difficult to prove, for example starting from Definition~\ref{def:soEvenSchur}, the following property that will be useful later on to deduce Corollary~\ref{coro:symLPP_0}:
\begin{equation}
\label{eq:soEvenVanishingVar}
\lim_{\alpha \downarrow 0} \alpha^k \so^{(2n+2)}_{(k,\lambda_\epsilon)} (x_1,\dots,x_n,\alpha^{-1})
= \so^{(2n)}_{\lambda_\epsilon} (x_1,\dots,x_n)
\end{equation}
for any integer (respectively, half-integer) $k$ such that $k\geq \lambda_1$, assuming that $\lambda_\epsilon$ is a signed $n$-partition (respectively, signed $n$-half-partition).

\vskip 1mm

Let us now pass to odd orthogonal characters.
Given an $n$-partition or $n$-half-partition $\lambda$, let us denote by $\oGT^{(2n)}_{\lambda}$ the set of all orthogonal patterns of height $2n$ and shape $\lambda$.
\begin{definition}
\label{def:soOddSchur}
The {\bf\emph{$(2n+1)$-orthogonal Schur polynomial}} indexed by an $n$-partition or $n$-half-partition $\lambda$ is the Laurent polynomial in variables $x=(x_1,\dots,x_n)$ defined by
\begin{equation}
\label{eq:soOddSchur}
\begin{split}
\so^{(2n+1)}_{\lambda}(x)
:= \sum_{z\in \oGT^{(2n)}_{\lambda}}
x_1^{z_{1,1}}
\prod_{i=2}^{n} x_i^{\sign(z_{2i-3,i-1}) \sign(z_{2i-1,i}) [\type(z)_{2i-1} - \type(z)_{2i-2}]} \, .
\end{split}
\end{equation}
\end{definition}
Notice that the weight monomials in~\eqref{eq:soEvenSchur} and~\eqref{eq:soOddSchur} coincide, but the sets of patterns over which the sums are taken differ.
Observe also that the shape $\lambda$ does not appear in the weights of~\eqref{eq:soOddSchur}: it just plays the role of ``bounding from above'' the entries of the previous rows.

The Weyl group of $\SO_{2n+1}(\C)$ is of type BC, as was the case for $\Sp_{2n}(\C)$; this means that, as characters of $\SO_{2n+1}(\C)$, $(2n+1)$-orthogonal Schur polynomials are invariant under permutation of the variables $x_1,\dots,x_n$ and multiplicative inversion of any $x_i$.
This property immediately follows from the Weyl character formula:
\begin{equation}
\label{eq:soOddWeyl}
\so^{(2n+1)}_{\lambda}(x) = \frac{\underset{1\leq i,j\leq n}{\det}\left( x_j^{\lambda_i + n-i+1/2} - x_j^{-(\lambda_i + n-i+1/2)} \right)}
{\underset{1\leq i,j\leq n}{\det}\left( x_j^{n-i+1/2} - x_j^{-(n-i+1/2)} \right)} \, .
\end{equation}

A property that relates odd and even orthogonal characters indexed by a ``rectangular (half-)partition'' $u^{k}:= \underbrace{(u,\dots,u)}_{k \text{ times}}$  is the following:
\begin{equation}
\label{eq:soEvenRect=soOddRect}
\so^{(2n+2)}_{u^{n+1}_{\epsilon}}(x_1,\dots,x_n,1)
= \so^{(2n+1)}_{u^{n}}(x_1,\dots,x_n) \, ,
\end{equation}
with $u\in\frac{1}{2} \Z_{\geq 0}$ and $\epsilon=\pm$ being any sign.
The latter can be easily verified using Definitions~\ref{def:soEvenSchur} and~\ref{def:soOddSchur} and noticing that, due to the interlacing conditions, the first $2n$ rows of an orthogonal pattern of height $2n+1$ and shape $u^{n+1}_{\epsilon}$ form an orthogonal pattern of height $2n$ and shape $u^{n}$.

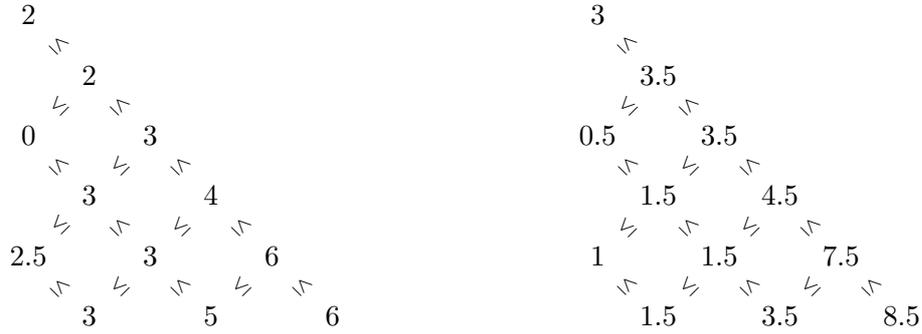
\begin{figure}
\centering
\begin{subfigure}[b]{.5\textwidth}
\centering
\captionsetup{width=.92\textwidth}%
\begin{tikzpicture}[scale=0.8, every path/.style={draw=none}, every node/.style={fill=white, scale=1, inner sep=1.5pt}, ineq1/.style={rotate=45, node contents={\tiny $\leq$}}, ineq2/.style={rotate=-45, node contents={\tiny $\leq$}}]

\node (z11) at (1,-1) {$2$};
\node (z21) at (2,-2) {$2$};
\node (z22) at (0,-2) {};
\node (z31) at (3,-3) {$3$};
\node (z32) at (1,-3) {$0$};
\node (z41) at (4,-4) {$4$};
\node (z42) at (2,-4) {$3$};
\node (z43) at (0,-4) {};
\node (z51) at (5,-5) {$6$};
\node (z52) at (3,-5) {$3$};
\node (z53) at (1,-5) {$2.5$};
\node (z61) at (6,-6) {$6$};
\node (z62) at (4,-6) {$5$};
\node (z63) at (2,-6) {$3$};
\node (z64) at (0,-6) {};

% arrows
\path (z11) -- node[ineq2]{} (z21);
\path (z32) -- node[ineq1]{} (z21) -- node[ineq2]{} (z31);
\path (z32) -- node[ineq2]{} (z42) -- node[ineq1]{} (z31) -- node[ineq2]{} (z41);
\path (z53) -- node[ineq1]{} (z42) -- node[ineq2]{} (z52) -- node[ineq1]{} (z41) -- node[ineq2]{} (z51);
\path (z53) -- node[ineq2]{} (z63) -- node[ineq1]{} (z52) -- node[ineq2]{} (z62) -- node[ineq1]{} (z51) -- node[ineq2]{} (z61);

\end{tikzpicture}
\subcaption{A split orthogonal pattern of height $6$, shape $(6,5,3)$, and type $(2,0,1,4,4.5,2.5)$, with one atypical entry $2.5$.}
\label{subfig:soGTpatternEx1}
\end{subfigure}%
\begin{subfigure}[b]{.5\textwidth}
\centering
\captionsetup{width=.92\textwidth}%
\begin{tikzpicture}[scale=0.8, every path/.style={draw=none}, every node/.style={fill=white, scale=1, inner sep=1.5pt}, ineq1/.style={rotate=45, node contents={\tiny $\leq$}}, ineq2/.style={rotate=-45, node contents={\tiny $\leq$}}]

\node (z11) at (1,-1) {$3$};
\node (z21) at (2,-2) {$3.5$};
\node (z22) at (0,-2) {};
\node (z31) at (3,-3) {$3.5$};
\node (z32) at (1,-3) {$0.5$};
\node (z41) at (4,-4) {$4.5$};
\node (z42) at (2,-4) {$1.5$};
\node (z43) at (0,-4) {};
\node (z51) at (5,-5) {$7.5$};
\node (z52) at (3,-5) {$1.5$};
\node (z53) at (1,-5) {$1$};
\node (z61) at (6,-6) {$8.5$};
\node (z62) at (4,-6) {$3.5$};
\node (z63) at (2,-6) {$1.5$};
\node (z64) at (0,-6) {};

% arrows
\path (z11) -- node[ineq2]{} (z21);
\path (z32) -- node[ineq1]{} (z21) -- node[ineq2]{} (z31);
\path (z32) -- node[ineq2]{} (z42) -- node[ineq1]{} (z31) -- node[ineq2]{} (z41);
\path (z53) -- node[ineq1]{} (z42) -- node[ineq2]{} (z52) -- node[ineq1]{} (z41) -- node[ineq2]{} (z51);
\path (z53) -- node[ineq2]{} (z63) -- node[ineq1]{} (z52) -- node[ineq2]{} (z62) -- node[ineq1]{} (z51) -- node[ineq2]{} (z61);

\end{tikzpicture}
\subcaption{A split orthogonal pattern of height $6$, shape $(8.5,3.5,1.5)$, and type $(3,0.5,0.5,2,4,3.5)$, with atypical entries $3$, $1$.}
\label{subfig:soGTpatternEx2}
\end{subfigure}%
\caption{Split orthogonal patterns.
The generic split orthogonal pattern has the same graphical representation as a symplectic pattern of the same height, see Figure~\ref{subfig:spGTpattern}.
All entries are non-negative, but may be either integers or half-integers here, with the constraint that all entries \emph{except odd ends} are of the same type.
As usual, the inequalities illustrate the interlacing conditions.}
\label{fig:soGTpatterns}
\end{figure}

Crucially for the development of this work, odd orthogonal characters can be also defined as generating functions of another kind of patterns, introduced by Proctor~\cite{Pro94}, where no negative entries are allowed but integer and half-integer entries might be simultaneously present.
Let us introduce them.
We call {\bf\emph{split orthogonal pattern}} of height $2n$ -- see Figure \ref{fig:soGTpatterns} -- any ``half-triangular'' array $z=(z_{i,j})_{1\leq i\leq 2n, \, 1\leq j\leq \lceil i/2 \rceil}$ of height $2n$ that satisfies the following properties:
\begin{itemize}
\item
the entries $z_{2i-1,i}$ for $1\leq i\leq n$, which again we call {\bf\emph{odd ends}}, are in $\frac{1}{2}\Z_{\geq 0}$;
\item
the other entries are either all simultaneously in $\Z_{\geq 0}$ or all simultaneously in $\frac{1}{2} + \Z_{\geq 0}$;
\item
the usual interlacing conditions~\eqref{eq:interlacing} hold.
\end{itemize}
Assuming that $z_{2n,1}$ is an integer (respectively, half-integer), we call {\bf\emph{atypical}} all the half-integer (respectively, integer) entries of the array.
According to the conditions above, atypical entries are necessarily odd ends, as shown in Figure~\ref{fig:soGTpatterns}.
Notice that, by definition, any symplectic pattern is a split orthogonal pattern of the same height where all entries are integers (in particular, with no atypical entries).
The definitions of shape and type are the usual ones; just notice that the shape of a split orthogonal pattern of height $2n$ is either an $n$-partition or an $n$-half-partition, as even rows do not contain atypical entries.

For a given $n$-partition (respectively, $n$-half-partition) $\lambda$, a split orthogonal pattern $z$ of height $2n$ and shape $\lambda$ can be equivalently viewed as an upwards interlacing sequence
\[
\Lambda = \left(\ \emptyset = \lambda^{(0)} \prec \lambda^{(1)} \prec \dots \prec \lambda^{(2n)} = \lambda \right)
\]
such that:
\begin{itemize}
\item
$\lambda^{(2i-1)}$ is a real $i$-partition with the first $i-1$ parts in $\Z_{\geq 0}$ (respectively, in $\frac{1}{2} + \Z_{\geq 0}$) and the $i$-th part in $\frac{1}{2} \Z_{\geq 0}$, for $1\leq i\leq n$;
\item
$\lambda^{(2i)}$ is an $i$-partition (respectively, an $i$-half-partition) for $0\leq i\leq n$.
\end{itemize}

Split orthogonal patterns bijectively correspond to certain tableaux introduced by Koike and Terada~\cite{KT87, KT90}.
Consider a semi-standard Young tableaux $T$ in the alphabet $1 < \mathring{1} < \overline{1} < 2 < \mathring{2} < \overline{2} < \dots < n < \mathring{n} < \overline{n}$ such that (i) all entries in row $i$ are larger than or equal to $i$ and (ii) symbol $\mathring{i}$ appears at most once in row $i$ and never in any other row.
The bijection then works similarly as the one between symplectic tableaux and symplectic patterns, with the convention that each extra symbol $\mathring{i}$ should be counted both as a half $i$ and as a half $\overline{i}$.
More precisely, set $z_{2i-1,j}$ to be the number of entries not greater than $i$ in the $j$-th row of $T$, increased by $1/2$ if the $j$-th row also contains $\mathring{i}$; also, set $z_{2i,j}$ to be the number of entries not greater than $\overline{i}$ in the $j$-th row of $T$.
Then, $z$ is a split orthogonal pattern of height $2n$.
In the correspondence $z \leftrightarrow \Lambda \leftrightarrow T$, we have $\shape(z) = \lambda = \shape(T)$, $\mathrm{type}(z)_{2i-1} = |\lambda^{(2i-1)}| - |\lambda^{(2i-2)}| = \#\{ \text{$i$'s in } T\} + \frac{1}{2} \#\{ \text{$\mathring{i}$'s in } T\}$, and $\mathrm{type}(z)_{2i} = |\lambda^{(2i)}| - |\lambda^{(2i-1)}| = \#\{ \text{$\overline{i}$'s in } T\} + \frac{1}{2} \#\{ \text{$\mathring{i}$'s in } T\}$.

It turns out that odd orthogonal Schur polynomials can be also defined as generating functions of split orthogonal patterns, using the same weight monomials as in the definition of even symplectic Schur polynomials -- see~\eqref{eq:spSchur}.
Given an $n$-partition or $n$-half-partition $\lambda$, let us denote by $\soGT^{(2n)}_{\lambda}$ the set of all split orthogonal patterns of height $2n$ and shape $\lambda$.
\begin{definition}
\label{def:soOddSchurSplit}
The {\bf \emph{$(2n+1)$-orthogonal Schur polynomial}} indexed by an $n$-partition or $n$-half-partition $\lambda$ can be defined as
\begin{equation}
\label{eq:soOddSchurSplit}
\begin{split}
\so^{(2n+1)}_{\lambda}(x)
:= \sum_{z\in \soGT^{(2n)}_{\lambda}}
\prod_{i=1}^{n} x_i^{\type(z)_{2i} - \type(z)_{2i-1}} \, .
\end{split}
\end{equation}
\end{definition}
For an analogous definition in terms of Koike-Terada tableaux, see~\cite{KT87, KT90}.
For the equivalence of Definitions~\ref{def:soOddSchur} and~\ref{def:soOddSchurSplit}, we refer to~\cite{Pro94}.

\section{Transition between characters}
\label{sec:transition}

In this section we introduce two classes of symmetric polynomials that, via a tuning parameter, establish a transition between classical groups' characters of different types.
We define them via generating functions of patterns and provide, for the first class, determinantal formulas, as we have done for the classical characters in Section~\ref{sec:patterns}.

In Subsection~\ref{subsec:bijectionPatterns} we describe a combinatorial bijection between split orthogonal and symplectic patterns.
In Subsection~\ref{subsec:CBtransition} we define a class of polynomials that interpolate between characters of types C and B; we also express them in terms of symplectic characters using the aforementioned combinatorial bijection and provide determinantal formulas of ``Weyl character type'', which permit us to link them to Koornwinder polynomials.
In Subsection~\ref{subsec:DBtransition} we introduce a second class of polynomials that interpolate between characters of types D and B.

\subsection{A bijection between symplectic and split orthogonal  patterns}
\label{subsec:bijectionPatterns}

Here we introduce a combinatorial bijection between a class of split orthogonal patterns with a fixed shape $\lambda$ and symplectic patterns with a ``perturbed'' shape.
The proof of this result is fairly natural and straightforward in the case of $\lambda$ being a half-partition.
On the other hand, the case of $\lambda$ being an \emph{integer} partition is more interesting, as it requires a non-trivial algorithmic procedure.
In the latter case, our bijection can be also equivalently viewed as a correspondence between Koike-Terada orthogonal tableaux~\cite{KT87, KT90} and Sundaram's orthogonal tableaux~\cite{Sun90a}; for more details, see the remarks on the \emph{Relation to Sundaram's tableaux} at the end of Section~\ref{subsec:CBtransition}.

For a split orthogonal pattern $z$ of height $2n$, we denote by $\atyp(z)\in \{0,1\}^n$ the $n$-tuple whose $i$-th entry equals $1$ if and only if $z_{2i-1,i}$ is \emph{atypical}, according to the definition given in Subsection~\ref{subsec:so}.
E.g., for the pattern in Figure~\ref{subfig:soGTpatternEx1}, $\atyp(z) = (0,0,1)$.
We also denote by $\abs{v}:= \sum_{i=1}^n \abs{v_i}$ the $1$-norm of any $v\in\R^n$.

\begin{theorem}
\label{thm:sp-soPatternsBijection}
For any $n$-partition $\lambda$ and integer $1\leq k\leq n$, there exists a bijection
\begin{equation}
\label{eq:sp-soPatternsBijection_1}
\{ z \in \soGT^{(2n)}_{\lambda} \colon \abs{\atyp(z)}=k \}
\longleftrightarrow
\{ z' \in \spGT^{(2n)}_{\lambda-\epsilon} \colon \epsilon \in\{0,1\}^n, \, \abs{\epsilon} = k \}
\end{equation}
that satisfies, for $1\leq i\leq n$,
\begin{equation}
\label{eq:typePreserved_1}
\type(z)_{2i} - \type(z)_{2i-1}
= \type(z')_{2i} - \type(z')_{2i-1} \, .
\end{equation}
Furthermore, for any $n$-half-partition $\lambda$ and $\epsilon\in\{0,1\}^n$, there exists a bijection
\begin{equation}
\label{eq:sp-soPatternsBijection_2}
\{ z \in \soGT^{(2n)}_{\lambda} \colon \atyp(z)=\epsilon \}
\longleftrightarrow
\{ z' \in \spGT^{(2n)}_{\lambda-1/2} \}
\end{equation}
that satisfies, for $1\leq i\leq n$,
\begin{equation}
\label{eq:typePreserved_2}
\type(z)_{2i} - \type(z)_{2i-1}
= \type(z')_{2i} - \type(z')_{2i-1} + \epsilon_i - \frac{1}{2} \, .
\end{equation}
\end{theorem}

\begin{proof}
We first prove~\eqref{eq:sp-soPatternsBijection_2}-\eqref{eq:typePreserved_2}, which is straightforward.
Let $\lambda$ be an $n$-half-partition and $\epsilon:=(\epsilon_1,\dots, \epsilon_n) \in\{0,1\}^n$.
Given a split orthogonal pattern $z$ of height $2n$, shape $\lambda$ and such that $\atyp(z)=\epsilon$, we define $z'$ by setting $z'_{i,j} = z_{i,j}$ for all atypical (i.e.\ integer) entries and $z'_{i,j} := z_{i,j} - 1/2$ for all the other entries of $z$.
The entries of the new pattern $z'$ are all integers and still satisfy the interlacing conditions, hence $z'$ is a symplectic pattern of height $2n$ and shape $\lambda - 1/2$.
It is immediate to verify that, if $\epsilon$ remains fixed, $z\mapsto z'$ is a bijection.
Since $\type(z')_{2i} = \type(z)_{2i} - \epsilon_i /2$ and $\type(z')_{2i-1} = \type(z)_{2i-1} - (1-\epsilon_i)/2$, we deduce~\eqref{eq:typePreserved_2}.

\vspace{1mm}
Let us now prove~\eqref{eq:sp-soPatternsBijection_1}-\eqref{eq:typePreserved_1}.
For the purpose of this proof, let $\mathcal{I} := \{(i,j)\colon 1\leq i\leq 2n, \, 1\leq j\leq \lceil i/2 \rceil\}$ be the index set of any ``half-triangular'' pattern of height $2n$.
Moreover, define a \emph{nearest neighbor path} to be a sequence in $\mathcal{I}$ such that the element that comes after $(i,j)$ is either $(i+1,j)$ or $(i+1,j+1)$; with respect to the graphical representation of a pattern, this is a downwards path.
Finally, by \emph{reverse nearest neighbor path} we mean a sequence in $\mathcal{I}$ such that the element that comes after $(i,j)$ is either $(i-1,j)$ or $(i-1,j-1)$; graphically, this is an upwards path in a pattern.

\begin{figure}
\centering
\begin{subfigure}[b]{.5\textwidth}
\centering
\begin{tikzpicture}[scale=0.45, inner sep=0.8]

\node (z11) at (1,-1) {$0.5$};
\node (z21) at (2,-2) {$1$};
\node (z22) at (0,-2) {};
\node (z31) at (3,-3) {$3$};
\node (z32) at (1,-3) {$1$};
\node (z41) at (4,-4) {$4$};
\node (z42) at (2,-4) {$2$};
\node (z43) at (0,-4) {};
\node (z51) at (5,-5) {$4$};
\node (z52) at (3,-5) {$2$};
\node (z53) at (1,-5) {$0$};
\node (z61) at (6,-6) {$4$};
\node (z62) at (4,-6) {$3$};
\node (z63) at (2,-6) {$2$};
\node (z64) at (0,-6) {};

\draw[densely dotted, red, very thick] (z11) -- (z21) -- (z32) -- (z42) -- (z52) -- (z63);

\node at (6.9,-3.5) {$\longleftrightarrow$};

\begin{scope}[shift={(8,0)}]

\node (z11) at (1,-1) {$0$};
\node (z21) at (2,-2) {$0$};
\node (z22) at (0,-2) {};
\node (z31) at (3,-3) {$3$};
\node (z32) at (1,-3) {$0$};
\node (z41) at (4,-4) {$4$};
\node (z42) at (2,-4) {$1$};
\node (z43) at (0,-4) {};
\node (z51) at (5,-5) {$4$};
\node (z52) at (3,-5) {$1$};
\node (z53) at (1,-5) {$0$};
\node (z61) at (6,-6) {$4$};
\node (z62) at (4,-6) {$3$};
\node (z63) at (2,-6) {$1$};
\node (z64) at (0,-6) {};

\draw[densely dotted, red, very thick] (z11) -- (z21) -- (z32) -- (z42) -- (z52) -- (z63);
\end{scope}

\end{tikzpicture}
\end{subfigure}%
\begin{subfigure}[b]{.5\textwidth}
\centering
\begin{tikzpicture}[scale=0.45, inner sep=0.8pt]

\node (z11) at (1,-1) {$0.5$};
\node (z21) at (2,-2) {$2$};
\node (z22) at (0,-2) {};
\node (z31) at (3,-3) {$2$};
\node (z32) at (1,-3) {$1.5$};
\node (z41) at (4,-4) {$3$};
\node (z42) at (2,-4) {$2$};
\node (z43) at (0,-4) {};
\node (z51) at (5,-5) {$3$};
\node (z52) at (3,-5) {$2$};
\node (z53) at (1,-5) {$2$};
\node (z61) at (6,-6) {$4$};
\node (z62) at (4,-6) {$3$};
\node (z63) at (2,-6) {$2$};
\node (z64) at (0,-6) {};

\draw[densely dotted, red, very thick] (z32) -- (z42) -- (z53) -- (z63);

\draw[densely dotted, blue, very thick] (z11) -- (z21) -- (z31) -- (z41) -- (z51) -- (z62);

\node at (6.9,-3.5) {$\longleftrightarrow$};

\begin{scope}[shift={(8,0)}]

\node (z11) at (1,-1) {$0$};
\node (z21) at (2,-2) {$1$};
\node (z22) at (0,-2) {};
\node (z31) at (3,-3) {$1$};
\node (z32) at (1,-3) {$1$};
\node (z41) at (4,-4) {$2$};
\node (z42) at (2,-4) {$1$};
\node (z43) at (0,-4) {};
\node (z51) at (5,-5) {$2$};
\node (z52) at (3,-5) {$2$};
\node (z53) at (1,-5) {$1$};
\node (z61) at (6,-6) {$4$};
\node (z62) at (4,-6) {$2$};
\node (z63) at (2,-6) {$1$};
\node (z64) at (0,-6) {};

\draw[densely dotted, red, very thick] (z32) -- (z42) -- (z53) -- (z63);

\draw[densely dotted, blue, very thick] (z11) -- (z21) -- (z31) -- (z41) -- (z51) -- (z62);
\end{scope}

\end{tikzpicture}
\end{subfigure}

\caption{An illustration of bijection~\eqref{eq:sp-soPatternsBijection_1} for $n=3$ and $\lambda=(4,3,2)$.
In the two examples, for $k=1$ and $k=2$ respectively, the left-hand pattern is split orthogonal and the right-hand pattern is symplectic.
All and only the entries of the patterns modified throughout the algorithm lie along the $k$ colored non-intersecting paths (see proof of Theorem~\ref{thm:sp-soPatternsBijection}).
In the $k=2$ example, the red path is the first to be constructed in the map ``$\to$''; conversely, the blue path is the first to be constructed in the reverse map ``$\leftarrow$''.}
\label{fig:so-spBijection}
\end{figure}
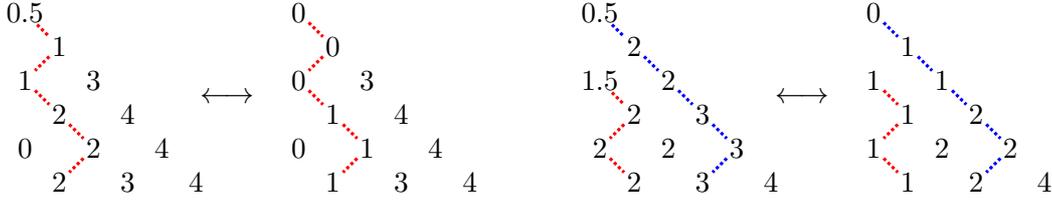

Let $\lambda$ be an $n$-partition and $1\leq k\leq n$.
For any split orthogonal pattern $z$ of height $2n$ and shape $\lambda$ with $k$ atypical entries, we will construct a symplectic pattern $z'$ of the same height and shape $\lambda-\epsilon$, being $\epsilon \in \{0,1\}^n$ with exactly $k$ ones.
We will do this via an algorithmic procedure, for which we refer to Figure~\ref{fig:so-spBijection} as a guiding pictorial example.
Let $1\leq l_1 < \dots < l_k \leq n$ be such that $z_{2l_1-1,l_1}, \dots, z_{2l_k-1, l_k}$ are all and only the atypical entries of $z$.
We construct the output pattern $z'$ by starting from the input $z$ and performing the following actions for all $l=l_k,\dots,l_1$ consecutively in \emph{decreasing order} of $l$:
\begin{enumerate}
\item Design a nearest neighbor path starting at $(2l-1,l)$ and ending on the $(2n)$-th row (i.e.\ the bottom row of the pattern) such that, given any $(i,j)$ in the path:
\begin{itemize}
\item if $z_{i,j} = z_{i+1,j+1}$, the path goes from $(i,j)$ to $(i+1,j+1)$;
\item if $z_{i,j} > z_{i+1,j+1}$, the path goes from $(i,j)$ to $(i+1,j)$;
\item if $(i+1,j+1)\notin\mathcal{I}$ and $(i+1,j)\in\mathcal{I}$, the path goes from $(i,j)$ to $(i+1,j)$;
\item if $(i+1,j+1)\notin\mathcal{I}$ and $(i+1,j)\notin\mathcal{I}$, which happens if and only if $i=2n$, the path stops at $(i,j)$.
\end{itemize}
\item Update $z$ by subtracting $1/2$ from the first entry $z_{2l-1,l}$ and $1$ from all the other entries along the path constructed in the previous step.
\end{enumerate}

Such a procedure generates $k$ uniquely determined nearest neighbor paths within $z$, which we claim to be non-intersecting.
This in particular implies that each path ends at a different index $(2n,j)$ of the bottom row.
Therefore, at the end of the algorithm the shape $\lambda$ of $z$ has been modified by subtracting $1$ from exactly $k$ distinct parts; the shape of $z'$ is thus $\lambda-\epsilon$, for some $\epsilon \in\{0,1\}^n$ with exactly $k$ ones.
To show the non-intersecting property, assume by contradiction that a given path $\pi$ intersects at least one of the previously constructed paths.
Let $(i+1,j+1)$ be the point of $\pi$ with smallest $i$ such that $(i+1,j+1)$ also belongs to some other path $\pi'$.
Then, the index that comes before $(i+1,j+1)$ in path $\pi$ is necessarily $(i,j)$ and, by construction, before updating the entries along path $\pi$, it has to be that $z_{i,j} = z_{i+1,j+1}$.
But since $(i+1,j+1)$ also belongs, by assumption, to the previously constructed path $\pi'$, entry $z_{i+1,j+1}$ must have been already decreased by at least $1/2$ during the update along path $\pi'$.
On the other hand, by our assumption that $(i+1,j+1)$ is the site with minimal $i$ where $\pi$ intersects any other path, we know that $(i,j)$ does not belong to any previous path, so $z_{i,j}$ was left unchanged by previous updates.
This means that the \emph{initial} pattern $z$, before any updates, satisfied $z_{i,j} < z_{i+1,j+1}$, contradicting the interlacing conditions.
This concludes the proof of the non-intersecting property of the update paths.

We now prove that the output pattern $z'$ has non-negative integer entries.
Notice first that the entries along each update path are non-decreasing (both before and after the corresponding update), due to the interlacing conditions and the construction.
Even though each path might contain various odd ends, the only atypical entry among these is the first one, since each atypical entry is the starting point of a path and paths do not intersect.
This implies that, along each path, the first entry is a half-integer $\geq 1/2$, whereas the next ones are integers $\geq 1$.
Now, each of the $k$ updates amounts to shifting the first entry along the update path by $-1/2$ and the remaining ones by $-1$; consequently, after such an update, the resulting pattern has still non-negative entries, with one atypical entry less ($z_{2l-1,l}$ is initially half-integer but becomes integer after the update).
At the end of the algorithm, all the entries of $z$ will thus be non-negative integers.

To conclude that the algorithm provides indeed the desired map, we are left to prove that it preserves the interlacing conditions of the pattern.
We will show that any given update preserves them, assuming inductively that all the previous updates do.
Since the update decreases all and only the entries along an update path $\pi$, it suffices to prove that, for all $(i,j) \in \pi$, after the update, we have $z_{i,j}\geq z_{i-1,j}$ (if $(i-1,j)\in\mathcal{I}$) and $z_{i,j}\geq z_{i+1,j+1}$ (if $(i+1,j+1)\in\mathcal{I}$).
We will only prove the former inequality: the proof of the latter is analogous, so we omit it.
If $z_{i,j}$ is an odd end there is nothing to prove, so we may assume this is not the case; in particular, $z_{i,j}$ will be an integer (before and after the update).
If $(i-1,j)$ belongs to $\pi$ and $z_{i-1,j}$ is half-integer, then before the update we have $z_{i,j} \geq z_{i-1,j} + 1/2$; after the update, $z_{i-1,j}$ and $z_{i,j}$ are decreased by $1/2$ and $1$, respectively, hence $z_{i,j}\geq z_{i-1,j}$ still holds.
If $(i-1,j)$ belongs to $\pi$ and $z_{i-1,j}$ is integer, then both $z_{i-1,j}$ and $z_{i,j}$ are decreased by $1$ and the interlacing condition between them continues to hold after the update.
We may assume from now on that $(i-1,j)$ does \emph{not} belong to $\pi$.
Let $(a,b)$ be the bottommost (i.e.\ with largest $a$) element of $\pi$ such that $a-b=i-1-j$.
The portion of the pattern of interest, \emph{before} the update along $\pi$, is then illustrated in the following diagram:
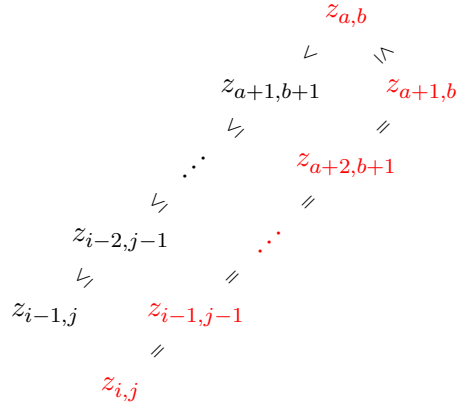
\[
\begin{tikzpicture}[scale=0.7, every path/.style={draw=none}, every node/.style={fill=white, scale=1, inner sep=1.5pt}, leq1/.style={rotate=45, node contents={\tiny $\leq$}}, leq2/.style={rotate=-45, node contents={\tiny $\leq$}}, less1/.style={rotate=45, node contents={\tiny $<$}}, eq1/.style={rotate=45, node contents={\tiny $=$}}]

\node (1) {$z_{i-1,j}$};
\node (2) at ($ (1) + (45:2) $) {$z_{i-2,j-1}$};
\node[rotate=45] (3) at ($ (2) + (45:2) $) {$\dots$};
\node (4) at ($ (3) + (45:2) $) {$z_{a+1,b+1}$};
\node[text=red] (5) at ($ (4) + (45:2) $) {$z_{a,b}$};
\node[text=red] (6) at ($ (5) + (-45:2) $) {$z_{a+1,b}$};
\node[text=red] (7) at ($ (6) + (-135:2) $) {$z_{a+2,b+1}$};
\node[text=red, rotate=45] (8) at ($ (7) + (-135:2) $) {$\dots$};
\node[text=red] (9) at ($ (8) + (-135:2) $) {$z_{i-1,j-1}$};
\node[text=red] (10) at ($ (9) + (-135:2) $) {$z_{i,j}$};

% arrows
\path (1) -- node[leq1]{} (2) -- node[leq1]{} (3) -- node[leq1]{} (4) -- node[less1]{} (5) -- node[leq2]{} (6);
\path (10) -- node[eq1]{} (9) -- node[eq1]{} (8) -- node[eq1]{} (7) -- node[eq1]{} (6);
\end{tikzpicture}
\]
Due to the choice of $(a,b)$, the path must contain $(a,b)$, $(a+1,b)$, $(a+2,b+1)$, \dots, $(i-1,j-1)$, $(i,j)$: the corresponding entries along $\pi$ are highlighted in red.
Let us justify the ordering of the entries in the diagram above:
\begin{itemize}
\item the weak inequalities follow from the interlacing conditions of $z$;
\item due to the choice of $(a,b)$, we know that $(a+1,b+1)$ does not belong to $\pi$, hence by the rules of the algorithm $z_{a,b} > z_{a+1,b+1}$;
\item the equalities follow from the fact that $\pi$ contains $(a+1,b)$, $(a+2,b+1)$, \dots, $(i-1,j-1)$, $(i,j)$ and, again, from the rules of the algorithm.
\end{itemize}
From the diagram above we deduce that $z_{i-1,j} < z_{i,j}$.
If $z_{i-1,j}$ were half-integer, then there would have been a previous path starting at $(i-1,j)$ and necessarily passing by $(i,j)$, thus violating the non-intersecting properties of the paths.
Therefore, $z_{i-1,j}$ and $z_{i,j}$ are both integers and, \emph{before} the update along $\pi$, $z_{i-1,j} < z_{i,j}$ actually means $z_{i-1,j} +1 \leq z_{i,j}$.
\emph{After} the update along $\pi$, $z_{i,j}$ is decreased by $1$ and $z_{i-1,j}$ is left unchanged, so the latter inequality turns into the desired $z_{i-1,j} \leq z_{i,j}$.
\vskip 1mm

Finally, to prove identity~\eqref{eq:typePreserved_1}, first notice that throughout our algorithm the same quantities are subtracted from two consecutive rows, unless either row contains an atypical entry; in the latter case, an extra $1/2$ is subtracted from the lower row.
This translates, in terms of the types of the patterns, as:
\begin{align*}
\type(z')_{2i-1} &= \begin{cases}
\type(z)_{2i-1} - \frac{1}{2} &\text{if $z_{2i-1,i}\in \frac{1}{2}+\Z$,} \\
\type(z)_{2i-1} &\text{otherwise,}
\end{cases}
\\
\type(z')_{2i} &= \begin{cases}
\type(z)_{2i} - \frac{1}{2} &\text{if $z_{2i-1,i}\in \frac{1}{2}+\Z$,} \\
\type(z)_{2i} &\text{otherwise.}
\end{cases}
\end{align*}
From the above, it is immediate to deduce~\eqref{eq:typePreserved_1}.
\vskip 1mm

We now prove that our map is a bijection, by describing the inverse algorithm that maps $z'$ to $z$.
Let $z'$ be a symplectic pattern of height $2n$ and shape $\lambda-\epsilon$, where $\epsilon \in\{0,1\}^n$ has exactly $k$ entries equal to $1$.
Let $1\leq m_1 < \dots < m_k \leq n$ be all and only the indices such that $\epsilon_{m_1}= \dots = \epsilon_{m_k}=1$.
We construct $z$ by starting from the input $z'$ and performing the following actions for all $m=m_1,\dots,m_k$ consecutively, in \emph{increasing order} of $m$:
\begin{enumerate}
\item Design a reverse nearest neighbor path starting at $(2n,m)$ and ending at some odd end $(2l-1,l)$, such that, given any $(i,j)$ in the path:
\begin{itemize}
\item if $z'_{i,j} = z'_{i-1,j-1}$, then the index that comes after $(i,j)$ is $(i-1,j-1)$;
\item if $z'_{i,j} < z'_{i-1,j-1}$ and $(i-1,j)\in\mathcal{I}$, then the index that comes after $(i,j)$ is $(i-1,j)$;
\item if $z'_{i,j} < z'_{i-1,j-1}$ and $(i-1,j)\notin\mathcal{I}$, then the path stops at $(i,j)$;
\item if $(i-1,j-1)\notin \mathcal{I}$ and $(i-1,j)\in \mathcal{I}$, the index that comes after $(i,j)$ is $(i-1,j)$;
\item if $(i-1,j-1)\notin\mathcal{I}$ and $(i-1,j)\notin\mathcal{I}$, i.e.\ $(i,j)=(1,1)$, then the path stops at $(i,j)$.
\end{itemize}
\item Update $z'$ by adding $1/2$ to the last entry $z'_{2l-1,l}$ and $1$ to all the other entries along the path constructed in the previous step.
\end{enumerate}
One can verify that the above algorithm returns a split orthogonal pattern $z$ of height $2n$ and shape $\lambda$ with $k$ atypical entries, and that our direct algorithm maps $z\mapsto z'$, as desired.
\end{proof}

\subsection{Transition between even symplectic and odd orthogonal characters}
\label{subsec:CBtransition}

Recall from Section~\ref{sec:patterns} (see in particular~\eqref{eq:spSchur} and~\eqref{eq:soOddSchurSplit}) that $(2n+1)$-orthogonal Schur polynomials can be defined via the same weight monomials as $(2n)$-symplectic Schur polynomials.
The difference is that orthogonal characters are generated by a larger set of patterns, whose entries may also be half-integers according to certain rules.
This observation, in combination with a probabilistic motivation that will emerge in Section~\ref{sec:results}, lead us to define a class of symmetric functions that interpolate between characters of types C and B via a parameter $\beta$.

\begin{definition}
\label{def:CBtrans}
We define the {\bf\emph{$\CB$-interpolating Schur polynomial}} to be the following function in variables $x=(x_1,\dots,x_n)$, parametrized by $\beta$ and indexed by an $n$-partition or $n$-half-partition $\lambda$:
\begin{equation}
\label{eq:CBtrans}
\begin{split}
\trans{\lambda}(x; \beta) 
:= \sum_{z\in \soGT^{(2n)}_{\lambda}}
\beta^{\abs{\atyp(z)}}
\prod_{i=1}^{n} x_i^{\type(z)_{2i} - \type(z)_{2i-1}} \, ,
\end{split}
\end{equation}
where $\abs{\atyp(z)}$ is the number of atypical entries of $z$.
\end{definition}
If $\beta=0$, the sum in~\eqref{eq:CBtrans} is over all split orthogonal patterns of height $2n$ and shape $\lambda$ with no atypical entries; in particular, if $\lambda$ is an $n$-partition (respectively, $n$-half-partition), these are $(2n)$-symplectic patterns (respectively, $(2n)$-symplectic patterns where each entry is increased by $1/2$).
Using~\eqref{eq:spSchur}, one then essentially recovers a $(2n)$-symplectic Schur polynomial in both cases:
\begin{equation}
\label{eq:transition=Sp}
\trans{\lambda}(x; 0) =
\begin{cases}
\sp^{(2n)}_{\lambda}(x) &\text{if $\lambda$ is an $n$-partition} \, , \\
\left[\prod_{i=1}^n x_i\right]^{-1/2} \sp^{(2n)}_{\lambda - 1/2}(x) &\text{if $\lambda$ is an $n$-half-partition} \, ,
\end{cases}
\end{equation}
with $\lambda - 1/2 := (\lambda_1 - 1/2, \dots, \lambda_n - 1/2)$.
On the other hand, it is clear that, for $\beta=1$, \eqref{eq:CBtrans} reduces to the Definition~\ref{def:soOddSchurSplit} of $(2n+1)$-orthogonal Schur polynomial:
\begin{equation}
\label{eq:transition=So}
\trans{\lambda}(x; 1) =
\so^{(2n+1)}_{\lambda}(x)
\end{equation}
for all $n$-partitions and $n$-half-partitions $\lambda$.

Using the combinatorial bijection introduced in Subsection~\ref{subsec:bijectionPatterns}, 
we are able to express $\CB$-interpolating Schur polynomials in terms of even symplectic characters:
\begin{theorem}
\label{thm:CBtrans=sp}
If $\lambda$ is an $n$-partition, then
\begin{equation}
\label{eq:CBtrans=sp1}
\trans{\lambda}(x; \beta)
= \sum_{\epsilon\in\{0,1\}^n} \beta^{\abs{\epsilon}} \sp^{(2n)}_{\lambda-\epsilon}(x) \, ,
\end{equation}
where $\abs{\epsilon}$ is the $1$-norm of $\epsilon$ and by convention $\sp^{(2n)}_{\mu}(x) := 0$ if $\mu$ is not a partition\footnote{Notice that $\lambda-\epsilon:= (\lambda_1 -\epsilon_1, \dots, \lambda_n - \epsilon_n)$ is not necessarily a partition for all $\epsilon\in \{0,1\}^n$, as the non-decreasing ordering might fail.}.
If $\lambda$ is an $n$-half-partition, then
\begin{equation}
\label{eq:CBtrans=sp2}
\trans{\lambda}(x; \beta)
= \prod_{i=1}^n \left[\beta x_i^{1/2} + x_i^{-1/2} \right]
\sp^{(2n)}_{\lambda - \frac{1}{2}}(x) \, .
\end{equation}
\end{theorem}
Observe that, in~\eqref{eq:CBtrans=sp2}, $\lambda-\frac{1}{2}$ is an \emph{integer} partition, so $\sp^{(2n)}_{\lambda - \frac{1}{2}}$ is indeed a symplectic Schur polynomial.
\begin{proof}
By Definition~\ref{def:CBtrans}, we may write
\begin{equation}
\label{eq:CBtrans_2}
\begin{split}
\trans{\lambda}(x; \beta)
= \sum_{\epsilon\in\{0,1\}^n} \sum_{\substack{z\in \soGT^{(2n)}_{\lambda}\colon \\ \atyp(z)=\epsilon}}
\beta^{\abs{\epsilon}}
\prod_{i=1}^{n} x_i^{\type(z)_{2i} - \type(z)_{2i-1}} \, .
\end{split}
\end{equation}

Let first $\lambda$ be an $n$-partition.
By~\eqref{eq:sp-soPatternsBijection_1}-\eqref{eq:typePreserved_1} of Theorem~\ref{thm:sp-soPatternsBijection}, we then have:
\begin{align*}
\trans{\lambda}(x; \beta)
&= \sum_{k=0}^n \beta^k \sum_{\substack{z\in\soGT^{(2n)}_{\lambda}\colon \\ \abs{\atyp(z)}=k}} \prod_{i=1}^n x_i^{\type(z)_{2i} - \type(z)_{2i-1}} \\
&= \sum_{k=0}^n \beta^k \sum_{\substack{\epsilon\in \{0,1\}^n\colon \\ \abs{\epsilon}=k}} \sum_{z'\in\spGT^{(2n)}_{\lambda-\epsilon}} \prod_{i=1}^n x_i^{\type(z')_{2i} - \type(z')_{2i-1}} \, ,
\end{align*}
keeping in mind that $\spGT^{(2n)}_{\lambda-\epsilon}$ is empty if $\lambda-\epsilon$ is not a partition.
By Definition~\ref{def:spSchur}, the rightmost sum over $z'\in \spGT^{(2n)}_{\lambda-\epsilon}$ equals $\sp^{(2n)}_{\lambda-\epsilon}(x)$, which implies~\eqref{eq:CBtrans=sp1}.

Let now $\lambda$ be an $n$-half-partition.
By~\eqref{eq:sp-soPatternsBijection_2}-\eqref{eq:typePreserved_2} of Theorem~\ref{thm:sp-soPatternsBijection}, we may rewrite~\eqref{eq:CBtrans_2} as
\begin{align*}
\trans{\lambda}(x; \beta)
&= \sum_{\epsilon\in\{0,1\}^n}
\sum_{z'\in \spGT^{(2n)}_{\lambda-\frac{1}{2}}}
\prod_{i=1}^{n} \beta^{\epsilon_i} x_i^{\type(z')_{2i} - \type(z')_{2i-1} + \epsilon_i - 1/2} \\
&= \sum_{z'\in \spGT^{(2n)}_{\lambda-\frac{1}{2}}}
\prod_{i=1}^{n} \left[ \sum_{\epsilon_i = 0}^1 \beta^{\epsilon_i} x_i^{\epsilon_i - 1/2} \right]
x_i^{\type(z')_{2i} - \type(z')_{2i-1}} \\
&= \prod_{i=1}^n \left[\beta x_i^{1/2} + x_i^{-1/2} \right]
\sum_{z'\in \spGT^{(2n)}_{\lambda-\frac{1}{2}}}
x_i^{\type(z')_{2i} - \type(z')_{2i-1}} \, .
\end{align*}
By Definition~\ref{def:spSchur}, the latter sum equals $\sp^{(2n)}_{\lambda - \frac{1}{2}}(x)$, which implies~\eqref{eq:CBtrans=sp2}.
\end{proof}

As a consequence of the latter theorem and the invariance properties of symplectic characters, we can deduce the invariance properties of the $\CB$-interpolating Schur polynomials with respect to the variables $x_i$'s, for any fixed $\beta$.
Namely, $\trans{\lambda}(x;\beta)$ is always symmetric in the variables $x_i$'s and, when $\lambda$ is a partition, also invariant under multiplicative inversion of any $x_i$.
However, \eqref{eq:CBtrans=sp2} implies that $\trans{\lambda}(x;\beta)$ is \emph{not} invariant under inversion of the variables when $\lambda$ is a half-partition, unless $\beta = 1$ (which corresponds to the odd orthogonal case).

Notice that, in the case $\beta=0$, Theorem~\ref{thm:CBtrans=sp} just reduces to~\eqref{eq:transition=Sp}.
On the other hand, the specialization to $\beta=1$ leads to:

\begin{corollary}
\label{coro:soOdd=sp}
If $\lambda$ is an $n$-partition, then
\begin{equation}
\label{eq:soOdd=sp1}
\so^{(2n+1)}_{\lambda}(x)
= \sum_{\epsilon\in\{0,1\}^n}  \sp^{(2n)}_{\lambda-\epsilon}(x) \, ,
\end{equation}
with the same conventions as in Theorem~\ref{thm:CBtrans=sp}.
If $\lambda$ is an $n$-half-partition, then
\begin{equation}
\label{eq:soOdd=sp2}
\so^{(2n+1)}_{\lambda}(x)
= \prod_{i=1}^n \left[x_i^{1/2} + x_i^{-1/2} \right]
\sp^{(2n)}_{\lambda - \frac{1}{2}}(x) \, .
\end{equation}
\end{corollary}

Identity~\eqref{eq:soOdd=sp1} first appeared in~\cite{Pro89b,Sun90a}.
On the other hand, \eqref{eq:soOdd=sp2}, which easily follows from the Weyl character formulas~\eqref{eq:soOddWeyl} and~\eqref{eq:spWeyl}, can be found e.g.\ in~\cite{Pro94}.

We now obtain determinantal formulas of Weyl character type for the $\CB$-interpolating polynomials.
In principle, we could do this via a generalization of Proctor's proof~\cite{Pro93} for the Weyl character formula of type C.
However, this approach would be considerably long, as was the case in~\cite{Pro93}, and not innovative.
Instead, we propose a different strategy that reduces the proof of the general $\beta$ case to the Weyl character formula of type C, via the formulas of Theorem~\ref{thm:CBtrans=sp} and determinant expansions.
\begin{theorem}
\label{thm:transitionWeyl}
If $\lambda$ is an $n$-partition, then
\begin{equation}
\label{eq:transitionWeyl1}
\trans{\lambda}(x;\beta)
= \frac{\underset{1\leq i,j\leq n}{\det}\left(
x_j^{\lambda_i + n-i+1} - x_j^{-(\lambda_i + n-i+1)}
+ \beta \left[x_j^{\lambda_i + n-i} - x_j^{-(\lambda_i + n-i)}\right]
\right)}
{\underset{1\leq i,j\leq n}{\det}\left( x_j^{n-i+1} - x_j^{-(n-i+1)} \right)} \, .
\end{equation}
If $\lambda$ is an $n$-half-partition,
then
\begin{equation}
\label{eq:transitionWeyl2}
\trans{\lambda}(x;\beta)
= \frac{\underset{1\leq i,j\leq n}{\det}\left( \left[\beta x_j^{1/2} + x_j^{-1/2}\right] \left[
x_j^{\lambda_i + n-i+1/2} - x_j^{-(\lambda_i + n-i+1/2)}
\right] \right)}
{\underset{1\leq i,j\leq n}{\det}\left( x_j^{n-i+1} - x_j^{-(n-i+1)} \right)} \, .
\end{equation}
\end{theorem}
\begin{proof}
Let $\lambda$ be an $n$-partition.
By~\eqref{eq:spWeyl}, the symplectic characters appearing in~\eqref{eq:CBtrans=sp1} can be written as
\begin{equation}
\label{eq:spWeyl_transProof1}
\sp^{(2n)}_{\lambda-\epsilon}(x)
= \frac{\underset{1\leq i,j\leq n}{\det}\left( x_j^{\lambda_i - \epsilon_i + n-i+1} - x_j^{-(\lambda_i - \epsilon_i + n-i+1)} \right)}
{\underset{1\leq i,j\leq n}{\det}\left( x_j^{n-i+1} - x_j^{-(n-i+1)} \right)} \, ,
\end{equation}
for $\epsilon \in\{0,1\}^n$ such that $\lambda-\epsilon$ is a partition.
When $\lambda-\epsilon$ is \emph{not} a partition, in the sense that the non-decreasing condition fails, we have $\lambda_i - \epsilon_i < \lambda_{i+1} - \epsilon_{i+1}$ for some $i$.
In such a case, since $\lambda_i \geq \lambda_{i+1}$ and $\epsilon_i, \epsilon_{i+1}\in\{0,1\}$, it must hold that $\lambda_i - \epsilon_i = \lambda_{i+1}-\epsilon_{i+1} -1$.
Therefore, the $i$-th and $(i+1)$-th rows of the numerator matrix in~\eqref{eq:spWeyl_transProof1} are equal and the right-hand side of the equation vanishes.
On the other hand, when $\lambda-\epsilon$ is not a partition, the left-hand side also vanishes by the convention adopted in Theorem~\ref{thm:CBtrans=sp}.
We conclude that~\eqref{eq:spWeyl_transProof1} is actually valid for all $\epsilon\in\{0,1\}^n$.
Using the multilinearity of determinants, from~\eqref{eq:CBtrans=sp1} we then obtain:
\begin{align*}
\trans{\lambda}(x; \beta)
&= \sum_{\epsilon\in\{0,1\}^n} \beta^{\epsilon_1 + \dots + \epsilon_n}
\frac{\underset{1\leq i,j\leq n}{\det}\left( x_j^{\lambda_i - \epsilon_i + n-i+1} - x_j^{-(\lambda_i - \epsilon_i + n-i+1)} \right)}
{\underset{1\leq i,j\leq n}{\det}\left( x_j^{n-i+1} - x_j^{-(n-i+1)} \right)} \\
&= \frac{\underset{1\leq i,j\leq n}{\det}\left( \sum_{\epsilon_i = 0}^1 \beta^{\epsilon_i} \left[x_j^{\lambda_i - \epsilon_i + n-i+1} - x_j^{-(\lambda_i - \epsilon_i + n-i+1)} \right] \right)}
{\underset{1\leq i,j\leq n}{\det}\left( x_j^{n-i+1} - x_j^{-(n-i+1)} \right)} \, .
\end{align*}
The latter formula is clearly equivalent to~\eqref{eq:transitionWeyl1}.

Let now $\lambda$ be an $n$-half-partition.
By~\eqref{eq:spWeyl}, we can rewrite \eqref{eq:CBtrans=sp2} as
\[
\trans{\lambda}(x; \beta)
= \prod_{i=1}^n \left[\beta x_i^{1/2} + x_i^{-1/2} \right]
\frac{\underset{1\leq i,j\leq n}{\det}\left( x_j^{\lambda_i + n-i+1/2} - x_j^{-(\lambda_i + n-i+1/2)} \right)}
{\underset{1\leq i,j\leq n}{\det}\left( x_j^{n-i+1} - x_j^{-(n-i+1)} \right)} \, .
\]
Applying the multilinear property to take the prefactor into the numerator determinant, we obtain~\eqref{eq:transitionWeyl2}.
\end{proof}

Notice that the above determinantal formulas have denominators of type C, just because they are deduced from~\eqref{eq:CBtrans=sp1} and~\eqref{eq:CBtrans=sp2}.
In particular, for $\beta=0$, it is immediate to recover~\eqref{eq:transition=Sp} using the Weyl character formula~\eqref{eq:spWeyl} for symplectic characters.
On the other hand, for $\beta=1$, \eqref{eq:transitionWeyl1} and~\eqref{eq:transitionWeyl2} provide further determinantal expressions of odd orthogonal characters, equivalent to the Weyl character formula of type B.
This can be shown starting from~\eqref{eq:soOddWeyl} and multiplying the $j$-th column of both the numerator and the denominator matrices by $x_j^{1/2} + x_j^{-1/2}$:
\begin{align*}
\so^{(2n+1)}_{\lambda}(x)
&= \frac{\underset{1\leq i,j\leq n}{\det}\left( \left[ x_j^{1/2} + x_j^{-1/2}\right] \left[
x_j^{\lambda_i + n-i+1/2} - x_j^{-(\lambda_i + n-i+1/2)}
\right] \right)}
{\underset{1\leq i,j\leq n}{\det}\left( \left[ x_j^{1/2} + x_j^{-1/2}\right] \left[x_j^{n-i+1/2} - x_j^{-(n-i+1/2)} \right] \right)} \\
&= \frac{\underset{1\leq i,j\leq n}{\det}\left(
x_j^{\lambda_i + n-i+1} - x_j^{-(\lambda_i + n-i+1)}
+ x_j^{\lambda_i + n-i} - x_j^{-(\lambda_i + n-i)}
\right)}
{\underset{1\leq i,j\leq n}{\det}\left( x_j^{n-i+1} - x_j^{-(n-i+1)} + x_j^{n-i} - x_j^{-(n-i)} \right)} \\
&= \frac{\underset{1\leq i,j\leq n}{\det}\left(
x_j^{\lambda_i + n-i+1} - x_j^{-(\lambda_i + n-i+1)}
+ x_j^{\lambda_i + n-i} - x_j^{-(\lambda_i + n-i)}
\right)}
{\underset{1\leq i,j\leq n}{\det}\left( x_j^{n-i+1} - x_j^{-(n-i+1)} \right)} \, .
\end{align*}
The latter equality follows from a manipulation on the denominator determinant: subtract the $n$-th row from the $(n-1)$-th row, then subtract the $(n-1)$-th row from the $(n-2)$-th row and so on.
The above display proves our claim that both~\eqref{eq:transitionWeyl1} and~\eqref{eq:transitionWeyl2} reduce to~\eqref{eq:soOddWeyl} for $\beta=1$.

\vskip 2mm

\noindent {\bf Relation to Sundaram's tableaux.}
There are some profound links  between the results of Subsections~\ref{subsec:bijectionPatterns}-\ref{subsec:CBtransition} and the combinatorics of Sundaram's orthogonal tableaux~\cite{Sun90a}, as we now explain.
A \textbf{\emph{Sundaram's orthogonal tableau}} is a semistandard Young tableau in the alphabet $1< \overline{1} < 2 < \overline{2} < \dots < n < \overline{n} < \infty$ such that (i) the entries are \emph{weakly} increasing along the rows and down the columns, (ii) all ``finite'' entries $1,\overline{1},2,\overline{2},\dots,n,\overline{n}$ are in \emph{strict} increasing order down the columns (iii) the entries in row $i$ are not less than $i$, and (iv) there is at most one $\infty$ symbol in each row.
The latter condition ensures that the cells that contain $\infty$ form a so-called \emph{vertical strip} contained in the shape $\lambda$ of the tableau.
Removing the $\infty$'s from a Sundaram's tableau of shape $\lambda$, one simply obtains a symplectic tableau (as defined in Subsection~\ref{subsec:sp}) of a certain shape $\mu$ such that $\lambda/\mu$ is a vertical strip.
Vice versa, a symplectic tableau is a Sundaram's tableau with no $\infty$'s.

Fix now a partition $\lambda$ and consider the combinatorial bijection~\eqref{eq:sp-soPatternsBijection_1}.
Recalling from Subsection~\eqref{subsec:sp} the correspondence between symplectic patterns and tableaux, one can realize that the set of symplectic patterns $z' \in \spGT^{(2n)}_{\lambda-\epsilon}$ with $\epsilon \in\{0,1\}^n$ is in bijection with the set of Sundaram's tableaux of shape $\lambda$.
In the tableau, the diagram $\mu:= \lambda-\epsilon$ contains the ``finite'' entries and the skew shape $\lambda/ \mu$, of size $\abs{\epsilon}$, contains the $\infty$'s; moreover, the fact that each $\epsilon_i$ is either $0$ or $1$ forces $\mu$ to be a \emph{vertical strip}.
On the other hand, we already observed in Subsection~\ref{subsec:so} that split orthogonal patterns $z\in \soGT^{(2n)}_{\lambda}$ bijectively correspond to Koike-Terada tableaux~\cite{KT87, KT90} of shape $\lambda$; the number of atypical entries in the pattern equals the number of ``circled'' symbols in the Koike-Terada tableau.
Therefore, \eqref{eq:sp-soPatternsBijection_1} can be seen as a correspondence between Koike-Terada tableaux with $k$ ``circled'' symbols and Sundaram's tableaux with $k$ occurrences of the $\infty$ symbol.
After completion of this work, and led by the useful comments of an anonymous referee, we realized that a bijection of this type had been also discovered in~\cite{CS13}: their approach is based on a \emph{jeu de taquin} procedure on tableaux that can be shown to be equivalent to our construction of non-intersecting paths on patterns (see proof of Theorem~\ref{thm:sp-soPatternsBijection}).
From an algorithmic point of view, $k$ represents the number of non-intersecting paths in our proof of Theorem~\ref{thm:sp-soPatternsBijection} (respectively, the number of \emph{jeu de taquin} operations in the framework of~\cite{CS13}) needed to map the split orthogonal pattern onto the symplectic pattern with perturbed shape (respectively, the Koike-Terada tableau onto the Sundaram's tableau).

Sundaram~\cite{Sun90a} showed that the $(2n+1)$-orthogonal Schur polynomials indexed by a partition $\lambda$ can be expressed as
\begin{equation}
\label{eq:so_Sundaram}
\so^{(2n+1)}_{\lambda}(x)
= \sum_{T} x_i^{\#\{\text{$\overline{i}$'s in $T$}\} - \#\{\text{$i$'s in $T$}\}}
= \sum_{\mu} \sp^{(2n)}_{\mu}(x) \, ,
\end{equation}
where the first sum is over all Sundaram's tableaux $T$ of shape $\lambda$, whereas the second sum is over all partitions $\mu \subseteq \lambda$ such that the skew shape $\lambda/ \mu$ is a vertical strip.
The first equality was proven in~\cite{Sun90a} using an insertion algorithm.
The second equality follows immediately from the definition of Sundaram's tableaux and the definition of a symplectic character as generating function of symplectic tableaux.
Notice that the equality between the left-hand side and the right-hand side of~\eqref{eq:so_Sundaram} is the way in which~\eqref{eq:soOdd=sp1} first appeared in~\cite{Pro89b, Sun90a}.
On the other hand, \eqref{eq:CBtrans=sp1} allows us to interpret our $\CB$-interpolating Schur polynomials (when indexed by an integer partition) as generating functions of Sundaram's tableaux where all the $\infty$'s are assigned a weight $\beta$:
\begin{equation}
\label{eq:CB_Sundaram}
\trans{\lambda}(x; \beta) = \sum_{T} \beta^{\#\{\text{$\infty$'s in $T$}\}}
x_i^{\#\{\text{$\overline{i}$'s in $T$}\} - \#\{\text{$i$'s in $T$}\}} \, ,
\end{equation}
with the sum running over all Sundaram's tableaux $T$ of shape $\lambda$.
When $\beta=0$, the $\infty$'s are not allowed in the tableaux and we recover $\sp^{(2n)}_{\lambda}(x)$; when $\beta=1$, the $\infty$'s are given weight $1$ and we recover $\so^{(2n+1)}_{\lambda}(x)$.
\vskip 2mm

\noindent {\bf Relation to Koornwinder polynomials.}
Koornwinder polynomials~\cite{Koo92} can be viewed as a $\BC$-analog of standard Macdonald polynomials, in the sense that they are associated to the root system of type $\BC$ instead of type A.
They depend on the usual parameters $q$ and $t$ of Macdonald polynomials as well as four extra interchangeable parameters $t_0, t_1, t_2, t_3$.
We now briefly introduce them, following the exposition of~\cite{RW21}.
Denoting by
\[
(z;q)_{\infty} := \prod_{k=0}^{\infty} (1-q^k z)
\]
the $q$-shifted factorial, we define the Koornwinder density in variables $x=(x_1,\dots,x_n)$ by
\[
\Delta(x;q,t;t_0,t_1,t_2,t_3)
:= \prod_{i=1}^n \Bigg[ \prod_{\epsilon\in\{\pm 1\}} \frac{ (x_i^{2\epsilon};q)_{\infty} }{ \prod_{k=0}^3 (t_k x_i^{\epsilon}; q)_{\infty} } \Bigg]
\prod_{1\leq i<j\leq n} \Bigg[ \prod_{\epsilon, \delta \in\{\pm 1\}} \frac{(x_i^{\epsilon} x_j^{\delta}; q)_{\infty}}{(t x_i^{\epsilon} x_j^{\delta}; q)_{\infty} } \Bigg] \, .
\]
For $\abs{q}, \abs{t}, \abs{t_0}, \dots, \abs{t_3} < 1$, we then define the inner product
\[
\langle f,g\rangle^{(n)}_{q,t;t_0,t_1,t_2,t_3} :=
\frac{1}{2^n n! (2\pi\i)^n} \int_{\mathbb{T}^n} f(x) g(x^{-1}) \Delta(x;q,t;t_0,t_1,t_2,t_3) \prod_{i=1}^n \frac{\diff x_i}{x_i} \, ,
\]
where $\mathbb{T}^n:=\{x\in\C^n\colon \abs{x_1}=\dots=\abs{x_n}=1\}$ is the $n$-dimensional complex torus and $f,g$ are Laurent polynomials with coefficients in $\C$.
Koornwinder polynomials are then defined as the unique $\langle \cdot , \cdot \rangle_{q,t;t_0,t_1,t_2,t_3}$-orthogonal family of $\BC$-invariant Laurent polynomials $K_{\lambda} = K_{\lambda}(x;q,t;t_0,t_1,t_2,t_3)$ on $\C$ indexed by $n$-partitions $\lambda$ with leading coefficient $x^{\lambda}$.
In other words, they satisfy the following properties:
\begin{itemize}
\item
they are invariant under permutations of the $x_i$'s and inversion of any of them;
\item
they satisfy $\langle K_{\lambda} , K_{\mu} \rangle_{q,t;t_0,t_1,t_2,t_3} = 0$ for all $\lambda \neq \mu$;
\item
for any $n$-partition $\mu$, the coefficient of the monomial $x^{\mu} := x_1^{\mu_1} \cdots x_n^{\mu_n}$ in $K_{\lambda}$ is zero unless $\mu \leq \lambda$ (in the ``dominance order'' sense, i.e.\ $\mu_1 + \dots + \mu_k \leq \lambda_1 + \dots + \lambda_k$ for $1\leq k\leq n$), and the coefficient of $x^{\lambda}$ is precisely $1$.
\end{itemize}

Setting $q=0$ and any\footnote{Since the Koornwinder density $\Delta$ is symmetric with respect to the $t_k$'s, Koornwinder polynomials also are.} two of $t_0,\dots,t_3$ to be zero in $K_{\lambda}$, one recovers the so-called \emph{Hall-Littlewood polynomials} of type $\BC$~\cite{Ven15}, which possess an expansion over the Weyl group of type $\BC$.
When taking also $t=0$, such an expansion takes a determinantal form~\cite{RW21}:
\[
K_{\lambda}(x;0,0;a,b,0,0)
= \frac{\underset{1\leq i,j\leq n}{\det}\left(
x_j^{\lambda_i + n-i-1}(x_j - a)(x_j - b) - x_j^{-\lambda_i - n+i-1} (ax_j - 1) (bx_j -1) \right)}
{\underset{1\leq i,j\leq n}{\det}\left( x_j^{n-i+1} - x_j^{-(n-i+1)} \right)} \, .
\]
It is immediate to see that the latter expression corresponds to~\eqref{eq:transitionWeyl1} for $a=-\beta$ and $b=0$, thus implying that $\CB$-interpolating polynomials belong to the Koornwinder family.
More precisely, for any $n$-partition $\lambda$, we have
\[
\trans{\lambda}(x;\beta) \equiv K_{\lambda}(x;0,0;-\beta,0,0,0) \, .
\]
The general definition of Koornwinder polynomials, given before, is abstract and difficult to handle in practice.
On the other hand, our construction of $\CB$-interpolating polynomials provides a concrete and explicit combinatorial interpretation, based on Gelfand-Tsetlin patterns, of a one-parameter specialization of Koornwinder polynomials.

It would be interesting to investigate if, and to what extent, the results of the present article extend to a further one-parameter generalization of $\CB$-interpolating polynomials, corresponding to the above determinantal specialization of Koornwinder polynomials (allowing $b\neq 0$).
Do these polynomials still have a combinatorial interpretation as generating functions of certain patterns?
Do they satisfy similar decomposition identities?
Do they appear in any LPP models?

\subsection{Transition between even and odd orthogonal characters}
\label{subsec:DBtransition}

Recall from~\eqref{eq:soEvenSchur} and~\eqref{eq:soOddSchur} that $(2n)$- and $(2n+1)$-orthogonal characters can be defined via the same weight monomials.
The patterns that generate odd orthogonal characters have one more row (which does not appear in the weight monomials, though) than the patterns that generate even orthogonal characters.
Again motivated by a probabilistic significance that will emerge in Section~\ref{sec:results}, it is then natural to introduce polynomials that interpolate, via an extra parameter $\alpha$, between the characters of types D and B.

\begin{definition}
\label{def:DBtrans}
We define the {\bf\emph{$\DB$-interpolating Schur polynomials}} to be the following function in $x=(x_1,\dots,x_n)$, with parameter $\alpha$, and indexed by an $n$-partition or $n$-half-partition $\lambda$:
\begin{equation}
\label{eq:DBtrans}
\transDB{\lambda}(x; \alpha)
:= \sum_{z\in \oGT^{(2n)}_{\lambda}}
x_1^{z_{1,1}}
\prod_{i=2}^{n} x_i^{\sign(z_{2i-3,i-1}) \sign(z_{2i-1,i}) [\type(z)_{2i-1} - \type(z)_{2i-2}]}
\alpha^{\sum_{i=1}^n (\lambda_i - z_{2n-1,i})} \, .
\end{equation}
\end{definition}

Notice that the exponent of $\alpha$ in~\eqref{eq:DBtrans} does \emph{not} coincide with $\type(z)_{2n}$ according to~\eqref{eq:soType}, as $z_{2n-1,n}$ (the only entry involved that might be negative) is not taken in absolute value.
Moreover, due to the interlacing conditions~\eqref{eq:interlacingAbs}, the exponent of $\alpha$ is non-negative, while it equals $0$ if and only if $\lambda_i = z_{2n-1,i}$ for all $1\leq i\leq n$.
Therefore, for $\alpha=0$, the general term of the sum in~\eqref{eq:DBtrans} vanishes unless the $(2n-1)$-th row equals the shape $\lambda$; we are then reduced to sum over $z\in \oGT_{\lambda}^{(2n-1)}$, thus obtaining the even orthogonal character $\so^{(2n)}_{\lambda}(x)$ defined in~\eqref{eq:soEvenSchur}.
It is likewise obvious that, for $\alpha=1$, \eqref{eq:DBtrans} reduces to the odd orthogonal character $\so^{(2n+1)}_{\lambda}(x)$ defined in~\eqref{eq:soOddSchur}.
To sum up, as announced, the functions defined above interpolate between characters of type D and B:
\begin{equation}
\label{eq:DBinterpolation}
\transDB{\lambda}(x; \alpha)
= \begin{cases}
\so^{(2n)}_{\lambda}(x) &\text{if } \alpha=0 \, , \\
\so^{(2n+1)}_{\lambda}(x) &\text{if } \alpha=1 \, ,
\end{cases}
\end{equation}
for all $n$-partitions or $n$-half-partitions $\lambda$.
We stress that a transition between types D and B may only exist for unsigned partitions, although even orthogonal characters can be also indexed by signed partitions.

We can also express a $\DB$-interpolating Schur polynomial as a ``linear combination'' of even orthogonal characters, with coefficients being powers of $\alpha$:
\begin{proposition}
\label{prop:DBbranching}
For all $n$-partitions (respectively, $n$-half-partitions) $\lambda$, we have
\begin{equation}
\label{eq:DBbranching}
\transDB{\lambda}(x; \alpha)
= \sum_{\mu_{\epsilon} \prec \lambda} \alpha^{\sum_{i=1}^{n-1} (\lambda_i - \mu_i) + (\lambda_n - \epsilon\mu_n)} \cdot \so^{(2n)}_{\mu_{\epsilon}}(x) \, ,
\end{equation}
where the sum is over all signed $n$-partitions (respectively, signed $n$-half-partitions) $\mu_{\epsilon}$ that upwards interlace with $\lambda$.
In particular, when $\lambda=u^{n}:= \underbrace{(u,\dots,u)}_{n \text{ times}}$ for some $u\in\frac{1}{2}\Z_{\geq 0}$, we have
\begin{equation}
\label{eq:DBbranchingRectangular}
\transDB{u^{n}}(x; \alpha)
= \sum_{k=0}^{2u} \alpha^k \cdot \so^{(2n)}_{(u^{n-1}, u-k )} (x) \, .
\end{equation}
\end{proposition}
\begin{proof}
The sum in~\eqref{eq:DBtrans}, which defines a $\DB$-interpolating Schur polynomial, is over orthogonal patterns $z$ of height $2n$ and shape $\lambda$.
Setting $\mu_i := \abs{z_{2n-1,i}}$ for $1\leq i\leq n$ and $\epsilon := \sign(z_{2n-1,n})$, it turns out that $\mu_{\epsilon}$ is any signed $n$-partition (or any signed $n$-half-partition, if $\lambda$ is an $n$-half-partition) that upwards interlaces with $\lambda$.
On the other hand, the first $2n-1$ rows of $z$ form a new orthogonal pattern of height $2n-1$ and shape $\mu_\epsilon$.
The sum in~\eqref{eq:DBtrans} may thus be split into two nested sums, over $\mu_\epsilon \prec \lambda$ and over patterns in $\oGT^{(2n-1)}_{\mu_\epsilon}$ respectively:
\[
\begin{split}
\transDB{\lambda}(x; \alpha)
= \sum_{\mu_\epsilon \prec \lambda}
&\alpha^{\sum_{i=1}^{n-1} (\lambda_i - \mu_i) + (\lambda_n - \epsilon\mu_n)} \\
&\times \sum_{z\in \oGT^{(2n-1)}_{\mu_\epsilon}}
x_1^{z_{1,1}}
\prod_{i=2}^{n} x_i^{\sign(z_{2i-3,i-1}) \sign(z_{2i-1,i}) [\type(z)_{2i-1} - \type(z)_{2i-2}]} \, .
\end{split}
\]
By Definition~\ref{def:soEvenSchur}, the inner sum is a $(2n)$-orthogonal Schur polynomial in $x$, indexed by $\mu_\epsilon$; this yields~\eqref{eq:DBbranching}.

Let us now specialize~\eqref{eq:DBbranching} to the case $\lambda = u^{n}$, for $u\in\frac{1}{2}\Z_{\geq 0}$.
The interlacing between $\mu_\epsilon$ and $u^{n}$, defined by~\eqref{eq:interlacingAbs}, forces $\lambda_i = \mu_i =u$ for $1\leq i\leq n-1$.
It then suffices to sum over all integers $0\leq k\leq 2u$, where $k = \lambda_n - \epsilon\mu_n = u -\epsilon\mu_n$ (notice that $k$ is always integer, independently of whether $u$ is integer or half-integer).
This readily yields~\eqref{eq:DBbranchingRectangular}.
\end{proof}

The case $\alpha=1$ in~\eqref{eq:DBbranching} degenerates to the classical branching rule from $\SO_{2n+1}(\C)$ to $\SO_{2n}(\C)$ -- see~\cite{Pro94}.
It follows directly from~\eqref{eq:DBbranching} and the invariance properties of even orthogonal characters that $\DB$-interpolating polynomials, for any fixed $\alpha$, are symmetric in the variables $x_1,\dots,x_n$ and invariant under inversion of an \emph{even} number of them.

Throughout this work we will be especially interested in $\DB$-interpolating Schur polynomials indexed by ``rectangular (half-)partitions'', as in~\eqref{eq:DBbranchingRectangular}.
In this case, it turns out that our interpolating function with \emph{arbitrary} parameter $\alpha$ essentially reduces to a rectangular shaped even orthogonal character with one extra variable $\alpha^{-1}$, as the next proposition states.
\begin{proposition}
\label{prop:DBrectangular=so}
For $u\in\frac{1}{2}\Z_{\geq 0}$, we have
\begin{equation}
\label{eq:DBrectangular=so}
\transDB{u^{n}}(x; \alpha)
= \alpha^u \cdot \so^{(2n+2)}_{u^{n+1}}(x,\alpha^{-1}) \, .
\end{equation}
\end{proposition}
\begin{proof}
Any orthogonal pattern $z$ of height $2n$ and shape $u^{n}$, due to the interlacing conditions, satisfies $z_{2n-1,i}=u$ for $1\leq i\leq n-1$.
Therefore, by Definition~\ref{def:DBtrans} we have
\begin{equation}
\label{eq:DBrectangular=so_proof1}
\transDB{u^{n}}(x; \alpha)
= \sum_{z\in \oGT^{(2n)}_{u^{n}}}
x_1^{z_{1,1}}
\prod_{i=2}^{n} x_i^{\sign(z_{2i-3,i-1}) \sign(z_{2i-1,i}) [\type(z)_{2i-1} - \type(z)_{2i-2}]}
\alpha^{u-z_{2n-1,n}} \, .
\end{equation}
On the other hand, by Definition~\ref{def:soEvenSchur} we have
\begin{equation}
\label{eq:DBrectangular=so_proof2}
\begin{split}
\so^{(2n+2)}_{u^{n+1}}(x,\alpha^{-1})
= \sum_{z\in \oGT^{(2n+1)}_{u^{n+1}}}
&x_1^{z_{1,1}}
\prod_{i=2}^{n} x_i^{\sign(z_{2i-3,i-1}) \sign(z_{2i-1,i}) [\type(z)_{2i-1} - \type(z)_{2i-2}]} \\
&\times \alpha^{-\sign(z_{2n-1,n}) \sign(z_{2n+1,n+1}) [\type(z)_{2n+1} - \type(z)_{2n}]} \, .
\end{split}
\end{equation}
Now, by the interlacing conditions, any orthogonal pattern $z$ of height $2n+1$ and shape $u^{n+1}$ must have $(2n)$-th row equal to $u^{n}$ and $(2n-1)$-th row equal to $(u^{n-1}, z_{2n-1,n})$.
This implies:
\[
\sign(z_{2n+1,n+1}) = \sign(u) = +1 \, ,
\qquad
\type(z)_{2n+1} = u \, ,
\qquad
\type(z)_{2n} = u-\abs{z_{2n-1,n}} \, .
\]
The exponent of $\alpha$ in~\eqref{eq:DBrectangular=so_proof2} then equals $-z_{2n-1,n}$.
Moreover, the sum in~\eqref{eq:DBrectangular=so_proof2} can be now taken over orthogonal patterns of height $2n$ and shape $u^{n}$, thus obtaining:
\[
\so^{(2n+2)}_{u^{n+1}}(x,\alpha^{-1})
= \sum_{z\in \oGT^{(2n)}_{u^{n}}}
x_1^{z_{1,1}}
\prod_{i=2}^{n} x_i^{\sign(z_{2i-3,i-1}) \sign(z_{2i-1,i}) [\type(z)_{2i-1} - \type(z)_{2i-2}]} \alpha^{-z_{2n-1,n}} \, .
\]
Comparing the latter with~\eqref{eq:DBrectangular=so_proof1}, we obtain~\eqref{eq:DBrectangular=so}.
\end{proof}

\section{Character identities and last passage percolation}
\label{sec:results}

In Section~\ref{sec:intro} we have introduced the Last Passage Percolation (LPP) model.
In this section we explain how LPP with certain symmetries on the weight array is related to character identities and decompositions of the irreducible polynomial representations of classical groups.

For bounded Cauchy or Littlewood sums, we will often use the following conventions.
If $\mu$ is a fixed $n$-partition (respectively, $n$-half-partition), a sum over $\lambda \subseteq \mu$ will be taken on all $n$-partitions (respectively, $n$-half-partitions) $\lambda$ such that $\lambda \subseteq \mu$.
Analogously, a sum over $\lambda_\epsilon \subseteq \mu$ will be taken either on signed $n$-partitions or on signed $n$-half-partitions according to whether $\mu$ is an $n$-partition or an $n$-half-partition.
Recall also the notation $u^{n}$ for the $n$-tuple $(u,\dots,u)$, and $(a,b) :=(a_1,\dots,a_n,b_1,\dots,b_m)$ for the concatenation of two tuples $a=(a_1,\dots,a_n)$ and $b=(b_1,\dots,b_m)$.
Finally, recall the notation $\oddrows$ from~\eqref{eq:oddRows}.

\subsection{Antidiagonally symmetric LPP and decompositions of symplectic and odd orthogonal characters}
\label{subsec:antisymLPP}

Let us first consider the LPP model with weight array $\{W_{i,j}\colon 1\leq i,j\leq N\}$ symmetric about the antidiagonal $\{i+j=N+1\}$, i.e.\ such that $W_{i,j} = W_{N-j+1,N-i+1}$ for all $(i,j)$.
In this case, the link to combinatorics emerges when the weights on and above the antidiagonal are independent and such that, for all $k\in\Z_{\geq 0}$,
\begin{equation}
\label{eq:antisymWeights}
\P(W_{i,j} = k) =
\begin{cases}
(1-p_{N-i+1} p_j) (p_{N-i+1} p_j)^k &\text{if } i+j< N+1 \, , \\
\dfrac{1-p_j^2}{1 + \beta p_j} \beta^{k \bmod 2} p_j^k &\text{if } i+j = N+1 \, . \\
\end{cases}
\end{equation}
We define the normalization constant (whose dependence on the $p_i$'s is dropped from the notation) for the joint distribution of the above weights:
\begin{equation}
\label{eq:antisymLPP_normalization}
c^{\antisymSq}_{\beta}
:= \prod_{1\leq i<j\leq N} \frac{1}{1-p_i p_j}
\prod_{1\leq j\leq N} \frac{1 + \beta p_j}{1-p_j^2} \, .
\end{equation}
Denote by $L^{\antisymSq}_{\beta}(N,N)$ the point-to-point LPP time from $(1,1)$ to $(N,N)$ with a weight array symmetric about the antidiagonal and distributed as in~\eqref{eq:antisymWeights}.
Baik and Rains~\cite{BR01a} showed that the distribution of $L^{\antisymSq}_{\beta}(N,N)$ is given in terms of classical Schur polynomials.
Here, our main result states that the same distribution can be also expressed in terms of the $\CB$-interpolating Schur polynomials with parameter $\beta$ introduced in Subsection~\ref{subsec:CBtransition}:
\begin{theorem}
\label{thm:antisymLPP}
For $u\in\frac{1}{2}\Z_{\geq 0}$, the following quantities are equal:
\begin{align*}
        A^{\antisymSq}_{\beta} &:= c^{\antisymSq}_{\beta} \cdot \P\left(L^{\antisymSq}_{\beta}(N,N) \leq 2u\right) , \\
        B^{\antisymSq}_{\beta} &:= \sum_{\mu \subseteq (2u)^{N}} \beta^{\oddrows\mu} \cdot \schur^{(N)}_{\mu} (p_1,\dots,p_N) \, , \\
        C^{\antisymSq}_{\beta} &:= \left[ \prod_{i=1}^N p_i \right]^u \trans{u^{N}} (p_1,\dots,p_N;\beta) \, , \\
        D^{\antisymSq}_{\beta} &:= \left[ \prod_{i=1}^N p_i \right]^u \sum_{\lambda \subseteq u^{N-n}} \trans{(u^{2n-N}, \lambda)} (p_1,\dots, p_n;\beta) \cdot \trans{\lambda} (p_{n+1}, \dots, p_{N};\beta) \, ,
\end{align*}
where $D^{\antisymSq}_{\beta}$ is valid for any integer $n$ with $\lceil N/2 \rceil \leq n \leq N$.
\end{theorem}

Before discussing Theorem~\ref{thm:antisymLPP}, we deduce its specializations to the cases $\beta=0$ and $\beta=1$.
For $\beta=0$, each weight on the antidiagonal is even and distributed as twice a geometric random variable, i.e.\ $\P(W_{i,j} = k) = (1-p_j^2) p_j^k$ for $k\in 2\Z_{\geq 0}$ and $i+j=N+1$.
In particular, $L^{\antisymSq}_{0}(N,N)$ is almost surely even, hence it suffices to compute its distribution function at even integers.
We may therefore ignore the case of $u$ half-integer in the next corollary.
Thanks to~\eqref{eq:transition=Sp}, $\CB$-interpolating Schur polynomials degenerate to \emph{even} symplectic characters for $\beta=0$.
Moreover, for $\beta=0$ we establish a further connection with \emph{odd} symplectic characters.
\begin{corollary}
\label{coro:antisymLPP_0}
For $u\in\Z_{\geq 0}$, the following quantities are equal:
\begin{align*}
        A^{\antisymSq}_0 &:= c^{\antisymSq}_0 \cdot \P\left(L^{\antisymSq}_{0}(N,N) \leq 2u\right) , \\
        B^{\antisymSq}_0 &:= \sum_{\substack{\mu\subseteq (2u)^{N}\colon \\ \oddrows\mu=0}} \schur^{(N)}_{ \mu} (p_1,\dots,p_N) \, , \\
        C^{\antisymSq}_0 &:= \left[ \prod_{i=1}^N p_i \right]^u \sp^{(2N)}_{u^{N}} (p_1,\dots,p_N) \, , \\
        D^{\antisymSq}_0 &:= \left[ \prod_{i=1}^N p_i \right]^u
\sum_{\lambda\subseteq u^{N-n}} \sp^{(2n)}_{(u^{2n-N}, \lambda)} (p_1,\dots,p_n) \cdot
\sp^{(2N-2n)}_{\lambda} (p_{n+1}, \dots, p_{N}) \, , \\
		E^{\antisymSq}_0 &:= \left[ \prod_{i=1}^{N-1} p_i \right]^u 
\sum_{\lambda\subseteq u^{N-n}}  \sp^{(2n+1)}_{(u^{2n+1-N}, \lambda)} (p_1,\dots,p_n;p_N) \cdot
\sp^{(2N-2n-1)}_{\lambda} (p_{n+1}, \dots, p_{N-1}; p_N) \, ,
\end{align*}
where $D^{\antisymSq}_{0}$ holds for $\lceil N/2 \rceil \leq n \leq N$ and $E^{\antisymSq}_{0}$ holds for $\lceil (N-1)/2 \rceil \leq n \leq N-1$.
\end{corollary}

Let us now consider $\beta=1$.
All the weights on the antidiagonal now follow the geometric distribution defined by $\P(W_{i,j} = k) = (1-p_j) p_j^k$ for $k\in \Z_{\geq 0}$ and $i+j=N+1$.
Recalling from~\eqref{eq:transition=So} that $\CB$-interpolating Schur polynomials degenerate to odd orthogonal characters for $\beta=1$, we have:
\begin{corollary}
\label{coro:antisymLPP_1}
For $u\in\frac{1}{2}\Z_{\geq 0}$, the following quantities are equal:
\begin{align*}
        A^{\antisymSq}_1 &:= c^{\antisymSq}_{1} \cdot \P\left(L^{\antisymSq}_{1}(N,N) \leq 2u\right) , \\
        B^{\antisymSq}_1 &:= \sum_{\mu \subseteq (2u)^{N}} \schur^{(N)}_{\mu} (p_1,\dots,p_N) \, , \\
        C^{\antisymSq}_1 &:= \left[ \prod_{i=1}^N p_i \right]^u \so^{(2N+1)}_{u^{N}} (p_1,\dots,p_N) \, , \\
        D^{\antisymSq}_1 &:= \left[ \prod_{i=1}^N p_i \right]^u
\sum_{\lambda \subseteq u^{N-n}} \so^{(2n+1)}_{(u^{2n-N}, \lambda)} (p_1,\dots, p_n) \cdot
\so^{(2N-2n+1)}_{\lambda} (p_{n+1}, \dots, p_{N}) \, ,
\end{align*}
where $D^{\antisymSq}_1$ holds for any $\lceil N/2 \rceil \leq n \leq N$.
\end{corollary}

Let us discuss the above results.
Notice first that, if $N=2n$ in Theorem~\ref{thm:antisymLPP}, the two $\CB$-interpolating Schur polynomials of $D^{\antisymSq}_{\beta}$ have the same number $n$ of variables and are both indexed by the same $n$-partition $\lambda$.
Therefore, for $N=2n$, $D^{\antisymSq}_{0}$ and $D^{\antisymSq}_{1}$ read as bounded Cauchy sums of, respectively, even symplectic and odd orthogonal Schur polynomials of the same shape $\lambda$.
Analogously, in the special case $N=2n+1$, $E^{\antisymSq}_{0}$ reads as a bounded Cauchy sum of odd symplectic Schur polynomials of the same shape $\lambda$.

Identity $A^{\antisymSq}_{\beta}= B^{\antisymSq}_{\beta}$, whose proof we omit, corresponds to~\eqref{eq:antisymSchur_intro} and traces back to Baik and Rains~\cite{BR01a} (see also~\cite{Fer04} for a Poissonized model); it can be proved by applying the standard $\RSK$ correspondence on square matrices with symmetry about the antidiagonal.
We will rather prove in Section~\ref{sec:RSK} another identity that relates to last passage percolation, i.e.\ $A^{\antisymSq}_\beta = C^{\antisymSq}_\beta$, reformulating the probabilistic model in terms of point-to-line paths, applying the $\RSK$ on triangular arrays and then using certain pattern transformations.
This point of view is inspired by earlier works of the authors on LPP models and a positive temperature version of it known as \emph{log-gamma directed polymer}~\cite{BZ19a, NZ17, Bis18}.

In section~\ref{sec:nearlyRectangular} we will also give a direct proof of $B^{\antisymSq}_\beta = C^{\antisymSq}_\beta$ based on an identity established by Krattenthaler~\cite{Kra98} for a symplectic character of ``nearly rectangular'' shape.
This proof involves classical tools from the theory of symmetric functions such as the dual Pieri rule, but also uses our formulas~\eqref{eq:CBtrans=sp1} and~\eqref{eq:CBtrans=sp2} that express CB-interpolating Schur polynomials in terms of symplectic characters.
Identity $B^{\antisymSq}_\beta = C^{\antisymSq}_\beta$ implicitly appeared in~\cite[Theorem 4.1]{RW21} in a more general form involving Macdonald polynomials, but {\em crucially} only in the case $u$ half-integer.
More precisely, if $(q,t)$ are the Macdonald parameters, the degeneration $q=t=0$ of Rains-Warnaar's formula coincides with $B^{\antisymSq}_\beta = C^{\antisymSq}_\beta$ for $u$ half-integer.
The latter is the ``trivial'' case when $\CB$-interpolating polynomials essentially reduce to symplectic characters, as the parameter $\beta$ factorizes out -- see Theorem~\ref{thm:CBtrans=sp}.

Specializations $B^{\antisymSq}_0 = C^{\antisymSq}_0$ and $B^{\antisymSq}_1 = C^{\antisymSq}_1$ are known and respectively due to~\cite[Theorem 4.1]{Ste90b} and~\cite[Ex.\ I.5.16]{Mac95} (see also~\cite[Corollary 7.4]{Ste90a}).
Our result should then be viewed as unifying such special cases.

Finally, in Section~\ref{sec:decomposition} we will provide bijective proofs of $C^{\antisymSq}_\beta = D^{\antisymSq}_\beta$ and $C^{\antisymSq}_0 = E^{\antisymSq}_0$ based on decomposition of split orthogonal patterns and symplectic patterns respectively.
For convenience, we reformulate the latter identities (and their specializations) in a separate theorem, as follows.
\begin{theorem}
\label{thm:decompositionCB}
Given integers $n \geq m\geq 1$, we have
\begin{align}
\label{eq:decompositionCB}
\trans{u^{n+m}}(x,y;\beta)
&= \sum_{\lambda\subseteq u^{m}}
\trans{(u^{n-m}, \lambda)}(x;\beta) \cdot
\trans{\lambda}(y;\beta) \, , &
&u\in\frac{1}{2}\Z_{\geq 0} \, ,
\intertext{where $x=(x_1, \dots, x_n)$ and $y=(y_1, \dots, y_m)$.
In particular, for $\beta=0$ and $\beta=1$, we deduce:}
\label{eq:decompositionC}
\sp^{(2n+2m)}_{u^{n+m}}(x,y)
&= \sum_{\lambda\subseteq u^{m}}
\sp^{(2n)}_{(u^{n-m}, \lambda)}(x) \cdot
\sp^{(2m)}_{\lambda}(y) \, , &
&u\in\Z_{\geq 0} \, , \\
\label{eq:decompositionB}
\so^{(2n+2m+1)}_{u^{n+m}}(x,y)
&= \sum_{\lambda\subseteq u^{m}}
\so^{(2n+1)}_{(u^{n-m}, \lambda)}(x) \cdot
\so^{(2m+1)}_{\lambda}(y) \, , &
&u\in\frac{1}{2}\Z_{\geq 0} \, .
\intertext{Moreover, we have}
\label{eq:decompositionOddSp}
\sp^{(2n+2m+2)}_{u^{n+m+1}}(x,y,s)
&= s^{-u} \!\!\!
\sum_{\lambda\subseteq u^{m+1}}
\!\!\!
\sp^{(2n+1)}_{(u^{n-m}, \lambda)}(x;s) \cdot
\sp^{(2m+1)}_{\lambda}(y;s) \, , &
&u\in\Z_{\geq 0} \, ,
\end{align}
where $s$ is an extra univariate variable.
\end{theorem}

As mentioned in the introduction, identities~\eqref{eq:decompositionC} and~\eqref{eq:decompositionB} have been also proved by Okada~\cite[Theorem 2.2]{Oka98} via determinantal calculus.

\subsection{Diagonally symmetric LPP and decompositions of even orthogonal characters.}
\label{subsec:symLPP}

We now discuss the LPP model for a weight array $\{W_{i,j}\colon 1\leq i,j\leq N\}$ symmetric about the diagonal $\{i=j\}$, i.e.\ such that $W_{i,j}=W_{j,i}$ for all $(i,j)$.
Here the link to combinatorics occurs when the weights on and above the antidiagonal are independent and such that, for all $k\in\Z_{\geq 0}$,
\begin{align}
\label{eq:symWeights}
\P(W_{i,j}=k)=
\begin{cases}
(1-p_i p_j)(p_i p_j)^k & 1\leq i<j\leq N  \, , \\
(1-\alpha p_j)(\alpha p_j)^k & 1\leq i=j\leq N \, ,
\end{cases}
\end{align}
with parameters $p_1, \dots, p_N, \alpha$.
Notice that here the parameter $\alpha$ modulates the intensity of the diagonal weights, and therefore plays a role analogous to the parameter $\beta$ in~\eqref{eq:antisymWeights} for the antidiagonal weights.
We define the normalization constant (whose dependence on the $p_i$'s is dropped from the notation) for the joint distribution of the above weights:
\begin{equation}
\label{eq:symLPP_normalization}
c^{\symSq}_{\alpha} :=
\prod_{1\leq i< j\leq N} \frac{1}{1-p_i p_j}
\prod_{1\leq j\leq N} \frac{1}{1-\alpha p_j} \, .
\end{equation}
Denote by $L^{\symSq}_{\alpha}(N,N)$ the point-to-point LPP time from $(1,1)$ to $(N,N)$ with weights symmetric about the diagonal and geometrically distributed as specified in~\eqref{eq:symWeights}.
We remark that, because of the symmetry constraint, $L^{\symSq}_{\alpha}(N,N)$ coincides with the LPP time with the same weights and paths restricted to stay on or above the main diagonal $\{i=j\}$.
Besides its well-established formula in terms of classical Schur polynomials~\cite{BR01a}, it turns out that the distribution function of $L^{\symSq}_{\alpha}(N,N)$ can be also expressed in terms of $\DB$-interpolating Schur polynomials with parameter $\alpha$ (see Subsection~\ref{subsec:DBtransition}) as well as even orthogonal characters, as the next theorem states.
\begin{theorem}
\label{thm:symLPP}
For $u\in\frac{1}{2}\Z_{\geq 0}$, the following quantities are equal:
\begin{align*}
        A^{\symSq}_{\alpha} &:= c^{\symSq}_{\alpha} \cdot \P\left(L^{\symSq}_{\alpha}(N,N) \leq 2u\right) , \\
        B^{\symSq}_{\alpha} &:= \sum_{\mu \subseteq (2u)^{N}}
\alpha^{\oddrows\mu'} \cdot
\schur^{(N)}_{\mu}(p_1,\dots,p_{N}) \, , \\
        C^{\symSq}_{\alpha} &:= \left[\prod_{i=1}^{N} p_i \right]^u \transDB{u^{N}}(p_1^{-1},\dots,p_{N}^{-1}; \alpha)
= \left[ \alpha \prod_{i=1}^{N} p_i \right]^u \so^{(2N+2)}_{u^{N+1}}(p_1^{-1},\dots,p_{N}^{-1},\alpha^{-1}) \, , \\
        D^{\symSq}_{\alpha} &:= \left[ \alpha \prod_{i=1}^N p_i \right]^u \sum_{\lambda_{\delta} \subseteq u^{N-n}} \so^{(2n+2)}_{(u^{2n+1-N}, \lambda_{\delta})}(p_1^{-1},\dots,p_n^{-1},\alpha^{-1}) \cdot \so^{(2N-2n)}_{\lambda_{\delta}}(p_{n+1}^{-1},\dots,p_{N}^{-1}) \, ,
\end{align*}
where $D^{\symSq}_{\alpha}$ holds for any $\lceil (N-1)/2 \rceil \leq n \leq N$.
\end{theorem}

Observe also that the two equivalent expressions of $C^{\symSq}_{\alpha}$ in terms of $\DB$-interpolating and even orthogonal Schur polynomials indexed by ``rectangular'' partitions are due to Proposition~\ref{prop:DBrectangular=so}.
For the diagonally symmetric LPP, we thus have a transition between characters of type D and B when $\alpha$ goes from $0$ to $1$.

The case $\alpha=0$ corresponds to all the weights on the diagonal being zero, thus $L^{\symSq}_{0}(2n,2n)$  equals the LPP time from $(1,2)$ to $(N-1,N)$ with paths restricted to stay strictly above the diagonal.
Using~\eqref{eq:DBinterpolation} and~\eqref{eq:soEvenVanishingVar}, we obtain:
\begin{corollary}
\label{coro:symLPP_0}
For $u\in\frac{1}{2}\Z_{\geq 0}$, the following quantities are equal:
\begin{align*}
        A^{\symSq}_{0} &:= c^{\symSq}_{0} \cdot \P\left(L^{\symSq}_{0}(N,N) \leq 2u\right) , \\
        B^{\symSq}_{0} &:= \sum_{\substack{\mu\subseteq (2u)^{N}, \\ \oddrows\mu'=0}} \schur^{(N)}_{\mu}(p_1,\dots,p_N) \, , \\
        C^{\symSq}_{0} &:= \left[ \prod_{i=1}^N p_i \right]^u \so^{(2N)}_{u^{N}}(p_1^{-1},\dots,p_N^{-1}) \, , \\
        D^{\symSq}_{0} &:= \left[ \prod_{i=1}^N p_i \right]^u \sum_{\lambda_\delta \subseteq u^{N-n}} \so^{(2n)}_{(u^{2n-N} , \lambda_\delta)}(p_1^{-1},\dots,p_n^{-1}) \cdot \so^{(2N-2n)}_{\lambda_\delta}(p_{n+1}^{-1},\dots,p_{N}^{-1}) \, ,
\end{align*}
where $D^{\symSq}_{0}$ holds for any $\lceil N/2 \rceil \leq n \leq N$.
\end{corollary}

For $\alpha=1$, the $\DB$-interpolating Schur polynomial and the even orthogonal character appearing in $C^{\symSq}_{\alpha}$ both reduce to the same odd orthogonal character, by~\eqref{eq:DBinterpolation} and~\eqref{eq:soEvenRect=soOddRect} respectively.
Theorem~\ref{thm:symLPP} thus specializes to:
\begin{corollary}
\label{coro:symLPP_1}
For $u\in\frac{1}{2}\Z_{\geq 0}$, the following quantities are equal:
\begin{align*}
        A^{\symSq}_{1} &:= c^{\symSq}_{1} \cdot \P\left(L^{\symSq}_{1}(N,N) \leq 2u\right) , \\
        B^{\symSq}_{1} &:= \sum_{\mu \subseteq (2u)^{N}}
\schur^{(N)}_{\mu}(p_1,\dots,p_{N}) \, , \\
        C^{\symSq}_{1} &:= \left[ \prod_{i=1}^{N} p_i \right]^u \so^{(2N+1)}_{u^{N}}(p_1,\dots,p_{N}) \, , \\
        D^{\symSq}_{1} &:= \left[ \prod_{i=1}^N p_i \right]^u \sum_{\lambda_{\delta} \subseteq u^{N-n}} \so^{(2n+2)}_{(u^{2n+1-N}, \lambda_{\delta})}(p_1,\dots,p_n,1) \cdot \so^{(2N-2n)}_{\lambda_{\delta}}(p_{n+1},\dots,p_{N}) \, ,
\end{align*}
where $D^{\symSq}_{\alpha}$ holds for any $\lceil (N-1)/2 \rceil \leq n \leq N$.
\end{corollary}

Notice that in $C^{\symSq}_1$ -- and therefore in $D^{\symSq}_1$ -- it is not necessary to invert the variables $p_1,\dots,p_N$ as in Theorem~\ref{thm:symLPP}, because odd orthogonal Schur polynomials are invariant under inversion of \emph{any} number of variables (see Subsection~\ref{subsec:so}).

Let us now discuss the results above and their proofs.
Identity $A^{\symSq}_{\alpha} = B^{\symSq}_{\alpha}$ was proved by Baik and Rains~\cite{BR01a} by applying the standard $\RSK$ on symmetric matrices.
In the case of antidiagonal symmetry (see Subsection~\ref{subsec:antisymLPP}), we are able to link directly the LPP model to the interpolating Schur polynomials, reformulating the probabilistic problem in terms of point-to-line paths.
Notice however that, in the case of diagonal symmetry, it does not seem to be possible to prove $A^{\symSq}_{\alpha} = C^{\symSq}_{\alpha}$ directly without passing through $B^{\symSq}_{\alpha}$.

In Section~\ref{sec:nearlyRectangular} we will prove $B^{\symSq}_{\alpha} = C^{\symSq}_{\alpha}$ using the ``branching rule'' of Proposition~\ref{prop:DBbranching} as well as an identity of Krattenthaler~\cite{Kra98} for an even orthogonal character of ``nearly rectangular'' shape.
The specializations $B^{\symSq}_0 = C^{\symSq}_0$ and $B^{\symSq}_1 = C^{\symSq}_1$ are already known and respectively due to~\cite[Theorem~2.3]{Oka98} and~\cite[Ex.\ I.5.16]{Mac95}.
Thus, again, our result should be viewed as unifying these special cases.

By comparing Corollaries~\ref{coro:antisymLPP_1} and~\ref{coro:symLPP_1} one can notice that $B^{\antisymSq}_1$ and $C^{\antisymSq}_1$ exactly coincide, respectively, with $B^{\symSq}_1$ and $C^{\symSq}_1$.
This implies that $A^{\antisymSq}_{1}$ and $A^{\symSq}_{1}$ are equivalent, i.e.\ the distributions of $L_1^{\antisymSq}(N,N)$ and $L_1^{\symSq}(N,N)$ are identical.
Observe also that the distribution~\eqref{eq:antisymWeights} for $\beta=1$ equals the distribution~\eqref{eq:symWeights} for $\alpha=1$, under the row reversal $i \mapsto N-i+1$.
However, the fact that a $\antisymSq$-symmetric LPP model be equivalent to the $\symSq$-symmetric LPP model obtained by reversing the rows of the weights has \emph{no} reason to hold in general.
It is rather specific to the distribution~\eqref{eq:antisymWeights} with $\beta=1$, and yields a non-trivial identity in law between piecewise linear functionals of geometric random variables, which in the special case $N=2$ reads as
\begin{equation}
\label{eq:equalityInLaw}
W_{1,1} + \max(W_{1,2},W_{2,1}) + W_{1,1} \stackrel{\text{d}}{\equiv}
W_{2,1} + \max(W_{1,1},W_{1,1}) + W_{1,2} \, .
\end{equation}
Using the identity $\max(a,b) = a+b-\min(a,b)$, \eqref{eq:equalityInLaw} is easily seen to be equivalent to the fact that the minimum of two independent geometric random variables, with parameters $p$ and $q$ respectively, follows a geometric distribution with parameter $pq$.
Thus, the equality in distribution between $L_1^{\antisymSq}(N,N)$ and $L_1^{\symSq}(N,N)$ can be viewed as a high dimensional analog of this fact.

Finally, in Section~\ref{sec:decomposition} we will prove, via decomposition of orthogonal patterns of odd height, that $C^{\symSq}_{\alpha} = D^{\symSq}_{\alpha}$ (and therefore $C^{\symSq}_0 = D^{\symSq}_0$).
We notice that, comparing again Corollaries~\ref{coro:antisymLPP_1} and~\ref{coro:symLPP_1}, $C^{\symSq}_1$ and $C^{\antisymSq}_1$ coincide exactly, whereas $D^{\symSq}_1$ and $D^{\antisymSq}_1$ do not.
More specifically, identities $C^{\symSq}_1 = D^{\symSq}_1$ and $C^{\antisymSq}_1 = D^{\antisymSq}_1$ are \emph{different} decompositions of the \emph{same} odd orthogonal character in terms of odd and even orthogonal characters, respectively.
For convenience, we reformulate in the next theorem the decomposition formulas to be proven.
\begin{theorem}
\label{thm:decompositionD}
For $u\in\frac{1}{2}\Z_{\geq 0}$, we have
\begin{align}
\label{eq:decompositionD}
\so^{(2n+2m)}_{u^{n+m}_{\epsilon}}(x,y)
&= \sum_{\lambda_{\delta} \subseteq u^{m}}
\so^{(2n)}_{(u^{n-m} , \lambda_{\delta})}(x) \cdot
\so^{(2m)}_{\lambda_{\delta\epsilon}}(y) \, , && n \geq m \, , \\
\label{eq:decompositionBintoD}
\so^{(2n+2m+1)}_{u^{n+m}}(x,y)
&= \sum_{\lambda_{\delta} \subseteq u^{m}}
\so^{(2n+2)}_{(u^{n+1-m} , \lambda_{\delta})}(x,1) \cdot
\so^{(2m)}_{\lambda_{\delta}}(y) \, , && n+1 \geq m \, ,
\end{align}
where $x=(x_1, \dots, x_n)$ and $y=(y_1, \dots, y_m)$.
\end{theorem}
Notice that the sign of the partition on the left-hand side of~\eqref{eq:decompositionD} is $\epsilon$; on the other hand, the signs of the partitions on the right-hand side are $\delta$ and $\delta\epsilon$ respectively, hence they are either equal or opposite according to whether $\epsilon$ is positive or negative.
We also stress that, in case $\lambda_+ = \lambda_-$ (i.e.\ $\lambda_m=0$), the signed partition is counted only once in the sum.

As stated in the introduction, \eqref{eq:decompositionD} was originally proven by Okada~\cite[Theorem 2.2]{Oka98} via determinantal calculus.

\subsection{Doubly symmetric LPP and decompositions of general linear characters.}
\label{subsec:doubleSymLPP}

Finally, let us consider the LPP model for a weight array $\{W_{i,j}\colon 1\leq i,j\leq 2n\}$ symmetric about both the antidiagonal $\{i+j=2n+1\}$ and the diagonal $\{i=j\}$, i.e.\ such that $W_{i,j} = W_{2n-j+1, 2n-i+1} = W_{j,i}$ for all $(i,j)$.
In this case, we choose the weights on or above the antidiagonal and on or above the diagonal to be independent and distributed as follows:
\begin{equation}
\label{eq:doubleSymWeights}
\P(W_{i,j} = k) =
\begin{cases}
(1-p_i p_j) (p_i p_j)^k &\text{if } i<j<2n-i+1 \, , \\
\dfrac{1-p_j^2}{1 + \beta p_j} \beta^{k \bmod 2} p_j^k &\text{if } i<j=2n-i+1 \, , \\
(1-\alpha p_j) (\alpha p_j)^k &\text{if } 1\leq i=j\leq n \, ,
\end{cases}
\end{equation}
with parameters $p_1, \dots, p_{2n}, \alpha, \beta$ satisfying $p_{2n-i+1} = p_i$ for all $1\leq i\leq n$.
We define the normalization constant for the joint distribution of the above weights:
\begin{equation}
\label{eq:doubleSymLPP_normalization}
c^{\doubleSymSq}_{\alpha,\beta} := \prod_{1\leq i < j\leq n} \frac{1}{(1-p_i p_j)^2}
\prod_{1\leq j\leq n} \frac{1 + \beta p_j}{1-p_j^2}
\frac{1}{1-\alpha p_j} \, .
\end{equation}
Let $L^{\doubleSymSq}_{\alpha,\beta}(2n,2n)$ be the point-to-point LPP time from $(1,1)$ to $(2n,2n)$ with a doubly symmetric weight array distributed according to~\eqref{eq:doubleSymWeights}.
We can express the distribution of $L^{\doubleSymSq}_{\alpha,\beta}(2n,2n)$ in terms of $\CB$-interpolating polynomials:
\begin{theorem}
\label{thm:doubleSymLPP}
For all $u\in\frac{1}{2}\Z_{\geq 0}$, we have
\begin{equation}
\label{eq:doubleSymLPP}
c^{\doubleSymSq}_{\alpha,\beta} \cdot \P\left(L^{\doubleSymSq}_{\alpha, \beta}(2n,2n) \leq 2u\right)
= \left[ \prod_{i=1}^n p_i \right]^u\sum_{\lambda \subseteq u^{n}} \alpha^{\oddrows(u^n-\lambda)'} \cdot \trans{\lambda}(p_1,\dots,p_n ;\beta) \, .
\end{equation}
\end{theorem}

Inspired by computations we carried out for related models in~\cite{BZ19a,Bis18}, we will prove~\eqref{eq:doubleSymLPP} by reformulating the problem in terms of symmetric point-to-line paths and applying the $\RSK$ correspondence on symmetric triangular arrays.
We recall from Subsection~\ref{subsec:CBtransition} that the $\CB$-interpolating Schur polynomial in~\eqref{eq:doubleSymLPP} degenerates to either a $(2n)$-symplectic or a $(2n+1)$-orthogonal Schur polynomial in the cases $\beta=0$ and $\beta=1$, respectively.
For $\beta=0$, we also have further expressions in terms of standard Schur polynomials, as next theorem states.
\begin{theorem}
\label{thm:doubleSymLPP_beta=0}
For all $u\in\Z_{\geq 0}$, the following four quantities are equal:
\begin{align*}
        A^{\doubleSymSq}_{\alpha,0} &:= c^{\doubleSymSq}_{\alpha,0} \cdot \P\left(L^{\doubleSymSq}_{\alpha,0}(2n,2n) \leq 2u\right) , \\
        B^{\doubleSymSq}_{\alpha,0} &:= \left[\prod_{i=1}^n p_i\right]^u
\sum_{\lambda \subseteq u^{n}}
\alpha^{\oddrows(u^n-\lambda)'} \cdot
\sp^{(2n)}_{\lambda}(p_1,\dots,p_n) \, , \\
        C^{\doubleSymSq}_{\alpha,0} &:= \left[\prod_{i=1}^n p_i\right]^u \schur^{(2n+1)}_{(u^{n} , 0^{n+1})}(p_1^{-1},\dots,p_n^{-1},p_1,\dots,p_n,\alpha) \, , \\
        D^{\doubleSymSq}_{\alpha,0} &:= \sum_{\mu \subseteq u^{n}} \schur^{(n)}_{\mu}(p_1,\dots,p_n) \cdot \schur^{(n+1)}_{(\mu,0)}(p_1,\dots,p_n,\alpha) \, .
\end{align*}
\end{theorem}

Identities $A^{\doubleSymSq}_{\alpha,0} = B^{\doubleSymSq}_{\alpha,0}$ and $A^{\doubleSymSq}_{\alpha,0} = D^{\doubleSymSq}_{\alpha,0}$ relate directly to the LPP model.
The first one is just~\eqref{eq:doubleSymLPP} for $\beta=0$.
The second one comes from a parallel approach to study the distribution of $L^{\doubleSymSq}_{\alpha, \beta}(2n,2n)$ adopted in~\cite{BR01a}, consisting in applying the classical $\RSK$ correspondence on a doubly symmetric square matrix.
Baik-Rains' formula, valid for general $\beta$, is in terms of ``self-dual'' Schur polynomials\footnote{As seen in Subsection~\ref{subsec:schur}, standard Schur polynomials can be seen as generating functions of Young tableaux of a certain shape.
On the other hand, ``self-dual'' Schur polynomials are generating functions of the only Young tableaux that are self-dual with respect to a combinatorial bijection known as \emph{Sch\"utzenberger involution}.}, and is omitted here for the sake of conciseness.
For $\beta=0$ the weights on the antidiagonal are all even: Forrester and Rains~\cite{FR07} thus deduced $A^{\doubleSymSq}_{\alpha,0} = D^{\doubleSymSq}_{\alpha,0}$, which instead involves standard Schur polynomials, using a bijection between self-dual Young tableaux with even rows, domino tableaux, and pairs of semi-standard Young tableux.

In Section~\ref{sec:nearlyRectangular} we will prove $B^{\doubleSymSq}_{\alpha,0} = C^{\doubleSymSq}_{\alpha,0}$ using  an identity of Krattenthaler~\cite{Kra98} for a standard Schur polynomial of ``nearly rectangular'' shape as well as the branching rule for general linear characters.
The specializations to $\alpha=0$ and $\alpha=1$ are known~\cite[Theorem~2.6]{Oka98} and read as:
\begin{align*}
\schur^{(2n)}_{(u^{n} , 0^{n})}(p_1^{-1}, \dots, p_n^{-1},
p_1, \dots, p_n)
&=
\begin{cases}
\displaystyle
\sum\limits_{\substack{\lambda \subseteq u^{n}, \\ \oddrows\lambda'=0}}
\sp^{(2n)}_{\lambda}(p_1, \dots, p_n) &\text{if $n$ is even,} \\ 
\displaystyle
\sum\limits_{\substack{\lambda \subseteq u^{n-1}, \\ \oddrows\lambda'=0}}
\sp^{(2n)}_{(u,\lambda)}(p_1, \dots, p_n) &\text{if $n$ is odd,}
\end{cases} \\
\schur^{(2n+1)}_{(u^{n} , 0^{n+1})}(p_1^{-1}, \dots, p_n^{-1},
p_1, \dots, p_n,1)
&= \sum_{\lambda \subseteq u^{n}}
\sp^{(2n)}_{\lambda}(p_1, \dots, p_n) \, .
\end{align*}
Therefore, our result unifies these special cases.
Notice that a $(q,t)$-Macdonald version of the first identity has recently appeared in~\cite[Corollary~1.3]{LRW20}; it would be interesting to investigate an $\alpha$-deformation of the latter, as in $B^{\doubleSymSq}_{\alpha,0} = C^{\doubleSymSq}_{\alpha,0}$.

Finally, a more general version of $C^{\doubleSymSq}_{\alpha,0} = D^{\doubleSymSq}_{\alpha,0}$ has been proved by Okada~\cite[Theorem 2.1]{Oka98} using determinantal calculus.
We will prove a further generalization of this formula, as the next theorem states, via decomposition of Gelfand-Tsetlin patterns.

\begin{theorem}
\label{thm:decompositionA}
For $n,m\geq 1$, $l:=\min(n,m)$ and $u\geq v\geq 0$, we have
\begin{align}
\label{eq:decompositionA}
\schur^{(n+m)}_{(u^{n} , v^{m})}(x,y)
&= \left[\prod_{i=1}^n x_i\right]^{u}
\left[\prod_{i=1}^m y_i\right]^{v}
\sum_{\mu \subseteq (u-v)^{(l)}}
\schur^{(n)}_{(\mu, 0^{n-l})} (x^{-1}) \cdot
\schur^{(m)}_{(\mu, 0^{m-l})} (y) \, ,
\end{align}
where $x=(x_1, \dots, x_n)$ and $y=(y_1, \dots, y_m)$.
\end{theorem}

We will see in the proof of the latter theorem (see in particular~\eqref{eq:s=sumSS_Okada}) that Okada's formula is a specialization of~\eqref{eq:decompositionA} valid for $n\geq m$ and $v=0$.
Notice also that, by putting $m=n+1$, $v=0$, $x_i^{-1}=y_i=p_i$ for $1\leq i\leq n$, and $y_{n+1}=\alpha$ in~\eqref{eq:decompositionA}, we indeed recover $C^{\doubleSymSq}_{\alpha,0} = D^{\doubleSymSq}_{\alpha,0}$.
Besides Gelfand-Tsetlin pattern decomposition and determinantal calculus, a third (algebraic) method to approach Theorem~\ref{thm:decompositionA} relies on skew Schur functions; we will outline this in Subsection~\ref{subsec:decompositionA}.

For the sake of simplicity, throughout this subsection we have restricted ourselves to weight matrices $W$ of even order $2n$, and therefore to ``even'' LLP times (i.e.\ for directed paths from $(1,1)$ to $(2n,2n)$).
However, analogously to $A^{\doubleSymSq}_{\alpha,0} = B^{\doubleSymSq}_{\alpha,0}$, one can show that the ``odd'' LPP time with doubly symmetric weights can be expressed as a bounded Littlewood sum of odd symplectic characters.

\section{$\RSK$ on triangular arrays and last passage percolation}
\label{sec:RSK}

In this section we prove the identities that express the distribution functions of our LPP models in terms of formulas involving $\CB$-interpolating Schur polynomials.
In particular, we prove $A^{\antisymSq}_\beta = C^{\antisymSq}_\beta$ in Theorem~\ref{thm:antisymLPP} (which specializes to $A^{\antisymSq}_0 = C^{\antisymSq}_0$ of Corollary~\ref{coro:antisymLPP_0} and $A^{\antisymSq}_1 = C^{\antisymSq}_1$ of Corollary~\ref{coro:antisymLPP_1} in the cases $\beta=0$ and $\beta=1$), as well as~\eqref{eq:doubleSymLPP} in Theorem~\ref{thm:doubleSymLPP} (which specializes to $A^{\doubleSymSq}_{\alpha,0} = B^{\doubleSymSq}_{\alpha,0}$ of Theorem~\ref{thm:doubleSymLPP_beta=0} in the case $\beta=0$).
The proofs we present are all based on the Robinson-Schensted-Knuth ($\RSK$)
correspondence applied to triangular arrays.
As described by Knuth~\cite{Knu70}, the $\RSK$ is a combinatorial bijection that maps matrices with non-negative integer entries to pairs of semi-standard Young tableaux, or equivalently pairs of Gelfand-Tsetlin patterns, of the same shape.
Via Fomin's growth diagrams~\cite{Fom95, Kra06}, it can be generalized to a bijection mapping a Young tableau of a given (not necessarily rectangular) shape to another Young tableau of the same shape.
If the entries of the input tableau are non-negative reals, then the output tableau has the constraint that the entry of the box $(i,j)$, denoted by $t_{i,j}$, obeys the interlacing constraints $t_{i,j}\geq \max( t_{i-1,j}, t_{i,j-1})$ for all sensible pairs $(i,j)$.
Restricting this bijection to square Young tableaux with non-negative integer entries yields the classical $\RSK$ correspondence: when the input tableau is a square matrix, the output tableau is a square matrix of the same dimensions and corresponds to the pair of Gelfand-Tsetlin patterns ``glued together'' along the common shape, which is the main diagonal of the output matrix.
For the sake of brevity we will not give the explicit construction of $\RSK$, but we will rather recall the properties that we need.

We first introduce some notation.
Let $\mathcal{I}$ be the index set of a Young diagram\footnote{Namely, $\mathcal{I}$ is a finite subset of $\N^2$ satisfying: if $(i,j)\in\mathcal{I}$, then $(i-1,j)\in\mathcal{I}$ if $i>1$, and $(i,j-1)\in\mathcal{I}$ if $j>1$.}, and let $t = \{t_{i,j} \colon (i,j)\in \mathcal{I}\}$ be a Young tableau with non-negative real entries.
We denote by
\begin{equation}
\label{eq:sumDiagonal}
\sigma_k(t) := \sum_{\substack{(i,j)\in\mathcal{I} \colon \\ j-i=k}} t_{i,j}
\end{equation}
the sum of the $k$-th diagonal of $t$.
We call $(i,j)$ {\bf\emph{outer index}} of $\mathcal{I}$ if none of the three sites $(i,j+1),(i+1,j),(i+1,j+1)$ belongs to $\mathcal{I}$, or equivalently if $\mathcal{I}\setminus\{(i,j)\}$ is still the index set of a Young diagram; we call $(i,j)$ {\bf\emph{inner index}} otherwise.
We say that $(i,j)\in\mathcal{I}$ is a {\bf\emph{border index}} if $(i+1,j+1)\notin \mathcal{I}$, or equivalently if it is the last index of its diagonal.
Clearly, every outer index is also a border index.
We call $t_{i,j}$ an {\bf\emph{outer/inner/border entry}} of $t$ if $(i,j)$ is an outer/inner/border index of $\mathcal{I}$, respectively.

For our purposes, it is essential to restrict the $\RSK$ bijection to input Young tableaux with inner and outer entries drawn from two different subsets $\cA$ and $\cB$ of $\R_{\geq 0}$ respectively.
In order for the image tableau to be of the same type (i.e., with inner and outer entries drawn from the same sets $\cA$ and $\cB$, respectively), we need $\cA\subseteq \cB$ to be submonoids of the additive monoid $\R_{\geq 0}$.
In other words, both $\cA$ and $\cB$ must be closed under addition and contain $0$.
In our proofs, we will need either $\cA=\cB=\Z_{\geq 0}$ (all entries are integers) or $\cA=\Z_{\geq 0}$ and $\cB=\frac{1}{2}\Z_{\geq 0}$ (the outer entries are also allowed to be half-integers).
The next proposition, whose proof is omitted, is a fairly straightforward adaptation of a few classical properties of $\RSK$ to this setting.
More details on this algorithm and its various properties and extensions can be found e.g.\ in~\cite{Knu70, Kir01, Bis18, Zyg18}.

\begin{proposition}
\label{prop:RSK}
Let $\mathcal{I}$ be the index set of a Young diagram and let $\cA\subseteq \cB \subseteq \R_{\geq 0}$ be submonoids of the additive monoid $\R_{\geq 0}$.
There exists a piecewise linear bijection
\begin{equation}
\label{eq:RSKmapping}
w = \{w_{i,j} \colon (i,j)\in \mathcal{I}\}
\mapsto
t = \{t_{i,j} \colon (i,j)\in \mathcal{I}\} \, ,
\end{equation}
called $\RSK$ correspondence, between Young tableaux with inner entries in $\cA$ and outer entries in $\cB$, such that the output tableau satisfies the ordering
\begin{equation}
\label{eq:RSKordering}
t_{i-1,j} \leq t_{i,j} \quad\text{if } i>1
\qquad\quad \text{and} \qquad\quad
t_{i,j-1} \leq t_{i,j} \quad\text{if } j>1 \, ,
\qquad
\text{for } (i,j)\in\mathcal{I} \, .
\end{equation}
The $\RSK$ satisfies the following properties:
\begin{enumerate}
\item
\label{prop:RSK_LPP}
For any border index $(m,n)\in \mathcal{I}$, denoting by $\Pi_{m,n}$ the set of all directed paths from $(1,1)$ to $(m,n)$, we have
\begin{equation}
\label{eq:RSK_LPP}
t_{m,n}
= \max_{\pi\in\Pi_{m,n}} \sum_{(i,j)\in \pi} w_{i,j} \, .
\end{equation}
\item
\label{prop:RSK_type}
For any border index $(m,n)\in \mathcal{I}$, we have
\begin{equation}
\label{eq:RSK_type}
\sigma_{n-m}(t)
= \sum_{i=1}^m \sum_{j=1}^n w_{i,j} \, .
\end{equation}
\item
\label{prop:RSK_symmetry}
The symmetry of tableaux about the diagonal is preserved by the $\RSK$ correspondence.
Namely, assume that $\mathcal{I}$ is a symmetric index set, i.e.\ $(i,j)\in \mathcal{I}$ if and only if $(j,i)\in\mathcal{I}$, and $w$ is a symmetric tableau, i.e.\ $w_{i,j}=w_{j,i}$ for all $(i,j)$; then $t$ is also a symmetric tableau.
In this case, denoting by $(n,n)$ the border index of the main diagonal, we also have that
\begin{equation}
\label{eq:RSK_symmetry_diagSum}
\sum_{j=1}^n w_{j,j} = \sum_{j=1}^n (-1)^{n-j} t_{j,j} \, .
\end{equation}
\end{enumerate}
\end{proposition}

Property~\ref{prop:RSK_LPP} is a particular case of Greene's Theorem~\cite{Gre74} and states that the last entry of any diagonal of the $\RSK$ output tableau can be interpreted as a point-to-point LPP time on the input variables.
Property~\ref{prop:RSK_type} relates the sum of a diagonal of the $\RSK$ output to the sum of certain rectangular subarrays of the input array.
Property~\ref{prop:RSK_symmetry} is useful for dealing with symmetric input tableaux.

In the following proofs we will consider the $\RSK$ on triangular Young tableaux of shape $(N,N-1,\dots,1)$, or equivalently indexed by $\mathcal{I} := \{(i,j)\in\N^2\colon i+j\leq N+1\}$.
In such a case, the outer indices are $\{(i,j)\in\N^2\colon i+j= N+1\}$, whereas the border indices are $\{(i,j)\in\N^2\colon N \leq i+j \leq N+1\}$.

\subsection{Proof of $A^{\antisymSqProtect}_\beta = C^{\antisymSqProtect}_\beta$ in Theorem~\ref{thm:antisymLPP}}
\label{subsec:RSK_antisymLPP}

We first see the connection between the point-to-point LPP with symmetry about the antidiagonal and the point-to-line LPP, defined in~\eqref{eq:pointToLineLPP} below.
Let us consider the point-to-point LPP $L^{\antisymSq}(N,N)$ associated with a square weight array symmetric about the antidiagonal $i+j=N+1$ (let us not specify the distribution of the weights for the moment).
Because of the symmetry constraint, at least one of the maximal paths\footnote{By maximal path we mean any of the allowed paths that maximize the passage time, see~\eqref{eq:LPP}. Notice that such a path does not need to be unique.} from $(1,1)$ to $(N,N)$ is symmetric about the antidiagonal; the weights collected along such a path will be all counted twice (once above and once below the antidiagonal), except the one on the antidiagonal itself.
Let us now consider the \emph{point-to-line} LPP time $L^{\pointToLine}(N)$ for directed paths starting at $(1,1)$ and ending at any point of the antidiagonal line $i+j=N+1$, which can be expressed in terms of point-to-point LPP times as
\begin{equation}
\label{eq:pointToLineLPP}
L^{\pointToLine}(N)
= \max_{i+j=N+1} L(i,j) \, .
\end{equation}
We then have that the $\antisymSq$-symmetric LPP coincides with twice the point-to-line LPP, i.e.\ $L^{\antisymSq}(N,N) = 2 \cdot L^{\pointToLine}(N)$, assuming that \emph{the weights of the point-to-line LPP are halved on the antidiagonal}.

Let us now consider the point-to-point last passage time $L^{\antisymSq}_{\beta}(N,N)$ with weights distributed as in~\eqref{eq:antisymWeights}.
Then, the corresponding point-to-line LPP $L^{\pointToLine}_\beta(N)$ is taken on a triangular array of independent weights $W=\{W_{i,j}\colon i+j\leq N+1\}$ distributed as
\begin{equation}
\label{eq:antisymWeightsPointToLine}
\P(W_{i,j} = k) =
\begin{cases}
(1-p_{N-i+1} p_j) (p_{N-i+1} p_j)^k &\text{for } k\in\Z_{\geq 0}, \, i+j< N+1 \, , \\
\dfrac{1-p_j^2}{1 + \beta p_j} \beta^{\1_{\frac{1}{2}+\Z}(k)} p_j^{2k} &\text{for } k\in\frac{1}{2} \Z_{\geq 0}, \, i+j = N+1 \, ,
\end{cases}
\end{equation}
where the indicator function $\1_{\frac{1}{2}+\Z}(k)$ gives $1$ if $k$ is a half-integer and $0$ if $k$ is an integer.
Given the identity $L^{\antisymSq}_\beta(N,N) = 2 \cdot L^{\pointToLine}_\beta(N)$, we are thus reduced to prove the point-to-line reformulation of $A^{\antisymSq}_\beta = C^{\antisymSq}_\beta$ in Theorem~\ref{thm:antisymLPP}:
\begin{equation}
\label{eq:antisymLPP_pointToLine}
c^{\antisymSq}_{\beta} \cdot \P\left(L^{\pointToLine}_\beta(N) \leq u\right)
= \left[ \prod_{i=1}^N p_i \right]^u \trans{u^{N}} (p_1,\dots,p_N;\beta)
\end{equation}
for $u\in\frac{1}{2}\Z_{\geq 0}$, where $L^{\pointToLine}_\beta(N)$ is taken on the modified weights~\eqref{eq:antisymWeightsPointToLine}.

To prove the latter, we first rewrite the joint distribution of $W$ under~\eqref{eq:antisymWeightsPointToLine}:
\begin{align*}
\P(W=w)
&= \prod_{i+j< N+1} (1- p_{N-i+1} p_j) (p_{N-i+1} p_j)^{w_{i,j}}
\prod_{i+j=N+1} \frac{1-p_j^2}{1 + \beta p_j} \beta^{\1_{\frac{1}{2}+\Z}(w_{i,j})} p_j^{2w_{i,j}} \\
&= \frac{1}{c^{\antisymSq}_{\beta}}
\beta^{\#\{j\colon w_{N-j+1,j} \in \frac{1}{2}+\Z\}}
\prod_{i=1}^N
p_i^{\sum_{j=1}^i w_{N-i+1,j}}
\prod_{j=1}^N p_j^{\sum_{i=1}^{N-j+1} w_{i,j}}
\end{align*}
for any triangular tableau $w=\{w_{i,j}\colon i+j\leq N+1\}$ such that $w_{i,j}\in\Z_{\geq 0}$ for $i+j< N+1$ and $w_{i,j}\in\frac{1}{2}\Z_{\geq 0}$ for $i+j= N+1$.
Taking $\cA=\Z_{\geq 0}$ and $\cB=\frac{1}{2}\Z_{\geq 0}$ in Proposition~\ref{prop:RSK}, we can consider $\RSK$ as a bijection $w\mapsto t$ between Young tableaux of shape $(N,N-1,\dots,1)$ with inner entries in $\Z_{\geq 0}$ and outer entries in $\frac{1}{2}\Z_{\geq 0}$, such that the output tableau satisfies ordering~\eqref{eq:RSKordering}.
By property~\ref{prop:RSK_LPP} of the proposition, each entry $t_{i,j}$ of the $\RSK$ output tableau such that $i+j=N+1$ is the point-to-point last passage time on $w$ from $(1,1)$ to $(i,j)$.
On the other hand, all inner $w_{i,j}$'s are integers, whereas the outer ones might also be half-integers; therefore, for each $(i,j)$ with $i+j=N+1$, $t_{i,j}$ is either integer, if $w_{i,j}$ is integer, or half-integer, if $w_{i,j}$ is half-integer.
This, in particular, implies that
\[
\#\left\{j\colon w_{N-j+1,j} \in \frac{1}{2}+\Z\right\}
= \#\left\{j\colon t_{N-j+1,j} \in \frac{1}{2}+\Z\right\}
\]
under the bijection $w\mapsto t$.
Using the latter identity and property~\ref{prop:RSK_type} of Proposition~\ref{prop:RSK}, we can write the distribution that $W$ induces on its $\RSK$ image $T$ as
\[
\P(T = t)
= \frac{1}{c^{\antisymSq}_{\beta}}
\beta^{\#\{j\colon t_{2n-j+1,j} \in \frac{1}{2}+\Z\}}
\prod_{i=1}^N
p_i^{\sigma_{-N+2i-1}(t)-\sigma_{-N+2i}(t)} \prod_{j=1}^N
p_j^{\sigma_{-N+2j-1}(t) - \sigma_{-N+2j-2}(t)}
\]
for all tableaux $t=\{t_{i,j}\colon i+j\leq N+1\}$ satisfying~\eqref{eq:RSKordering} and such that $t_{i,j}\in\Z_{\geq 0}$ for $i+j< N+1$ and $t_{i,j}\in\frac{1}{2}\Z_{\geq 0}$ for $i+j= N+1$.

On the other hand, it follows from~\eqref{eq:pointToLineLPP}, \eqref{eq:RSK_LPP} and~\eqref{eq:RSKordering} that the distribution function of $L^{\pointToLine}_\beta(N)$ is given by
\[
\P\left(L^{\pointToLine}_\beta(N) \leq u\right)
= \P\left(\max_{i+j=N+1} T_{i,j} \leq u \right)
= \sum_{t\colon \max t_{i,j}\leq u} \P(T=t) \, .
\]

\begin{figure}
\centering
\begin{tikzpicture}[scale=1.05, every path/.style={draw=none}, every node/.style={fill=white, scale=1, inner sep=1.5pt}, ineqE/.style={node contents={\tiny $\leq$}}, ineqW/.style={node contents={\tiny $\geq$}}, ineqS/.style={rotate=-90, node contents={\tiny $\leq$}}, ineqN/.style={rotate=-90, node contents={\tiny $\geq$}}, ineqSE/.style={rotate=-45, node contents={\tiny \blue{$\leq$}}}, ineqNW/.style={rotate=-45, node contents={\tiny \blue{$\geq$}}}]

\node (t11) at (1,-1) {$t_{1,1}$};
\node (t12) at (2,-1) {$t_{1,2}$};
\node (t13) at (3,-1) {$t_{1,3}$};
\node (t14) at (4,-1) {$t_{1,4}$};

\node (t21) at (1,-2) {$t_{2,1}$};
\node (t22) at (2,-2) {$t_{2,2}$};
\node (t23) at (3,-2) {$t_{2,3}$};

\node (t31) at (1,-3) {$t_{3,1}$};
\node (t32) at (2,-3) {$t_{3,2}$};

\node (t41) at (1,-4) {$t_{4,1}$};

\node[text=blue] (o) at (0.25,-0.25) {$0$};
\path (o) -- node[ineqSE]{} (t11);

\node[text=blue] (u1) at (5-0.25,-2+0.25) {$u$};
\node[text=blue] (u2) at (4-0.25,-3+0.25) {$u$};
\node[text=blue] (u3) at (3-0.25,-4+0.25) {$u$};
\node[text=blue] (u4) at (2-0.25,-5+0.25) {$u$};

\path (t14) -- node[ineqSE]{} (u1);
\path (t23) -- node[ineqSE]{} (u2);
\path (t32) -- node[ineqSE]{} (u3);
\path (t41) -- node[ineqSE]{} (u4);

\path (t11) -- node[ineqE]{} (t12) -- node[ineqE]{} (t13) -- node[ineqE]{} (t14);

\path (t21) -- node[ineqE]{} (t22) -- node[ineqE]{} (t23);

\path (t31) -- node[ineqE]{} (t32);

\path (t11) -- node[ineqS]{} (t21) -- node[ineqS]{} (t31) -- node[ineqS]{} (t41);

\path (t12) -- node[ineqS]{} (t22) -- node[ineqS]{} (t32);

\path (t13) -- node[ineqS]{} (t23);

\node at (7,-1.5) {\Large $\longmapsto$};

\begin{scope}[shift={(8,0)}]

\node (t11) at (1,-1) {$Z_{4,1}$};
\node (t12) at (2,-1) {$Z_{5,2}$};
\node (t13) at (3,-1) {$Z_{6,3}$};
\node (t14) at (4,-1) {$Z_{7,4}$};

\node (t21) at (1,-2) {$Z_{3,1}$};
\node (t22) at (2,-2) {$Z_{4,2}$};
\node (t23) at (3,-2) {$Z_{5,3}$};

\node (t31) at (1,-3) {$Z_{2,1}$};
\node (t32) at (2,-3) {$Z_{3,2}$};

\node (t41) at (1,-4) {$Z_{1,1}$};

\node[text=blue] (o1) at (5-0.25,-2+0.25) {$0$};
\node[text=blue] (o2) at (4-0.25,-3+0.25) {$0$};
\node[text=blue] (o3) at (3-0.25,-4+0.25) {$0$};
\node[text=blue] (o4) at (2-0.25,-5+0.25) {$0$};

\path (t14) -- node[ineqNW]{} (o1);
\path (t23) -- node[ineqNW]{} (o2);
\path (t32) -- node[ineqNW]{} (o3);
\path (t41) -- node[ineqNW]{} (o4);

\path (t11) -- node[ineqW]{} (t12) -- node[ineqW]{} (t13) -- node[ineqW]{} (t14);

\path (t21) -- node[ineqW]{} (t22) -- node[ineqW]{} (t23);

\path (t31) -- node[ineqW]{} (t32);

\path (t11) -- node[ineqN]{} (t21) -- node[ineqN]{} (t31) -- node[ineqN]{} (t41);

\path (t12) -- node[ineqN]{} (t22) -- node[ineqN]{} (t32);

\path (t13) -- node[ineqN]{} (t23);

\node[text=red] (s11) at (1,0) {$\bm{u}$};
\node[text=red] (s12) at (2,0) {$\bm{u}$};
\node[text=red] (s13) at (3,0) {$\bm{u}$};
\node[text=red] (s14) at (4,0) {$\bm{u}$};

\node[text=red] (s21) at (1,1) {$\bm{u}$};
\node[text=red] (s22) at (2,1) {$\bm{u}$};
\node[text=red] (s23) at (3,1) {$\bm{u}$};

\node[text=red] (s31) at (1,2) {$\bm{u}$};
\node[text=red] (s32) at (2,2) {$\bm{u}$};

\node[text=red] (s41) at (1,3) {$\bm{u}$};

\path (s11) -- node[ineqW]{} (s12) -- node[ineqW]{} (s13) -- node[ineqW]{} (s14);

\path (s21) -- node[ineqW]{} (s22) -- node[ineqW]{} (s23);

\path (s31) -- node[ineqW]{} (s32);

\path (t11) -- node[ineqN]{} (s11) -- node[ineqN]{} (s21) -- node[ineqN]{} (s31) -- node[ineqN]{} (s41);

\path (t12) -- node[ineqN]{} (s12) -- node[ineqN]{} (s22) -- node[ineqN]{} (s32);

\path (t13) -- node[ineqN]{} (s13) -- node[ineqN]{} (s23);

\path (t14) -- node[ineqN]{} (s14);

\end{scope}
\end{tikzpicture}
\caption{A pictorial representation of transformation~\eqref{eq:changeVariables} for $N=4$.
Array $t$, on the left-hand side, satisfies an ordering that the transformation reverses.
Moreover, on the right-hand side, extra entries equal to $u$ (in red, bold) are introduced, to form a split orthogonal pattern $Z$ of height $2N$ and shape $u^{N}$ (the picture should be rotated by $135$ degrees clockwise to visualize the pattern as in Figure~\ref{fig:soGTpatterns}).
In blue, we have illustrated the lower bound $0$ and upper bound $u$ of all entries of the arrays.}
\label{fig:changeVariables}
\end{figure}

Out of the array $t$, we now define a new array $Z = (Z_{i,j})_{1\leq i\leq 2N, \, 1\leq j\leq \lceil i/2 \rceil}$ by setting
\begin{equation}
\label{eq:changeVariables}
Z_{i,j} :=
\begin{cases}
u - t_{N+j-i,j} &\text{if } 0\leq i-j\leq N-1 \, , \\
u &\text{if } N\leq i-j\leq 2N-1 \, .
\end{cases}
\end{equation}
This transformation amounts to a change of variables plus an artificial definition of new fixed entries equal to $u$.
From its pictorial representation given in Figure~\ref{fig:changeVariables}, one can visualize the following facts:
\begin{itemize}
\item 
all the entries $Z_{i,j}$'s are bounded between $0$ and $u$, as the $t_{i,j}$'s are;
\item $Z$ satisfies the interlacing conditions~\eqref{eq:interlacing}, due to the ordering of the $t_{i,j}$'s;
\item each of the entries $Z_{2i-1,i}$, for $1\leq i\leq N$, runs in $\frac{1}{2}\Z_{\geq 0}$, as the outer $t_{i,j}$'s do;
\item the other $Z_{i,j}$'s are either all simultaneously in $\Z_{\geq 0}$, when $u$ is an integer, or all simultaneously in $\frac{1}{2} + \Z_{\geq 0}$, when $u$ is a half-integer.
\end{itemize}
By definition, array $Z$ is then a split orthogonal pattern of height $2N$ and shape $u^{N}$, whose \emph{atypical} entries are in bijection with the outer half-integer entries of $t$.
Conversely, every split orthogonal pattern of height $2N$ and shape $u^{N}$ can be constructed in such a way, starting from an array $t$ with the features described above.
Denoting by $\abs{Z_i} :=\sum_{j=1}^{\ceil{i/2}} Z_{i,j}$ the sum of the $i$-th row of $Z$, it is easy to see that
\[
\abs{Z_k} = \ceil{\frac{k}{2}} u - \sigma_{-N+k}(t) \qquad\qquad
\text{for } 0\leq k\leq 2N \, ,
\]
with the convention that $\sigma_{-N}(t) = \sigma_{N}(t):= 0$.
Recalling the notation $\abs{\atyp(Z)}$ for the number of atypical entries $Z$, we then obtain:
\begin{align*}
c^{\antisymSq}_{\beta} \cdot \P\left(L^{\pointToLine}_\beta(N) \leq u\right)
&= \sum_{Z \in \soGT^{(2N)}_{u^{N}}}
\beta^{|\atyp(Z)|}
\prod_{i=1}^N
p_i^{-\abs{Z_{2i-1}} + \abs{Z_{2i}}}
\prod_{j=1}^N
p_j^{u -\abs{Z_{2j-1}} +\abs{Z_{2j-2}}} \\
&= \prod_{j=1}^N p_j^u
\sum_{Z\in \soGT^{(2N)}_{u^{N}}}
\beta^{|\atyp(Z)|}
\prod_{i=1}^N p_i^{\type(Z)_{2i}-\type(Z)_{2i-1}} \, .
\end{align*}
We recognize the latter sum over $\soGT^{(2N)}_{u^{N}}$ to be the $\CB$-interpolating Schur polynomial appearing on the right-hand side of~\eqref{eq:antisymLPP_pointToLine}, thus proving the desired identity.

\begin{remark}
Setting $\beta=0$, we obtain an LPP formula involving symplectic characters, i.e.\ $A^{\antisymSq}_0 = C^{\antisymSq}_0$ of Corollary~\ref{coro:antisymLPP_0}.
We remark that the latter may be proved via a more direct and \emph{ad hoc} argument: let us quickly sketch this.
For $\beta=0$, all the weights~\eqref{eq:antisymWeightsPointToLine} of the corresponding point-to-line model are integers.
Following the outline of our proof above, one may then apply the $\RSK$ bijection of Proposition~\ref{prop:RSK} with $\cA = \cB = \Z_{\geq 0}$.
The subsequent transformation defines a split orthogonal pattern where all entries are integers, i.e.\ a symplectic pattern, which will generate the symplectic character appearing in $C^{\antisymSq}_0$.
\end{remark}

\subsection{Proof of Theorem~\ref{thm:doubleSymLPP}}
\label{subsec:RSK_doubleSymLPP}

Let us consider the point-to-point LPP $L^{\doubleSymSq}(N,N)$ associated to a square weight array symmetric about both the antidiagonal $i+j=N+1$ and the diagonal $i=j$, without specifying the distribution of the weights for the moment.
Let us also consider the point-to-line LPP time $L^{\pointToLineSym}(N)$ from point $(1,1)$ to the antidiagonal line $i+j=N+1$, taken on a triangular weight array symmetric about the diagonal $i=j$.
As in Subsection~\ref{subsec:RSK_antisymLPP}, we then have that $L^{\doubleSymSq}(N,N) = 2 \cdot L^{\pointToLineSym}(N)$, assuming that \emph{the weights of the point-to-line LPP are halved on the antidiagonal}.

In particular, if we consider $L^{\doubleSymSq}_{\alpha,\beta}(2n,2n)$ with the weights distributed as in~\eqref{eq:doubleSymWeights}, then the corresponding point-to-line LPP $L^{\pointToLineSym}_{\alpha,\beta}(2n)$ is taken on a triangular array of independent weights $W=\{W_{i,j}\colon i+j\leq 2n+1\}$ such that $W_{i,j}=W_{j,i}$ for all $i>j$ and
\begin{equation}
\label{eq:doubleSymWeights_pointToLine}
\P(W_{i,j} = k) =
\begin{cases}
(1-p_i p_j) (p_i p_j)^k &\text{for} \,\, k\in\Z_{\geq 0} \, , \,\, i<j<2n-i+1 \, , \\
\dfrac{1-p_j^2}{1 + \beta p_j} \beta^{\1_{\frac{1}{2}+\Z}(k)} p_j^{2k} &\text{for} \,\, k\in\frac{1}{2}\Z_{\geq 0} \, , \,\, i<j=2n-i+1 \, , \\
(1-\alpha p_j) (\alpha p_j)^k &\text{for} \,\, k\in\Z_{\geq 0} \, , \,\, 1\leq i=j\leq n \, ,
\end{cases}
\end{equation}
with $p_{2n-i+1} = p_i$ for all $1\leq i\leq n$.
Recalling~\eqref{eq:oddPartsConjugate}, and given the identity $L^{\doubleSymSq}_{\alpha,\beta}(2n,2n) = 2 \cdot L^{\pointToLineSym}_{\alpha,\beta}(2n)$, we are thus reduced to prove the point-to-line version of Theorem~\ref{thm:doubleSymLPP}:
\begin{equation}
\label{eq:doubleSymLPP_pointToLine}
c^{\doubleSymSq}_{\alpha,\beta} \cdot \P\left(L^{\pointToLineSym}_{\alpha, \beta}(2n) \leq u\right)
= \left[ \prod_{i=1}^n p_i \right]^u\sum_{\lambda \subseteq u^{n}} \alpha^{\sum_{i=1}^n (-1)^{n-i}(u-\lambda_i)} \cdot \trans{\lambda}(p_1,\dots,p_n ;\beta)
\end{equation}
for $u\in\frac{1}{2}\Z_{\geq 0}$, where $L^{\pointToLineSym}_{\alpha,\beta}(2n)$ is taken on the modified weight distribution~\eqref{eq:doubleSymWeights_pointToLine}.

To prove the latter we first observe that, if $W$ is distributed as in~\eqref{eq:doubleSymWeights_pointToLine}, then
\begin{align*}
&\P(W=w) \\
= \, &\prod_{i<j< 2n-i+1} \!\!\! (1-p_i p_j) (p_i p_j)^{w_{i,j}}
\prod_{i<j = 2n-i+1} \frac{1-p_j^2}{1 + \beta p_j} \beta^{\1_{\frac{1}{2}+\Z}(w_{i,j})} p_j^{2w_{i,j}}
\prod_{j=1}^n (1-\alpha p_j) (\alpha p_j)^{w_{j,j}} \\
= \, &\frac{1}{c^{\doubleSymSq}_{\alpha,\beta}}
\alpha^{\sum_{j=1}^n w_{j,j}} \cdot
\beta^{\#\{i<j=2n-i+1 \colon w_{i,j} \in \frac{1}{2}+\Z\}}
\prod_{i=1}^n p_i^{\sum_{j=1}^{2n-i+1} w_{i,j}}
\prod_{j=1}^n p_j^{\sum_{i=1}^j w_{i,2n-j+1}}
\end{align*}
for all symmetric triangular tableau $w$ of shape $(2n,2n-1,\dots,1)$ such that $w_{i,j}\in\Z_{\geq 0}$ for $i+j< 2n+1$ and $w_{i,j}\in\frac{1}{2}\Z_{\geq 0}$ for $i+j= 2n+1$.
Taking $\cA=\Z_{\geq 0}$ and $\cB=\frac{1}{2}\Z_{\geq 0}$ in Proposition~\ref{prop:RSK}, we now consider the $\RSK$ bijection $w\mapsto t$.
As $w$ is symmetric, so is $t$ by property~\ref{prop:RSK_symmetry}.
Via a similar argument as in Subsection~\ref{subsec:RSK_antisymLPP}, one can realize that, under this bijection,
\[
\#\left\{i<j=2n-i+1 \colon w_{i,j} \in \frac{1}{2}+\Z \right\}
= \#\left\{i<j=2n-i+1 \colon t_{i,j} \in \frac{1}{2}+\Z \right\} \, .
\]
Using the latter identity as well as properties~\ref{prop:RSK_type} and~\ref{prop:RSK_symmetry} of the proposition, we see that the distribution that $W$ induces on its $\RSK$ image $T$ is given by
\[
\begin{split}
\P(T=t)
= \, &\frac{1}{c^{\doubleSymSq}_{\alpha,\beta}}
\alpha^{\sum_{j=1}^n (-1)^{n-j} t_{j,j}} \cdot
\beta^{\#\{i<j=2n-i+1 \colon t_{i,j} \in \frac{1}{2}+\Z\}} \\
&\qquad\qquad \times \prod_{i=1}^n
p_i^{\sigma_{2n-2i+1}(t) - \sigma_{2n-2i+2}(t)}
\prod_{j=1}^n
p_j^{\sigma_{2n-2j+1}(t) - \sigma_{2n-2j}(t)}
\end{split}
\]
for all symmetric tableau $t$ of shape $(2n,2n-1,\dots,1)$, satisfying~\eqref{eq:RSKordering}, and such that $t_{i,j}\in\Z_{\geq 0}$ for $i+j< 2n+1$ and $t_{i,j}\in\frac{1}{2}\Z_{\geq 0}$ for $i+j= 2n+1$.

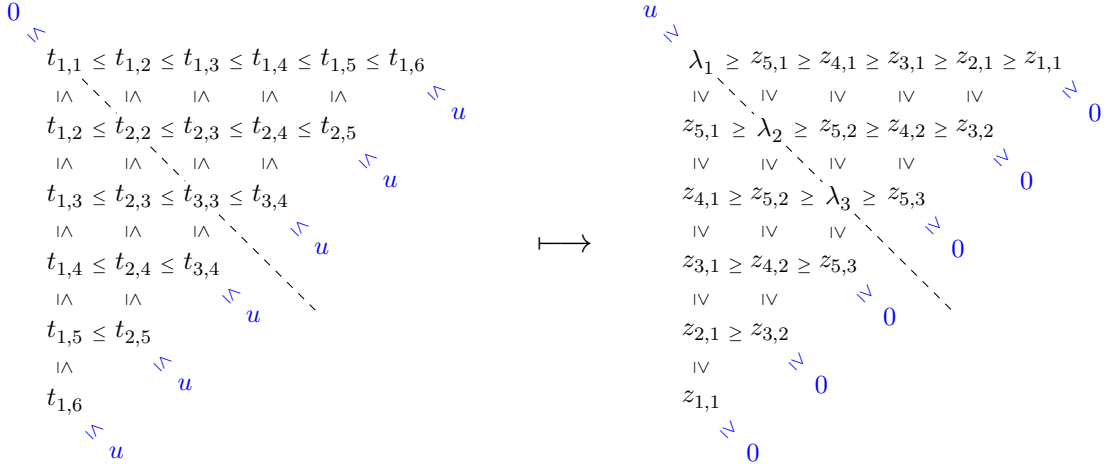
\begin{figure}
\centering
\begin{tikzpicture}[scale=0.9, every path/.style={draw=none}, every node/.style={scale=0.9, fill=white, scale=1, inner sep=1.5pt}, ineqE/.style={node contents={\tiny $\leq$}}, ineqW/.style={node contents={\tiny $\geq$}}, ineqS/.style={rotate=-90, node contents={\tiny $\leq$}}, ineqN/.style={rotate=-90, node contents={\tiny $\geq$}}, ineqSE/.style={rotate=-45, node contents={\tiny \blue{$\leq$}}}, ineqNW/.style={rotate=-45, node contents={\tiny \blue{$\geq$}}}]

\node (t11) at (1,-1) {$t_{1,1}$};
\node (t12) at (2,-1) {$t_{1,2}$};
\node (t13) at (3,-1) {$t_{1,3}$};
\node (t14) at (4,-1) {$t_{1,4}$};
\node (t15) at (5,-1) {$t_{1,5}$};
\node (t16) at (6,-1) {$t_{1,6}$};

\node (t21) at (1,-2) {$t_{1,2}$};
\node (t22) at (2,-2) {$t_{2,2}$};
\node (t23) at (3,-2) {$t_{2,3}$};
\node (t24) at (4,-2) {$t_{2,4}$};
\node (t25) at (5,-2) {$t_{2,5}$};

\node (t31) at (1,-3) {$t_{1,3}$};
\node (t32) at (2,-3) {$t_{2,3}$};
\node (t33) at (3,-3) {$t_{3,3}$};
\node (t34) at (4,-3) {$t_{3,4}$};

\node (t41) at (1,-4) {$t_{1,4}$};
\node (t42) at (2,-4) {$t_{2,4}$};
\node (t43) at (3,-4) {$t_{3,4}$};

\node (t51) at (1,-5) {$t_{1,5}$};
\node (t52) at (2,-5) {$t_{2,5}$};

\node (t61) at (1,-6) {$t_{1,6}$};

\draw[dashed] (t11) -- (t22) -- (t33) -- (4.7,-4.7);

\node[text=blue] (o) at (0.25,-0.25) {$0$};
\path (o) -- node[ineqSE]{} (t11);

\node[text=blue] (u1) at (7-0.25,-2+0.25) {$u$};
\node[text=blue] (u2) at (6-0.25,-3+0.25) {$u$};
\node[text=blue] (u3) at (5-0.25,-4+0.25) {$u$};
\node[text=blue] (u4) at (4-0.25,-5+0.25) {$u$};
\node[text=blue] (u5) at (3-0.25,-6+0.25) {$u$};
\node[text=blue] (u6) at (2-0.25,-7+0.25) {$u$};

\path (t16) -- node[ineqSE]{} (u1);
\path (t25) -- node[ineqSE]{} (u2);
\path (t34) -- node[ineqSE]{} (u3);
\path (t43) -- node[ineqSE]{} (u4);
\path (t52) -- node[ineqSE]{} (u5);
\path (t61) -- node[ineqSE]{} (u6);

\path (t11) -- node[ineqE]{} (t12) -- node[ineqE]{} (t13) -- node[ineqE]{} (t14) -- node[ineqE]{} (t15) -- node[ineqE]{} (t16);

\path (t21) -- node[ineqE]{} (t22) -- node[ineqE]{} (t23) -- node[ineqE]{} (t24) -- node[ineqE]{} (t25);

\path (t31) -- node[ineqE]{} (t32) -- node[ineqE]{} (t33) -- node[ineqE]{} (t34);

\path (t41) -- node[ineqE]{} (t42) -- node[ineqE]{} (t43);

\path (t51) -- node[ineqE]{} (t52);

\path (t11) -- node[ineqS]{} (t21) -- node[ineqS]{} (t31) -- node[ineqS]{} (t41) -- node[ineqS]{} (t51) -- node[ineqS]{} (t61);

\path (t12) -- node[ineqS]{} (t22) -- node[ineqS]{} (t32) -- node[ineqS]{} (t42) -- node[ineqS]{} (t52);

\path (t13) -- node[ineqS]{} (t23) -- node[ineqS]{} (t33) -- node[ineqS]{} (t43);

\path (t14) -- node[ineqS]{} (t24) -- node[ineqS]{} (t34);

\path (t15) -- node[ineqS]{} (t25);

\node at (8.3,-3.7) {\Large $\longmapsto$};

\begin{scope}[shift={(9.3,0)}]

\node (t11) at (1,-1) {$\lambda_{1}$};
\node (t12) at (2,-1) {$z_{5,1}$};
\node (t13) at (3,-1) {$z_{4,1}$};
\node (t14) at (4,-1) {$z_{3,1}$};
\node (t15) at (5,-1) {$z_{2,1}$};
\node (t16) at (6,-1) {$z_{1,1}$};

\node (t21) at (1,-2) {$z_{5,1}$};
\node (t22) at (2,-2) {$\lambda_{2}$};
\node (t23) at (3,-2) {$z_{5,2}$};
\node (t24) at (4,-2) {$z_{4,2}$};
\node (t25) at (5,-2) {$z_{3,2}$};

\node (t31) at (1,-3) {$z_{4,1}$};
\node (t32) at (2,-3) {$z_{5,2}$};
\node (t33) at (3,-3) {$\lambda_{3}$};
\node (t34) at (4,-3) {$z_{5,3}$};

\node (t41) at (1,-4) {$z_{3,1}$};
\node (t42) at (2,-4) {$z_{4,2}$};
\node (t43) at (3,-4) {$z_{5,3}$};

\node (t51) at (1,-5) {$z_{2,1}$};
\node (t52) at (2,-5) {$z_{3,2}$};

\node (t61) at (1,-6) {$z_{1,1}$};

\draw[dashed] (t11) -- (t22) -- (t33) -- (4.7,-4.7);

\node[text=blue] (u) at (0+0.25,0-0.25) {$u$};
\path (u) -- node[ineqNW]{} (t11);

\node[text=blue] (o1) at (7-0.25,-2+0.25) {$0$};
\node[text=blue] (o2) at (6-0.25,-3+0.25) {$0$};
\node[text=blue] (o3) at (5-0.25,-4+0.25) {$0$};
\node[text=blue] (o4) at (4-0.25,-5+0.25) {$0$};
\node[text=blue] (o5) at (3-0.25,-6+0.25) {$0$};
\node[text=blue] (o6) at (2-0.25,-7+0.25) {$0$};

\path (t16) -- node[ineqNW]{} (o1);
\path (t25) -- node[ineqNW]{} (o2);
\path (t34) -- node[ineqNW]{} (o3);
\path (t43) -- node[ineqNW]{} (o4);
\path (t52) -- node[ineqNW]{} (o5);
\path (t61) -- node[ineqNW]{} (o6);

\path (t11) -- node[ineqW]{} (t12) -- node[ineqW]{} (t13) -- node[ineqW]{} (t14) -- node[ineqW]{} (t15) -- node[ineqW]{} (t16);

\path (t21) -- node[ineqW]{} (t22) -- node[ineqW]{} (t23) -- node[ineqW]{} (t24) -- node[ineqW]{} (t25);

\path (t31) -- node[ineqW]{} (t32) -- node[ineqW]{} (t33) -- node[ineqW]{} (t34);

\path (t41) -- node[ineqW]{} (t42) -- node[ineqW]{} (t43);

\path (t51) -- node[ineqW]{} (t52);

\path (t11) -- node[ineqN]{} (t21) -- node[ineqN]{} (t31) -- node[ineqN]{} (t41) -- node[ineqN]{} (t51) -- node[ineqN]{} (t61);

\path (t12) -- node[ineqN]{} (t22) -- node[ineqN]{} (t32) -- node[ineqN]{} (t42) -- node[ineqN]{} (t52);

\path (t13) -- node[ineqN]{} (t23) -- node[ineqN]{} (t33) -- node[ineqN]{} (t43);

\path (t14) -- node[ineqN]{} (t24) -- node[ineqN]{} (t34);

\path (t15) -- node[ineqN]{} (t25);
\end{scope}
\end{tikzpicture}
\caption{A pictorial representation of change of variables~\eqref{eq:changeVariablesDoubleSym} for $n=3$.
The symmetric triangular array $t$, on the left-hand side, is mapped onto (two identical copies of) a split orthogonal pattern $z$, on the right-hand side.
The transformation reverses the ordering of the variables.
The dashed line on the right-hand side goes through the partition $\lambda=(\lambda_1,\dots,\lambda_n)=(z_{2n,1},\dots,z_{2n,n})$, which is the shape of $z$.
All entries, before and after the transformation, are bounded below by $0$ and above by $u$ (in blue).}
\label{fig:changeVariablesDoubleSym}
\end{figure}

Analogously to Subsection~\ref{subsec:RSK_antisymLPP}, we now write
\[
\P\left(L^{\pointToLineSym}_{\alpha,\beta}(2n) \leq u\right)
= \sum_{t\colon \max t_{i,j}\leq u} \P(T=t) \, .
\]
Next, we change variables in the latter summation, by setting
\begin{equation}
\label{eq:changeVariablesDoubleSym}
z_{i,j} := u - t_{2n+j-i,j} = u - t_{j,2n+j-i} \qquad \text{for} \,\, 1\leq i\leq 2n \, ,  \,\, 1\leq j\leq \ceil{i/2} \, .
\end{equation}
Such a transformation, illustrated in Figure~\ref{fig:changeVariablesDoubleSym}, defines a new array $z=(z_{i,j})_{1\leq i\leq 2n,  \, 1\leq j\leq \ceil{i/2}}$ that, due to the properties induced by $t$, turns out to be a split orthogonal pattern of height $2n$ and shape $\lambda \subseteq u^{n}$.
The \emph{atypical} entries of $z$ are in bijection with the half-integer entries of $t$ above (or, equivalently, below) the diagonal.
Using the fact that $\sigma_{2n-k}(t) = \ceil{k/2} u - \abs{z_k}$ for $0\leq k\leq 2n$, we then obtain:
\begin{align*}
& c^{\doubleSymSq}_{\alpha,\beta} \cdot \P\left(L^{\pointToLineSym}_{\alpha,\beta}(2n) \leq u\right) \\
= \, &\sum_{\lambda \subseteq u^{n}} \,
\sum_{z \in \soGT^{(2n)}_{\lambda}}
\alpha^{\sum_{j=1}^n (-1)^{n-j} (u-\lambda_j)}
\beta^{\abs{\atyp(z)}}
\prod_{i=1}^n
p_i^{u -\abs{z_{2i-1}} +\abs{z_{2i-2}}}
\prod_{j=1}^n
p_j^{-\abs{z_{2j-1}} + \abs{z_{2j}}}
 \\
= \, &\prod_{i=1}^n p_i^u
\sum_{\lambda \subseteq u^{n}}
\alpha^{\sum_{j=1}^n (-1)^{n-j}
(u-\lambda_j)}
\sum_{z \in \soGT^{(2n)}_{\lambda}}
\beta^{\abs{\atyp(z)}}
\prod_{i=1}^n
p_i^{\type(z)_{2i}-\type(z)_{2i-1}} \, .
\end{align*}
We recognize the latter sum over $\soGT^{(2n)}_{\lambda}$ to be the $\CB$-interpolating Schur polynomial appearing on the right-hand side of~\eqref{eq:doubleSymLPP_pointToLine}, thus proving the desired identity.

\section{Decomposition of Gelfand-Tsetlin patterns}
\label{sec:decomposition}

Using a method that we call \emph{decomposition of Gelfand-Tsetlin patterns} (actually including all types of patterns introduced in Section~\ref{sec:patterns}), here we prove Theorems~\ref{thm:decompositionCB}, \ref{thm:decompositionD}, and~\ref{thm:decompositionA}.
These results, in particular, imply the following identities: $C^{\antisymSq}_\beta = D^{\antisymSq}_\beta$ of Theorem~\ref{thm:antisymLPP} (and its specializations $C^{\antisymSq}_0 = D^{\antisymSq}_0$ of Corollary~\ref{coro:antisymLPP_0} and $C^{\antisymSq}_1 = D^{\antisymSq}_1$ of Corollary~\ref{coro:antisymLPP_1}), $C^{\antisymSq}_0 = E^{\antisymSq}_0$ of Corollary~\ref{coro:antisymLPP_0}, $C^{\symSq}_\alpha = D^{\symSq}_\alpha$ of Theorem~\ref{thm:symLPP} (and its specializations $C^{\symSq}_0 = D^{\symSq}_0$ of Corollary~\ref{coro:symLPP_0} and $C^{\symSq}_1 = D^{\symSq}_1$ of Corollary~\ref{coro:symLPP_1}) and $C^{\doubleSymSq}_{\alpha,0} = D^{\doubleSymSq}_{\alpha,0}$ of Theorem~\ref{thm:doubleSymLPP_beta=0}.

\subsection{Proof of Theorem~\ref{thm:decompositionCB}}
\label{subsec:decompositionCB}

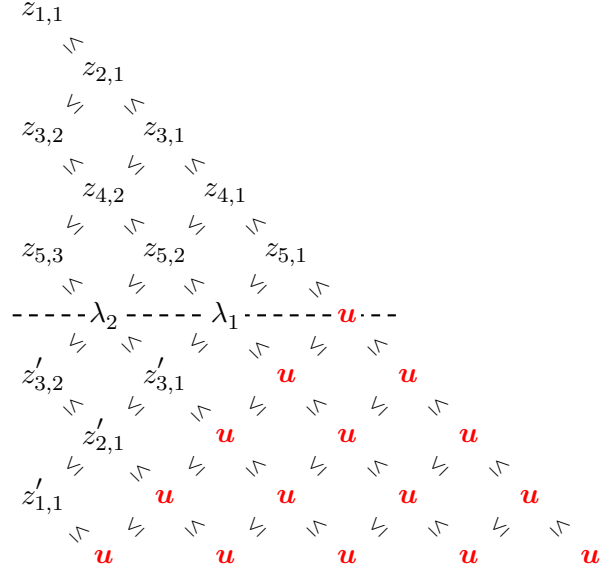
\begin{figure}
\centering
\begin{tikzpicture}[scale=0.8, every node/.style={fill=white, scale=1, inner sep=1.5pt}, ineq1/.style={rotate=45, node contents={\tiny $\leq$}}, ineq2/.style={rotate=-45, node contents={\tiny $\leq$}}]

\node (z11) at (1,-1) {$z_{1,1}$};
\node (z21) at (2,-2) {$z_{2,1}$};
\node (z22) at (0,-2) {};
\node (z31) at (3,-3) {$z_{3,1}$};
\node (z32) at (1,-3) {$z_{3,2}$};
\node (z41) at (4,-4) {$z_{4,1}$};
\node (z42) at (2,-4) {$z_{4,2}$};
\node (z43) at (0,-4) {};
\node (z51) at (5,-5) {$z_{5,1}$};
\node (z52) at (3,-5) {$z_{5,2}$};
\node (z53) at (1,-5) {$z_{5,3}$};
\node (z61) at (6,-6) {$\textcolor{red}{\bm{u}}$};
\node (z62) at (4,-6) {$\lambda_1$};
\node (z63) at (2,-6) {$\lambda_2$};
\node (z64) at (0,-6) {};
\node (z71) at (7,-7) {$\textcolor{red}{\bm{u}}$};
\node (z72) at (5,-7) {$\textcolor{red}{\bm{u}}$};
\node (z73) at (3,-7) {$z'_{3,1}$};
\node (z74) at (1,-7) {$z'_{3,2}$};
\node (z81) at (8,-8) {$\textcolor{red}{\bm{u}}$};
\node (z82) at (6,-8) {$\textcolor{red}{\bm{u}}$};
\node (z83) at (4,-8) {$\textcolor{red}{\bm{u}}$};
\node (z84) at (2,-8) {$z'_{2,1}$};
\node (z85) at (0,-8) {};
\node (z91) at (9,-9) {$\textcolor{red}{\bm{u}}$};
\node (z92) at (7,-9) {$\textcolor{red}{\bm{u}}$};
\node (z93) at (5,-9) {$\textcolor{red}{\bm{u}}$};
\node (z94) at (3,-9) {$\textcolor{red}{\bm{u}}$};
\node (z95) at (1,-9) {$z'_{1,1}$};
\node (z101) at (10,-10) {$\textcolor{red}{\bm{u}}$};
\node (z102) at (8,-10) {$\textcolor{red}{\bm{u}}$};
\node (z103) at (6,-10) {$\textcolor{red}{\bm{u}}$};
\node (z104) at (4,-10) {$\textcolor{red}{\bm{u}}$};
\node (z105) at (2,-10) {$\textcolor{red}{\bm{u}}$};
\node (z106) at (0,-10) {};

% arrows
\path (z11) -- node[ineq2]{} (z21);
\path (z32) -- node[ineq1]{} (z21) -- node[ineq2]{} (z31);
\path (z32) -- node[ineq2]{} (z42) -- node[ineq1]{} (z31) -- node[ineq2]{} (z41);
\path (z53) -- node[ineq1]{} (z42) -- node[ineq2]{} (z52) -- node[ineq1]{} (z41) -- node[ineq2]{} (z51);
\path (z53) -- node[ineq2]{} (z63) -- node[ineq1]{} (z52) -- node[ineq2]{} (z62) -- node[ineq1]{} (z51) -- node[ineq2]{} (z61);
\path (z74) -- node[ineq1]{} (z63) -- node[ineq2]{} (z73) -- node[ineq1]{} (z62) -- node[ineq2]{} (z72) -- node[ineq1]{} (z61) -- node[ineq2]{} (z71);
\path (z74) -- node[ineq2]{} (z84) -- node[ineq1]{} (z73) -- node[ineq2]{} (z83) -- node[ineq1]{} (z72) -- node[ineq2]{} (z82) -- node[ineq1]{} (z71) -- node[ineq2]{} (z81);
\path (z95) -- node[ineq1]{} (z84) -- node[ineq2]{} (z94) -- node[ineq1]{} (z83) -- node[ineq2]{} (z93) -- node[ineq1]{} (z82) -- node[ineq2]{} (z92) -- node[ineq1]{} (z81) -- node[ineq2]{} (z91);
\path (z95) -- node[ineq2]{} (z105) -- node[ineq1]{} (z94) -- node[ineq2]{} (z104) -- node[ineq1]{} (z93) -- node[ineq2]{} (z103) -- node[ineq1]{} (z92) -- node[ineq2]{} (z102) -- node[ineq1]{} (z91) -- node[ineq2]{} (z101);

\begin{pgfonlayer}{background}
\draw[dash pattern=on 3.3pt off 2.8pt, thick] (0.5,-6) -- (6.8,-6);
\end{pgfonlayer}

\end{tikzpicture}
\caption{A split orthogonal pattern of height $2n+2m$ and shape $u^{n+m}$, for $n\geq m$, can be decomposed into: a ``frozen'' triangular part of $u$'s, shown in red; a split orthogonal pattern $z$ of height $2n$ and shape $(u^{n-m}, \lambda)$, which overlaps with the frozen part if $n>m$; a split orthogonal pattern $z'$ of height $2m$ and shape $\lambda$.
The shapes of $z$ and $z'$ lie at the level of the ``cut'', illustrated by the dashed line.
Such a decomposition is valid, in particular, for symplectic patterns of even height.
The picture is for $n=3$ and $m=2$.}
\label{fig:frozenSoOddPattern}
\end{figure}

Let us start by proving~\eqref{eq:decompositionCB}.
The idea is to show that, for $n\geq m$ and $u\in\frac{1}{2}\Z_{\geq 0}$, there exists a bijection
\begin{equation}
\label{eq:soBijection}
\soGT^{(2n+2m)}_{u^{n+m}}
\quad \longleftrightarrow \quad
\bigcup_{\lambda \subseteq u^{m}}
\soGT^{(2n)}_{(u^{n-m} , \lambda)} \times \soGT^{(2m)}_{\lambda}
\end{equation}
that yields the desired identity.

We invite the reader to see Figure~\ref{fig:frozenSoOddPattern} for an illustration of the bijection, which can be constructed as follows.
We start by observing that any $Z\in \soGT^{(2n+2m)}_{u^{n+m}}$ satisfies $Z_{i,j}=u$ for $i-j \geq n+m$, due to the interlacing conditions.
We call \emph{frozen part} the portion of the pattern whose entries are all equal to $u$.
Let us now cut $Z$ horizontally at level $2n$ (from the top) and ignore all the frozen entries below such a cut, i.e.\ all $Z_{i,j}=u$ with $i-j \geq n+m$ and $i>2n$.
What remains can be seen as the union of two smaller split orthogonal patterns.
The first one, denoted by $z$, is made of the first $2n$ rows of $Z$ (from the top), i.e.\ $z_{i,j}:=Z_{i,j}$ for all $1\leq i\leq 2n$ and $1\leq j\leq \ceil{i/2}$.
The shape of $z$ is $(u^{n-m}, Z_{2n,n-m+1},\dots,Z_{2n,n})$.
The second pattern, denoted by $z'$, is obtained by reading from bottom to top the last $2m$ rows of $Z$ after removing the \emph{whole} frozen part: namely, $z'_{i,j} := Z_{2n+2m-i,n+m-i+j}$ for all $1\leq i\leq 2m$ and $1\leq j\leq \ceil{i/2}$.
The shape of $z'$ is
$(Z_{2n,n-m+1},\dots,Z_{2n,n})$.
Notice that, when $n>m$, $z$ overlaps with the frozen part; conversely, when $n=m$, there is no overlap and the two patterns $z$ and $z'$ have the same shape.
By definition of split orthogonal pattern, each of the odd ends of $Z$ is either an integer or a half-integer independently of everything else; therefore, the same holds for $z$ and $z'$.
Again by definition, all the entries of $Z$ \emph{except the odd ends} are either simultaneously integers or simultaneously half-integers; therefore, the same holds for $z$ and $z'$.
Moreover, as can be visualized in Figure \ref{fig:frozenSoOddPattern}, the interlacing conditions for $Z$ directly imply that: (1) $z$ and $z'$ also satisfy the interlacing conditions and (2) all entries of $z$ and $z'$ are less than or equal to $u$.
Changing now notation and denoting
\begin{align*}
\lambda
:=(\lambda_1,\dots,\lambda_m)
:=(Z_{2n,n-m+1},\dots,Z_{2n,n}) \, ,
\end{align*}
it turns out that the shape of $z$ is $(u^{n-m}, \lambda)$ and the shape of $z'$ is $\lambda$.
Here, $\lambda$ is an arbitrary $m$-partition or $m$-half-partition (according to whether $u$ is an integer or a half-integer\footnote{Notice that the $\lambda_i$'s occupy a row of even index, so none of them is an odd end.}) such that $\lambda \subseteq u^{m}$.
This proves that $z\in \soGT^{(2n)}_{(u^{n-m} , \lambda)}$ and $z' \in \soGT^{(2m)}_{\lambda}$, thus establishing bijection~\eqref{eq:soBijection}.

Under this bijection, it is easy to verify that $\type(Z)_i = \type(z)_i$ for $1\leq i \leq 2n$ and $\type(Z)_{2n+i} = u - \type(z')_{2m-i+1}$ for $1\leq i \leq 2m$.
It is also immediate that the number of atypical entries of $Z$ equals the number of atypical entries of $z$ and $z'$, i.e.\ $\abs{\atyp(Z)}=\abs{\atyp(z)} + \abs{\atyp(z')}$.
We will now use these facts to prove~\eqref{eq:decompositionCB}.
In particular, we will split the summation over all patterns $Z\in \soGT^{(2n+2m)}_{u^{n+m}}$ by first summing over the sub-patterns $z\in \soGT^{(2n)}_{(u^{n-m} , \lambda)}$ and $z'\in \soGT^{(2m)}_{\lambda}$ for a fixed $m$-(half-)partition  
$\lambda$, and then summing over all $\lambda \subseteq u^{m}$.
According to Definition~\ref{def:CBtrans}, we thus have:
\begin{align*}
&\trans{u^{n+m}}(x_1,\dots,x_n,y_m,\dots,y_1; \beta) \\
= \, &\sum_{Z \in \soGT^{(2n+2m)}_{u^{n+m}}}
\beta^{\abs{\atyp(Z)}}
\prod_{i=1}^n
x_i^{\type(Z)_{2i} - \type(Z)_{2i-1}}
\prod_{i=1}^m y_{i}^{\type(Z)_{2n+2m-2i+2} - \type(Z)_{2n+2m-2i+1}} \\
= \, &\sum_{\lambda\subseteq u^{m}}
\sum_{z \in \soGT^{(2n)}_{(u^{n-m} , \lambda)}}
\beta^{\abs{\atyp(z)}}
\prod_{i=1}^{n}
x_i^{\type(z)_{2i} - \type(z)_{2i-1}} \\*
&\qquad\qquad\qquad\qquad\qquad\qquad\qquad \times
\sum_{z' \in \soGT^{(2m)}_{\lambda}}
\beta^{\abs{\atyp(z')}}
\prod_{i=1}^n
y_i^{[u-\type(z')_{2i-1}] - [u-\type(z')_{2i}]} \\
=\,& \sum_{\lambda\subseteq u^{m}} 
\trans{(u^{n-m} , \lambda)} (x_1, \dots, x_n; \beta) \cdot
\trans{\lambda} (y_1, \dots, y_m; \beta) \, .
\end{align*}
Since $\CB$-interpolating Schur polynomials are symmetric (see Subsection~\ref{subsec:CBtransition}), we have
\[
\trans{u^{n+m}}(x_1,\dots,x_n,y_m,\dots,y_1; \beta)
= \trans{u^{n+m}}(x_1,\dots,x_n,y_1,\dots,y_m; \beta) \, ,
\]
which concludes the proof of~\eqref{eq:decompositionCB}.

The specializations of~\eqref{eq:decompositionCB} to $\beta=0$ and $\beta=1$ yield the corresponding identities~\eqref{eq:decompositionC} and~\eqref{eq:decompositionB} for even symplectic and odd orthogonal Schur polynomials, respectively.
Notice however that a more \emph{ad hoc} proof of~\eqref{eq:decompositionC} is based on the restriction of bijection~\eqref{eq:soBijection} to symplectic patterns of even height, which reads as
\begin{equation}
\label{eq:spBijection}
\spGT^{(2n+2m)}_{u^{n+m}}
\quad \longleftrightarrow \quad
\bigcup_{\lambda\subseteq u^{m}} \spGT^{(2n)}_{(u^{n-m} , \lambda)} \times \spGT^{(2m)}_{\lambda}
\end{equation}
for $u\in\Z_{\geq 0}$.

We now prove~\eqref{eq:decompositionOddSp}.
This is based on the bijection
\begin{equation}
\label{eq:oddSpBijection}
\spGT^{(2n+2m+2)}_{u^{n+m+1}}
\quad \longleftrightarrow \quad
\bigcup_{\lambda\subseteq u^{m+1}} \spGT^{(2n+1)}_{(u^{n-m} , \lambda)} \times \spGT^{(2m+1)}_{\lambda} \, ,
\end{equation}
valid for $n\geq m$ and $u\in\Z_{\geq 0}$, which can be proved via a ``graphical'' decomposition as we did for~\eqref{eq:soBijection}.
Essentially, one cuts a symplectic pattern $Z$ of height $2n+2m+2$ and rectangular shape $u^{m+n+1}$ at the level of the $(2n+1)$-th row, thus obtaining a frozen part of $u$'s and two sub-patterns $z$ and $z'$ of height $2n+1$ and $2m+1$, respectively.

By Definition~\ref{def:spSchur}, we have
\begin{align*}
&\sp^{(2n+2m+2)}_{u^{n+m+1}}(x_1,\dots,x_n,s^{-1},y_m,\dots,y_1) \\
= \, &\sum_{Z \in \spGT^{(2n+2m+2)}_{u^{n+m+1}}}
\prod_{i=1}^n
x_i^{\type(Z)_{2i} - \type(Z)_{2i-1}}
\cdot
s^{-[\type(Z)_{2n+2} - \type(Z)_{2n+1}]} \\*
&\qquad\qquad\qquad\qquad\qquad \times
\prod_{i=1}^m y_{i}^{\type(Z)_{2n+2m-2i+4} - \type(Z)_{2n+2m-2i+3}} \, . \\
\intertext{Using now bijection~\eqref{eq:oddSpBijection}, under which $\type(Z)_i = \type(z)_i$ for $1\leq i \leq 2n+1$ and $\type(Z)_{2n+1+i} = u - \type(z')_{2m+2-i}$ for $1\leq i \leq 2m+1$, the above expression becomes}
&\sum_{\lambda\subseteq u^{m+1}}
\sum_{z \in \spGT^{(2n+1)}_{(u^{n-m} , \lambda)}}
\prod_{i=1}^{n}
x_i^{\type(z)_{2i} - \type(z)_{2i-1}} \cdot
s^{\type(z)_{2n+1}} \\*
&\qquad\qquad\qquad\qquad\qquad \times
\sum_{z' \in \spGT^{(2m+1)}_{\lambda}}
\prod_{i=1}^m
y_i^{[u-\type(z')_{2i-1}] - [u-\type(z')_{2i}]} \cdot 
s^{-u+\type(z')_{2m+1}} \\
=\,& s^{-u} \sum_{\lambda\subseteq u^{m+1}} 
\sp^{(2n+1)}_{(u^{n-m} , \lambda)} (x_1, \dots, x_n; s) \cdot
\sp^{(2m+1)}_{\lambda} (y_1, \dots, y_m; s) \, .
\end{align*}
The latter equality follows from Definition~\ref{def:oddSpSchur} of odd symplectic characters.
On the other hand, recalling from Subsection~\ref{subsec:sp} the invariance properties of even symplectic characters, we have
\[
\sp^{(2n+2m+2)}_{u^{n+m+1}}(x_1,\dots,x_n,s^{-1},y_m,\dots,y_1)
= \sp^{(2n+2m+2)}_{u^{n+m+1}}(x_1,\dots,x_n,y_1,\dots,y_m,s) \, ,
\]
from which~\eqref{eq:decompositionOddSp} follows.

\subsection{Proof of Theorem~\ref{thm:decompositionD}}
\label{subsec:decompositionD}

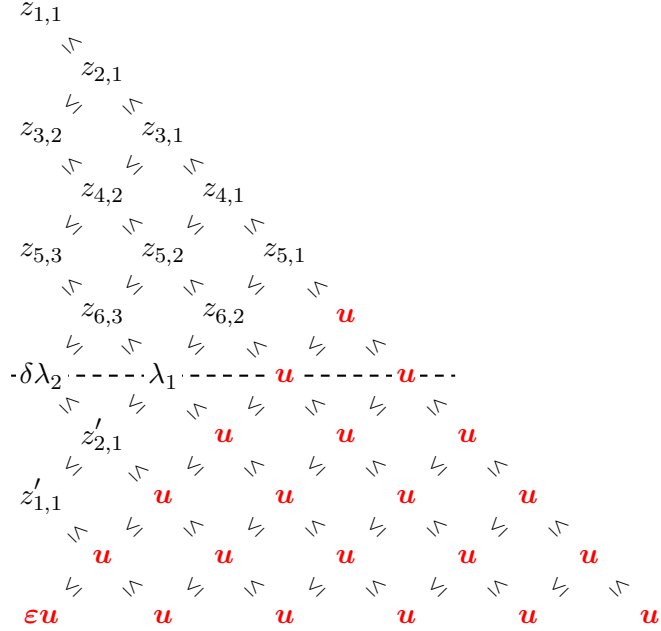
\begin{figure}
\centering
\begin{tikzpicture}[scale=0.8, every path/.style={draw=none}, every node/.style={fill=white, scale=1, inner sep=1.5pt}, ineq1/.style={rotate=45, node contents={\tiny $\leq$}}, ineq2/.style={rotate=-45, node contents={\tiny $\leq$}}]

\node (z11) at (1,-1) {$z_{1,1}$};
\node (z21) at (2,-2) {$z_{2,1}$};
\node (z22) at (0,-2) {};
\node (z31) at (3,-3) {$z_{3,1}$};
\node (z32) at (1,-3) {$z_{3,2}$};
\node (z41) at (4,-4) {$z_{4,1}$};
\node (z42) at (2,-4) {$z_{4,2}$};
\node (z43) at (0,-4) {};
\node (z51) at (5,-5) {$z_{5,1}$};
\node (z52) at (3,-5) {$z_{5,2}$};
\node (z53) at (1,-5) {$z_{5,3}$};
\node (z61) at (6,-6) {$\textcolor{red}{\bm{u}}$};
\node (z62) at (4,-6) {$z_{6,2}$};
\node (z63) at (2,-6) {$z_{6,3}$};
\node (z64) at (0,-6) {};
\node (z71) at (7,-7) {$\textcolor{red}{\bm{u}}$};
\node (z72) at (5,-7) {$\textcolor{red}{\bm{u}}$};
\node (z73) at (3,-7) {$\lambda_1$};
\node (z74) at (1,-7) {$\delta \lambda_2$};
\node (z81) at (8,-8) {$\textcolor{red}{\bm{u}}$};
\node (z82) at (6,-8) {$\textcolor{red}{\bm{u}}$};
\node (z83) at (4,-8) {$\textcolor{red}{\bm{u}}$};
\node (z84) at (2,-8) {$z'_{2,1}$};
\node (z85) at (0,-8) {};
\node (z91) at (9,-9) {$\textcolor{red}{\bm{u}}$};
\node (z92) at (7,-9) {$\textcolor{red}{\bm{u}}$};
\node (z93) at (5,-9) {$\textcolor{red}{\bm{u}}$};
\node (z94) at (3,-9) {$\textcolor{red}{\bm{u}}$};
\node (z95) at (1,-9) {$z'_{1,1}$};
\node (z101) at (10,-10) {$\textcolor{red}{\bm{u}}$};
\node (z102) at (8,-10) {$\textcolor{red}{\bm{u}}$};
\node (z103) at (6,-10) {$\textcolor{red}{\bm{u}}$};
\node (z104) at (4,-10) {$\textcolor{red}{\bm{u}}$};
\node (z105) at (2,-10) {$\textcolor{red}{\bm{u}}$};
\node (z106) at (0,-10) {};
\node (z111) at (11,-11) {$\textcolor{red}{\bm{u}}$};
\node (z112) at (9,-11) {$\textcolor{red}{\bm{u}}$};
\node (z113) at (7,-11) {$\textcolor{red}{\bm{u}}$};
\node (z114) at (5,-11) {$\textcolor{red}{\bm{u}}$};
\node (z115) at (3,-11) {$\textcolor{red}{\bm{u}}$};
\node (z116) at (1,-11) {$\textcolor{red}{\bm{\epsilon u}}$};

% arrows
\path (z11) -- node[ineq2]{} (z21);
\path (z32) -- node[ineq1]{} (z21) -- node[ineq2]{} (z31);
\path (z32) -- node[ineq2]{} (z42) -- node[ineq1]{} (z31) -- node[ineq2]{} (z41);
\path (z53) -- node[ineq1]{} (z42) -- node[ineq2]{} (z52) -- node[ineq1]{} (z41) -- node[ineq2]{} (z51);
\path (z53) -- node[ineq2]{} (z63) -- node[ineq1]{} (z52) -- node[ineq2]{} (z62) -- node[ineq1]{} (z51) -- node[ineq2]{} (z61);
\path (z74) -- node[ineq1]{} (z63) -- node[ineq2]{} (z73) -- node[ineq1]{} (z62) -- node[ineq2]{} (z72) -- node[ineq1]{} (z61) -- node[ineq2]{} (z71);
\path (z74) -- node[ineq2]{} (z84) -- node[ineq1]{} (z73) -- node[ineq2]{} (z83) -- node[ineq1]{} (z72) -- node[ineq2]{} (z82) -- node[ineq1]{} (z71) -- node[ineq2]{} (z81);
\path (z95) -- node[ineq1]{} (z84) -- node[ineq2]{} (z94) -- node[ineq1]{} (z83) -- node[ineq2]{} (z93) -- node[ineq1]{} (z82) -- node[ineq2]{} (z92) -- node[ineq1]{} (z81) -- node[ineq2]{} (z91);
\path (z95) -- node[ineq2]{} (z105) -- node[ineq1]{} (z94) -- node[ineq2]{} (z104) -- node[ineq1]{} (z93) -- node[ineq2]{} (z103) -- node[ineq1]{} (z92) -- node[ineq2]{} (z102) -- node[ineq1]{} (z91) -- node[ineq2]{} (z101);
\path (z116) -- node[ineq1]{} (z105) -- node[ineq2]{} (z115) -- node[ineq1]{} (z104) -- node[ineq2]{} (z114) -- node[ineq1]{} (z103) -- node[ineq2]{} (z113) -- node[ineq1]{} (z102) -- node[ineq2]{} (z112) -- node[ineq1]{} (z101) -- node[ineq2]{} (z111);

\begin{pgfonlayer}{background}
\draw[dash pattern=on 3.3pt off 2.8pt, thick] (0.5,-7) -- (7.8,-7);
\end{pgfonlayer}

\end{tikzpicture}
\caption{An orthogonal pattern of height $2n+2m-1$ and shape $u^{n+m}$, for $n\geq m$, can be decomposed into: a ``frozen'' triangular part of $u$'s, shown in red; an orthogonal pattern $z$ of height $2n-1$ and shape $(u^{n-m} , \lambda_\delta)$, which overlaps with the frozen part if $n>m$; an orthogonal pattern $z'$ of height $2m-1$ and shape $\lambda_\delta$.
The shapes of $z$ and $z'$ lie at the level of the ``cut'', illustrated by the dashed line.
The picture is for $n=4$ and $m=2$.}
\label{fig:frozenSoEvenPattern}
\end{figure}

Eq.~\eqref{eq:decompositionBintoD} immediately follows from~\eqref{eq:decompositionD} and~\eqref{eq:soEvenRect=soOddRect}, so it suffices to prove~\eqref{eq:decompositionD}.
For this, we are going to adapt the proof of Theorem~\ref{thm:decompositionCB} to the case of orthogonal patterns and characters of type D.
We wish now to show that, for $n\geq m$, $u\in \frac{1}{2}\Z_{\geq 0}$ and $\epsilon=\pm 1$, there exists a bijection
\begin{equation}
\label{eq:evenSoBijection}
\oGT^{(2n+2m-1)}_{u^{n+m}_{\epsilon}}
\quad \longleftrightarrow \quad
\bigcup_{\lambda_{\delta} \subseteq u^{m}} \oGT^{(2n-1)}_{(u^{n-m} , \lambda_{\delta})} \times \oGT^{(2m-1)}_{\lambda_{\delta}}
\end{equation}
that yields identity~\eqref{eq:decompositionD}.

Similarly to the cases considered in the previous subsection, an orthogonal pattern $Z$ of height $2n+2m-1$ and shape $u^{n+m}_{\epsilon}$ has a \emph{frozen part} of $u$'s, as visualized in Figure~\ref{fig:frozenSoEvenPattern}.
Ignoring all the frozen entries below the $(2n-1)$-th row, we are left with two smaller orthogonal patterns of height $2n-1$ and $2m-1$ respectively.
The first one, denoted by $z$, is defined by $z_{i,j}:=Z_{i,j}$ for all $1\leq i\leq 2n-1$ and $1\leq j\leq \ceil{i/2}$ and has shape $(u^{n-m}, Z_{2n-1,n-m+1},\dots,Z_{2n-1,n})$.
The second one, denoted by $z'$, is defined by $z'_{i,j} := Z_{2(n+m)-2-i,n+m-1-i+j}$ for all $1\leq i\leq 2m-1$ and $1\leq j\leq \ceil{i/2}$ and has shape $(Z_{2n-1,n-m+1},\dots,Z_{2n-1,n})$.
These are indeed orthogonal patterns as they inherit the properties of $Z$.
Denoting
\begin{align*}
\lambda_{\delta}
:=(\lambda_1,\dots,\delta\lambda_m)
:=(Z_{2n-1,n-m+1},\dots,Z_{2n-1,n}) \, ,
\end{align*}
it turns out that the shape of $z$ is $(u^{n-m}, \lambda_{\delta})$ and the shape of $z'$ is $\lambda_{\delta}$.
Here, $\lambda_\delta$ is an arbitrary signed $m$-partition or signed $m$-half-partition (according to whether $u$ is an integer or a half-integer) such that $\lambda_\delta \subseteq u^{m}$.
This proves that $z\in \oGT^{(2n-1)}_{(u^{n-m}  , \lambda_{\delta})}$ and $z' \in \oGT^{(2m-1)}_{\lambda_{\delta}}$, thus establishing bijection~\eqref{eq:evenSoBijection}.
Notice that $\type(Z)_i = \type(z)_i$ for $1\leq i \leq 2n-1$, $\type(Z)_{2n+i-1} = u - \type(z')_{2m-i}$ for $1\leq i \leq 2m-1$ and $\type(Z)_{2n+2m-1} = u$.

We now proceed to prove~\eqref{eq:decompositionD}.
Starting from Definition~\ref{def:soEvenSchur} of even orthogonal Schur polynomials and applying the results described above, we obtain:
\begin{align*}
&\so^{(2n+2m)}_{u^{n+m}_{\epsilon}}(x_1,\dots,x_n,y_m,\dots,y_1) \\
= &\sum_{Z \in \oGT^{(2n+2m-1)}_{u^{n+m}_{\epsilon}}}
x_1^{\sign(Z_{1,1}) \type(Z)_1}
\prod_{i=2}^n
x_i^{\sign(Z_{2i-3,i-1}) \sign(Z_{2i-1,i}) [\type(Z)_{2i-1} - \type(Z)_{2i-2}]} \\*
& \times
\prod_{i=1}^m y_{i}^{\sign(Z_{2n+2m-2i-1,n+m-i}) \sign(Z_{2n+2m-2i+1,n+m-i+1}) [\type(Z)_{2n+2m-2i+1} - \type(Z)_{2n+2m-2i}]} \\
= &\sum_{\lambda_\delta \subseteq u^{m}}
\sum_{z \in \oGT^{(2n-1)}_{(u^{n-m}  , \lambda_\delta)}}
\!\!\!\!\!\!\!\!\!
x_1^{\sign(z_{1,1}) \type(z)_1}
\prod_{i=2}^n
x_i^{\sign(z_{2i-3,i-1}) \sign(z_{2i-1,i}) [\type(z)_{2i-1} - \type(z)_{2i-2}]} \\*
&\times
\sum_{z' \in \oGT^{(2m-1)}_{\lambda_\delta}}
y_1^{\epsilon \sign(z'_{1,1})\type(z')_1} 
\prod_{i=2}^{m}
y_{i}^{\sign(z'_{2i-1,i}) \sign(z'_{2i-3,i-1}) [\type(z')_{2i-1} - \type(z')_{2i-2})]} \\
= &\sum_{\lambda_\delta \subseteq u^{m}}
\sum_{z \in \oGT^{(2n-1)}_{(u^{n-m} , \lambda_\delta)}}
\!\!\!\!\!\!\!\!\!
x_1^{\sign(z_{1,1}) \type(z)_1}
\prod_{i=2}^n
x_i^{\sign(z_{2i-3,i-1}) \sign(z_{2i-1,i}) [\type(z)_{2i-1} - \type(z)_{2i-2}]} \\*
&\times
\sum_{z' \in \oGT^{(2m-1)}_{\lambda_{\delta\epsilon}}}
y_1^{\sign(z'_{1,1})\type(z')_1} 
\prod_{i=2}^{m}
y_{i}^{\sign(z'_{2i-1,i}) \sign(z'_{2i-3,i-1}) [\type(z')_{2i-1} - \type(z')_{2i-2})]} \, .
\end{align*}
For the latter equality we have set $z'_{2i-1,i} \mapsto \epsilon z'_{2i-1,i}$ for $1\leq i \leq m$, thus changing, if $\epsilon=-1$, the sign of all odd ends of $z'$.
The two sums over orthogonal patterns are, by definition, the two even orthogonal Schur polynomials appearing on the right-hand side of~\eqref{eq:decompositionD}.
On the other hand, by symmetry, the variables of the initial orthogonal Schur polynomial of shape $u^{n+m}_{\epsilon}$ can be reordered to get the left-hand side of~\eqref{eq:decompositionD}.

\subsection{Proof of Theorem~\ref{thm:decompositionA}}
\label{subsec:decompositionA}

This proof is similar in spirit to the previous ones, but differs for the presence of \emph{two} ``frozen parts'' (instead of one) in a \emph{triangular} pattern (instead of ``half-triangular'').

Let $u\geq v\geq 0$ be integers.
We will first show the existence of two natural bijections:
\begin{align}
\label{eq:sBijection1}
\GT^{(n+m)}_{(u^{n}  , v^{m})}
\quad &\longleftrightarrow \quad
\bigcup_{v^{n} \subseteq \mu \subseteq u^{n}}
\GT^{(m)}_{(\mu, \, v^{m-n})} \times \GT^{(n)}_{\mu} &&\text{if } n\leq m \, , \\
\label{eq:sBijection2}
\GT^{(n+m)}_{(u^{n}  , v^{m})}
\quad &\longleftrightarrow \quad
\bigcup_{v^{m} \subseteq \mu \subseteq u^{m}}
\GT^{(n)}_{(u^{n-m}  , \mu)} \times \GT^{(m)}_{\mu} &&\text{if } n\geq m \, .
\end{align}

\begin{figure}
\centering
\begin{subfigure}[b]{.5\linewidth}
\centering
\begin{tikzpicture}[scale=0.8, every path/.style={draw=none}, every node/.style={fill=white, scale=1, inner sep=1.5pt}, ineq1/.style={rotate=45, node contents={\tiny $\leq$}}, ineq2/.style={rotate=-45, node contents={\tiny $\leq$}}]

\begin{pgfonlayer}{background}
\draw[dash pattern=on 3.3pt off 2.8pt, thick] (-1.8,-3) -- (3.8,-3);
\end{pgfonlayer}

\node (z11) at (1,-1) {$z_{1,1}$};
\node (z21) at (2,-2) {$z_{2,1}$};
\node (z22) at (0,-2) {$z_{2,2}$};
\node (z31) at (3,-3) {$\mu_1$};
\node (z32) at (1,-3) {$\mu_2$};
\node (z33) at (-1,-3) {$\blue{\bm{v}}$};
\node (z41) at (4,-4) {$\textcolor{red}{\bm{u}}$};
\node (z42) at (2,-4) {$z'_{1,1}$};
\node (z43) at (0,-4) {$\blue{\bm{v}}$};
\node (z44) at (-2,-4) {$\blue{\bm{v}}$};
\node (z51) at (5,-5) {$\textcolor{red}{\bm{u}}$};
\node (z52) at (3,-5) {$\textcolor{red}{\bm{u}}$};
\node (z53) at (1,-5) {$\blue{\bm{v}}$};
\node (z54) at (-1,-5) {$\blue{\bm{v}}$};
\node (z55) at (-3,-5) {$\blue{\bm{v}}$};

% arrows
\path (z22) -- node[ineq1]{} (z11) -- node[ineq2]{} (z21);
\path (z33) -- node[ineq1]{} (z22) -- node[ineq2]{} (z32) -- node[ineq1]{} (z21) -- node[ineq2]{} (z31);
\path (z44) -- node[ineq1]{} (z33) -- node[ineq2]{} (z43) -- node[ineq1]{} (z32) -- node[ineq2]{} (z42) -- node[ineq1]{} (z31) -- node[ineq2]{} (z41);
\path (z55) -- node[ineq1]{} (z44) -- node[ineq2]{} (z54) -- node[ineq1]{} (z43) -- node[ineq2]{} (z53) -- node[ineq1]{} (z42) -- node[ineq2]{} (z52) -- node[ineq1]{} (z41) -- node[ineq2]{} (z51);

\end{tikzpicture}
\subcaption{Pattern in $\GT^{(5)}_{(u^{2} , v^{3})}$.}
\label{subfig:frozenGTpattern_n<m}
\end{subfigure}%
\begin{subfigure}[b]{.5\linewidth}
\centering
\begin{tikzpicture}[scale=0.8, every path/.style={draw=none}, every node/.style={fill=white, scale=1, inner sep=1.5pt}, ineq1/.style={rotate=45, node contents={\tiny $\leq$}}, ineq2/.style={rotate=-45, node contents={\tiny $\leq$}}]

\begin{pgfonlayer}{background}
\draw[dash pattern=on 3.3pt off 2.8pt, thick] (-1.8,-3) -- (3.8,-3);
\end{pgfonlayer}

\node (z11) at (1,-1) {$z_{1,1}$};
\node (z21) at (2,-2) {$z_{2,1}$};
\node (z22) at (0,-2) {$z_{2,2}$};
\node (z31) at (3,-3) {$\textcolor{red}{\bm{u}}$};
\node (z32) at (1,-3) {$\mu_1$};
\node (z33) at (-1,-3) {$\mu_2$};
\node (z41) at (4,-4) {$\textcolor{red}{\bm{u}}$};
\node (z42) at (2,-4) {$\textcolor{red}{\bm{u}}$};
\node (z43) at (0,-4) {$z'_{1,1}$};
\node (z44) at (-2,-4) {$\blue{\bm{v}}$};
\node (z51) at (5,-5) {$\textcolor{red}{\bm{u}}$};
\node (z52) at (3,-5) {$\textcolor{red}{\bm{u}}$};
\node (z53) at (1,-5) {$\textcolor{red}{\bm{u}}$};
\node (z54) at (-1,-5) {$\blue{\bm{v}}$};
\node (z55) at (-3,-5) {$\blue{\bm{v}}$};

% arrows
\path (z22) -- node[ineq1]{} (z11) -- node[ineq2]{} (z21);
\path (z33) -- node[ineq1]{} (z22) -- node[ineq2]{} (z32) -- node[ineq1]{} (z21) -- node[ineq2]{} (z31);
\path (z44) -- node[ineq1]{} (z33) -- node[ineq2]{} (z43) -- node[ineq1]{} (z32) -- node[ineq2]{} (z42) -- node[ineq1]{} (z31) -- node[ineq2]{} (z41);
\path (z55) -- node[ineq1]{} (z44) -- node[ineq2]{} (z54) -- node[ineq1]{} (z43) -- node[ineq2]{} (z53) -- node[ineq1]{} (z42) -- node[ineq2]{} (z52) -- node[ineq1]{} (z41) -- node[ineq2]{} (z51);

\end{tikzpicture}
\subcaption{Pattern in $\GT^{(5)}_{(u^{3} , v^{2})}$.}
\label{subfig:frozenGTpattern_n>m}
\end{subfigure}
\caption{A Gelfand-Tsetlin pattern of height $n+m$ and shape $(u^{n}, v^{m})$ can be decomposed into: a ``frozen'' triangular part of $u$'s, in red; a ``frozen'' triangular part of $v$'s, in blue; two Gelfand-Tsetlin patterns $z$ and $z'$ whose shapes contain a common partition $\mu$.
If $n\leq m$ as in Figure~\ref{subfig:frozenGTpattern_n<m}, $z$ is of height $m$ and shape $(\mu, v^{m-n})$, whereas $z'$ is of height $n$ and shape $\mu$.
If $n\geq m$ as in Figure~\ref{subfig:frozenGTpattern_n>m}, $z$ is of height $n$ and shape $(u^{n-m}, \mu)$, whereas $z'$ is of height $m$ and shape $\mu$.
The shapes of $z$ and $z'$ lie at the level of the ``cut'', illustrated by the dashed line.}
\label{fig:frozenGTpatterns}
\end{figure}
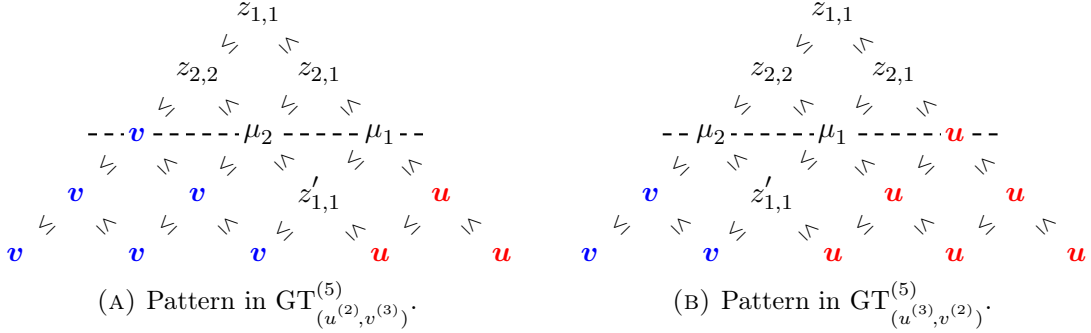

An illustration of these bijections is given by Figure~\ref{fig:frozenGTpatterns}.
A Gelfand-Tsetlin pattern $Z\in\GT^{(n+m)}_{(u^{n} , v^{m})}$ is characterized by a portion made of $u$'s only, which we call ``$u$-frozen part'', and a portion made of $v$'s only, which we call ``$v$-frozen part''.
More precisely, we have $Z_{i,j}=u$ for $i-j\geq m$ and $Z_{i,j}=v$ for $j\geq n+1$.
This phenomenon is due to the interlacing conditions~\eqref{eq:interlacing}, as in the case of (split) orthogonal patterns.
Assume first that $n\leq m$: in this case we cut $Z$ horizontally at level $m$ (from the top) and ignore all the frozen entries below such a cut, i.e.\ all $Z_{i,j}=u$ for $i-j\geq m$ and $Z_{i,j}=v$ for $j\geq n+1$, $i>m$.
What remains can be seen as the union of two Gelfand-Tsetlin patterns.
The first one, denoted by $z$, is made of the first $m$ rows of $Z$ (from the top), i.e.\ $z_{i,j}:=Z_{i,j}$ for all $1\leq j\leq i\leq m$.
The shape of $z$ is the partition $(Z_{m,1},\dots,Z_{m,n}, v^{m-n})$.
The second pattern, denoted by $z'$, is obtained by reading from bottom to top the last $n$ rows of $Z$ after removing the \emph{whole} frozen parts: namely, $z'_{i,j} := Z_{n+m-i,n-i+j}$ for all $1\leq j\leq i\leq n$.
The shape of $z'$ is the partition
$(Z_{m,1},\dots,Z_{m,n})$.
As can be visualized in Figure \ref{fig:frozenGTpatterns}, the interlacing conditions for $Z$ directly imply that 1) $z$ and $z'$ also satisfy the interlacing conditions, and 2) all entries of $z$ and $z'$ are bounded between $v$ and $u$.
Changing now notation and denoting
\[
\mu
:=(\mu_1,\dots,\mu_n)
:=(Z_{m,1},\dots,Z_{m,n}) \, ,
\]
it turns out that the shape of $z$ is $(\mu, v^{m-n})$ and the shape of $z'$ is $\mu$, where $\mu$ is an arbitrary $n$-partition such that $v\leq \mu_n\leq \dots \leq \mu_1 \leq u$.
This establishes bijection~\eqref{eq:sBijection1}.
In case $n\geq m$, we obtain a similar decomposition of $Z$ via a horizontal cut at level $n$ (from the top) instead of $m$.
Setting this time
\[
\mu := (\mu_1,\dots,\mu_m) := (Z_{n,n-m+1},\dots,Z_{n,n}) \, ,
\]
we see that $Z$ is in a bijective correspondence with a pair $(z,z')$ of Gelfand-Tsetlin patterns of height $n$ and $m$ respectively and shape $(u^{n-m}, \mu)$ and $\mu$ respectively, being $\mu$ an $m$-partition such that $v\leq \mu_m\leq \dots \leq \mu_1 \leq u$.
This proves~\eqref{eq:sBijection2}.
We also observe the following: when $n<m$, $z$ overlaps with the $v$-frozen part; when $n>m$, $z$ overlaps with the $u$-frozen part; when $n=m$, there is no overlap and the two patterns $z$ and $z'$ have the same shape $\mu$.

Identity~\eqref{eq:decompositionA} then follows from bijections~\eqref{eq:sBijection1} and~\eqref{eq:sBijection2}.
For the sake of simplicity, we will assume $v=0$ and show that
\begin{equation}
\label{eq:s=sumSSgeneral2}
\schur^{(n+m)}_{(u^{n} , 0^{m})}(x,y)
= \left[\prod_{i=1}^n x_i\right]^{u}
\sum_{\mu \subseteq u^{l}}
\schur^{(n)}_{(\mu, 0^{n-l})} (x^{-1}) \cdot
\schur^{(m)}_{(\mu, 0^{m-l})} (y) \, ,
\end{equation}
where $l:=\min(n,m)$.
The general case $v\leq u$ can be deduced from the latter by multiplying both sides of~\eqref{eq:s=sumSSgeneral2} by $[\prod_{i=1}^n x_i \prod_{i=1}^m y_i]^v$, applying~\eqref{eq:schurProp+} on the left-hand side, and finally replacing $u$ with $u-v$.

Assume first that $n\leq m$.
A Schur polynomial indexed by $(u^{n} , 0^{m})$ is, by Definition~\ref{def:schur}, a sum over Gelfand-Tsetlin patterns $Z$ of shape $(u^{n}, 0^{m})$.
Thanks to bijection~\eqref{eq:sBijection1}, for $v=0$, we can rewrite this by first summing over Gelfand-Tsetlin patterns $z$ and $z'$ of shape $(\mu, 0^{m-n})$ and $\mu$, respectively, for a fixed $m$-partition $\mu$, and then summing over all $\mu$ bounded above\footnote{There is no lower bound as $v=0$.} by $u$.
It is easy to see that, under the bijection, $\type(Z)_i = \type(z)_i$ for $1\leq i \leq m$ and $\type(Z)_{m+i} = u - \type(z')_{n-i+1}$ for $1\leq i \leq n$.
We thus obtain:
\begin{align*}
&\schur^{(n+m)}_{(u^{n} , 0^{m})}(y_1,\dots,y_m,x_n,\dots,x_1)
=\sum_{Z \in \GT^{(n+m)}_{(u^{n}, 0^{m})}}
\prod_{i=1}^m
y_i^{\type(Z)_{i}}
\prod_{i=1}^n x_{i}^{\type(Z)_{m+n-i+1}} \\
& \qquad \qquad =
\sum_{\mu \subseteq u^{n}}
\sum_{z' \in \GT^{(n)}_{\mu}}
\prod_{i=1}^{n}
x_i^{u-\type(z')_{i}}
\sum_{z \in \GT^{(m)}_{(\mu , 0^{m-n})}}
\prod_{i=1}^{m}
y_i^{\type(z)_{i}} \, \\
& \qquad \qquad = \left[\prod_{i=1}^n x_i\right]^{u}
\sum_{\mu \subseteq u^{n}}
\schur^{(n)}_{\mu} (x_1^{-1}, \dots, x_n^{-1}) \cdot
\schur^{(m)}_{(\mu , 0^{m-n})} (y_1, \dots, y_m) \, .
\end{align*}
Since Schur polynomials are symmetric, \eqref{eq:s=sumSSgeneral2} follows by reordering the variables $x_i$'s and $y_i$'s in the initial Schur polynomial indexed by $(u^{n} , 0^{m})$.

Assume now $n\geq m$.
Under the bijection $Z \longleftrightarrow (z,z')$ given by~\eqref{eq:sBijection2}, it turns out that $\type(Z)_i = \type(z)_i$ for $1\leq i \leq n$ and $\type(Z)_{n+i} = u - \type(z')_{m-i+1}$ for $1\leq i \leq m$.
Proceeding similarly as in the case $n\leq m$, we then obtain the identity
\[
\schur^{(n+m)}_{(u^{n} , 0^{m})}(x,y)
= \left[\prod_{i=1}^m y_i\right]^{u}
\sum_{\mu \subseteq u^{m}}
\schur^{(n)}_{(u^{n-m} , \mu)} (x) \cdot
\schur^{(m)}_{\mu} (y^{-1}) \, .
\]
Applying property~\eqref{eq:schurProp-} to the Schur polynomial in the $y$-variables on the right-hand side, we obtain:
\begin{equation}
\label{eq:s=sumSS_Okada}
\schur^{(n+m)}_{(u^{n} , 0^{m})}(x,y)
= \sum_{\mu \subseteq u^{m}}
\schur^{(n)}_{(u^{n-m} , \mu)} (x) \cdot
\schur^{(m)}_{(u-\mu_m,\dots,u-\mu_1)} (y) \, .
\end{equation}
We remark that~\eqref{eq:s=sumSS_Okada} corresponds to the identity of Okada, as it appears in~\cite[Theorem 2.1]{Oka98}.
We now elaborate it further by setting $\lambda_i := u-\mu_{m-i+1}$ for $1\leq i\leq m$ and summing over the new partition $\lambda$ thus defined:
\begin{align*}
\schur^{(n+m)}_{(u^{n} , 0^{m})}(x,y)
&= \sum_{\lambda \subseteq u^{m}}
\schur^{(n)}_{(u^{n-m} , u-\lambda_m,\dots,u-\lambda_1)} (x) \cdot
\schur^{(m)}_{\lambda} (y) \\
&= \left[\prod_{i=1}^n x_i\right]^{u} \sum_{\lambda \subseteq u^{m}}
\schur^{(n)}_{(\lambda, 0^{n-m})} (x^{-1}) \cdot
\schur^{(m)}_{\lambda} (y) \, ,
\end{align*}
where the latter equality follows from a further application of~\eqref{eq:schurProp-} to the Schur polynomial in the $x$-variables.
We conclude that~\eqref{eq:s=sumSSgeneral2} is true also for $n\geq m$.

\vskip 2mm

\noindent {\bf Alternative algebraic proof.}
As pointed out by an anonymous referee, Theorem~\ref{thm:decompositionA} has an alternative, purely algebraic proof, which we outline here for the reader's convenience.
As noted above, it is enough to prove the $v=0$ case.
Assume also, for the sake of simplicity, that $n\leq m$.
Using~\eqref{eq:schurProp-}, one is reduced to show that
\[
\schur^{(n+m)}_{u^{n}}(x,y)
= \sum_{\mu \subseteq u^{n}}
\schur^{(n)}_{u^n-\mu} (x) \cdot
\schur^{(m)}_{\mu} (y) \, .
\]
On the other hand, recall the classical decomposition formula for skew Schur functions $\schur_{\lambda/\nu}^{(n+m)}(x,y) = \sum_{\mu} \schur^{(n)}_{\lambda/\mu}(x) \cdot \schur^{(m)}_{\mu/\nu}(y)$, where it is meant that $\schur^{(n)}_{\lambda/\mu}=0$ whenever $\mu\nsubseteq \lambda$.
Using the latter with $\lambda=u^n$ and $\nu=0$, we have
\[
\schur^{(n+m)}_{u^n}(x,y) = \sum_{\mu} \schur^{(n)}_{u^n/\mu}(x) \cdot \schur^{(m)}_{\mu}(y) \, .
\]
Therefore, it remains to prove that $\schur^{(n)}_{u^n/\mu}(x) = \schur^{(n)}_{u^n-\mu} (x)$.
This can be easily shown, for example, by expanding the skew Schur function in terms of Littlewood-Richardson coefficients as $\schur^{(n)}_{u^n/\mu}(x)= \sum_{\nu} c^{u^n}_{\mu\nu} \schur^{(n)}_{\nu}(x)$ and noting that $c^{u^n}_{\mu\nu}=\delta_{\nu,u^n-\mu}$ (see e.g.~\cite[Eq.~(1)]{MY05}).
Alternatively, one may use, on both sides of the identity to be proven, the Jacobi-Trudi formula $\schur^{(n)}_{\lambda/\nu}(x)=\det(h_{\lambda_i-\nu_j-i+j}(x))_{1\leq i,j\leq n}$ (here $h_k(x)$ is the complete homogeneous symmetric polynomial of degree $k$ in the variables $x_1,\dots,x_n$).

\section{Identities for characters of (nearly) rectangular shape}
\label{sec:nearlyRectangular}

In this section we prove those identities stated in Section~\ref{sec:results} that express a rectangular shaped (interpolating) Schur polynomial as a bounded Littlewood sum of Schur polynomials of a \emph{different} type.
More specifically, we prove identities $B^{\antisymSq}_\beta = C^{\antisymSq}_\beta$ of Theorem~\ref{thm:antisymLPP}, $B^{\symSq}_{\alpha} = C^{\symSq}_{\alpha}$ of Theorem~\ref{thm:symLPP} and $B^{\doubleSymSq}_{\alpha,0} = C^{\doubleSymSq}_{\alpha,0}$ of Theorem~\ref{thm:doubleSymLPP_beta=0}.
Essentially, we provide a way to generalize known identities for rectangular shaped characters (see~\cite{Ste90a,Mac95,Oka98,Kra98}) to the interpolating Schur polynomials that we have introduced in Section~\ref{sec:transition}.
Our results can be proven using certain identities established by Krattenthaler~\cite{Kra98} for Schur polynomials of various types indexed by a ``nearly rectangular'' partition, i.e.\ a partition with rectangular shape except for the last row or column that might be shorter.

\subsection{Proof of $B^{\antisymSqProtect}_\beta = C^{\antisymSqProtect}_\beta$ of Theorem~\ref{thm:antisymLPP}}
\label{subsec:nearlyRectangularCB}

Let us fix a set of variables $p=(p_1,\dots,p_N)$ as in Theorem~\ref{thm:antisymLPP}.
The identity of Krattenthaler~\cite[Theorem~2]{Kra98} that we need for this proof is:
\begin{equation}
\label{eq:krattenthalerC}
\sum_{\substack{\lambda\subseteq (2u)^{N} \colon \\ \oddrows\lambda =k}} \schur^{(N)}_{ \lambda} (p)
= \left[ \prod_{i=1}^N p_i \right]^u \sp^{(2N)}_{(u^{N-k} , (u-1)^{k})} (p) \, ,
\end{equation}
for any non-negative integers $u$ and $k$ such that $k\leq N$.
In words, the sum on the left-hand side is taken over all $N$-partitions $\lambda$ bounded above by $2u$ and with exactly $k$ odd rows.

Assume first that, in Theorem~\ref{thm:antisymLPP}, $u$ is integer.
From~\eqref{eq:krattenthalerC} it follows that
\begin{align*}
B^{\antisymSq}_\beta
&=\sum_{\mu \subseteq (2u)^{N}} \beta^{\oddrows\mu} \cdot \schur^{(N)}_{\mu} (p)
= \sum_{k=0}^N \beta^k \sum_{\substack{\mu\subseteq (2u)^{N}, \\ \oddrows\mu =k}} \schur^{(N)}_{ \mu} (p) \\
&= \left[ \prod_{i=1}^N p_i \right]^u
\sum_{k=0}^N \beta^k
\sp^{(2N)}_{(u^{N-k} , (u-1)^{k})} (p)
= \left[ \prod_{i=1}^N p_i \right]^u
\sum_{k=0}^N \beta^k
\sum_{\substack{\epsilon \in\{0,1\}^N \colon \\ \abs{\epsilon}=k}} 
\sp^{(2N)}_{u^{N} - \epsilon} (p) \, .
\end{align*}
The latter equality follows from the fact that $(u^{N-k} , (u-1)^{k})$ is the only partition of the form $u^{N} - \epsilon$, where $\epsilon$ is a binary $N$-tuple with exactly $k$ ones and $N-k$ zeroes.
Using~\eqref{eq:CBtrans=sp1}, we conclude:
\[
B^{\antisymSq}_\beta
= \left[ \prod_{i=1}^N p_i \right]^u
\trans{u^{N}}(x;\beta)
= C^{\antisymSq}_\beta \, .
\]

Assume now $u\in \frac{1}{2} + \Z$.
We have
\[
B^{\antisymSq}_\beta
= \sum_{\mu \subseteq (2u)^{N}} \beta^{\oddrows\mu} \cdot \schur^{(N)}_{\mu} (p)
= \sum_{\substack{\nu \subseteq (2u-1)^{N} \colon \\ \oddrows\nu=0}}
\sum_{k=0}^N \beta^k
\sum_{\substack{\epsilon \in \{0,1\}^N \colon \\ \abs{\epsilon}=k}} \schur^{(N)}_{\nu+\epsilon}(p) \, .
\]
In the latter expression we have changed variables by setting $\epsilon_i := \mu_i \bmod 2$ and $\nu_i := \mu_i -\epsilon_i$ for $1\leq i\leq N$, thus obtaining an $N$-tuple $\epsilon\in\{0,1\}^N$ and an $N$-partition $\nu$ bounded above by $2u-1$ with even parts.
The convention for the right-hand side of the latter display is that, as usual, $\schur^{(N)}_{\nu+\epsilon}(p)$ vanishes if $\nu+\epsilon$ is not a partition.
The \emph{dual Pieri rule} (see e.g.~\cite{Mac95}) tells us that
\[
\sum_{\substack{\epsilon \in \{0,1\}^N \colon \\ \abs{\epsilon}=k}} \schur^{(N)}_{\nu+\epsilon}(p)
= \e^{(N)}_k(p) \cdot \schur^{(N)}_{\nu}(p) \, , \qquad\text{where}\qquad
\e^{(N)}_k(p) := \sum_{\substack{\epsilon\in\{0,1\}^N \colon \\ \abs{\epsilon}=k}} \prod_{i=1}^N p_i^{\epsilon_i}
\]
is the elementary symmetric polynomial of degree $k$ in $N$ variables.
Since $\e^{(N)}_k(p)$ is homogeneous of degree $k$, we have
\[
B^{\antisymSq}_\beta
= \sum_{\substack{\nu \subseteq (2u-1)^{N} \colon \\ \oddrows\nu=0}}
\sum_{k=0}^N
\e^{(N)}_k(\beta p) \schur^{(N)}_{\nu}(p)
= \left[\sum_{\epsilon\in\{0,1\}^N} \prod_{i=1}^N (\beta p_i)^{\epsilon_i} \right] \sum_{\substack{\nu \subseteq (2u-1)^{N} \colon \\ \oddrows\nu=0}}
\schur^{(N)}_{\nu}(p) \, .
\]
Since $u$ is a half-integer, we can now use~\eqref{eq:krattenthalerC} replacing $u$ with $u-\frac{1}{2}$ and setting $k=0$.
We thus obtain:
\[
B^{\antisymSq}_\beta
= \left[\prod_{i=1}^N (\beta p_i + 1)\right] \left[\prod_{i=1}^N p_i \right]^{u-\frac{1}{2}} \sp^{(2N)}_{(u-\frac{1}{2})^{N}}(x)
= \left[\prod_{i=1}^N p_i \right]^u
\trans{u^{N}}(p;\beta)
= C^{\antisymSq}_\beta \, ,
\]
thanks to~\eqref{eq:CBtrans=sp2}.

\subsection{Proof of $B^{\symSqProtect}_\alpha = C^{\symSqProtect}_\alpha$ of Theorem~\ref{thm:symLPP}}
\label{subsec:nearlyRectangularDB}

Let us fix a set of variables $p=(p_1,\dots,p_N)$ as in Theorem~\ref{thm:symLPP}.
We will use the following identity of Krattenthaler~\cite[Theorem~2]{Kra98}: for any $u\in\frac{1}{2}\Z_{\geq 0}$ and $k\in\Z_{\geq 0}$ such that $k\leq 2u$, we have
\begin{equation}
\label{eq:krattenthalerD}
\sum_{\substack{\lambda\subseteq (2u)^{N} \colon \\ \oddrows((2u)^N-\lambda)' =k}}
\!\!\!\!
\schur^{(N)}_{\lambda} (p)
= \left[ \prod_{i=1}^N p_i \right]^u \so^{(2N)}_{(u^{N-1} , u-k)} (p) \, .
\end{equation}
In words, the sum on the left-hand side is taken over all $N$-partitions $\lambda \subseteq (2u)^{N}$ such that the complement partition of $\lambda$ with respect to $(2u)^{N}$, i.e.\ $\mu:=(2u)^N-\lambda=(2u-\lambda_N, \dots, 2u-\lambda_1)$, has exactly $k$ odd columns.
It is easy to verify, by using~\eqref{eq:schurProp-} and replacing each $p_i$ with $p_i^{-1}$, that~\eqref{eq:krattenthalerD} is equivalent to:
\[
\sum_{\substack{\mu\subseteq (2u)^{N} \colon \\ \oddrows\mu' = k}} \schur^{(N)}_{\mu} (p)
= \left[ \prod_{i=1}^N p_i \right]^u \so^{(2N)}_{(u^{N-1} , u-k)} (p^{-1}) \, ,
\]
where $p^{-1} := (p_1^{-1},\dots, p_N^{-1})$.

Using the latter identity, along with~\eqref{eq:DBbranchingRectangular}, we obtain:
\begin{align*}
B^{\symSq}_\alpha
&= \sum_{\mu \subseteq (2u)^{N}}
\alpha^{\oddrows\mu'} \cdot
\schur^{(N)}_{\mu}(p)
= \sum_{k=0}^{2u} \alpha^k \sum_{\substack{\mu\subseteq (2u)^{N} \colon \\ \oddrows\mu' =k}} \schur^{(N)}_{ \mu} (p) \\
&= \left[ \prod_{i=1}^N p_i \right]^u
\sum_{k=0}^{2u} \alpha^k \cdot \so^{(2N)}_{(u^{N-1} , u-k )} (p^{-1})
= \left[ \prod_{i=1}^N p_i \right]^u
\transDB{u^{N}}(p^{-1}; \alpha)
= C^{\symSq}_\alpha \, .
\end{align*}

\subsection{Proof of $B^{\doubleSymSqProtect}_{\alpha,0} = C^{\doubleSymSqProtect}_{\alpha,0}$ of Theorem~\ref{thm:doubleSymLPP_beta=0}}
\label{subsec:nearlyRectangularA}

Let us fix a set of variables $p=(p_1,\dots,p_n)$ as in Theorem~\ref{thm:doubleSymLPP_beta=0}.
We will now use the following identity of Krattenthaler~\cite[Theorem~1]{Kra98}: for any non-negative integer $u$ and $k$ such that $k\leq u$, we have
\begin{equation}
\label{eq:krattenthalerA}
\sum_{\substack{\lambda \subseteq u^{n} \colon \\ \oddrows(u^n-\lambda)' =k}}
\!\!\!\!\!\!\!\!
\sp^{(2n)}_{\lambda}(p)
= \schur^{(2n)}_{(u^{n-1} , u-k , 0^{n})} (p,p^{-1}) \, ,
\end{equation}
where $p^{-1} := (p_1^{-1},\dots, p_n^{-1})$.
In words, the sum on the left-hand side is taken over all $n$-partitions $\lambda \subseteq u^{n}$ such that the complement partition $u^n-\lambda$ has exactly $k$ odd columns.

From~\eqref{eq:krattenthalerA} it follows that
\begin{equation}
\label{eq:nearlyRectangularA_1}
\begin{split}
\left[\prod_{i=1}^n p_i\right]^{-u}
B^{\doubleSymSq}_{\alpha,0}
&=\sum_{\lambda \subseteq u^{n}}
\alpha^{\oddrows(u^n-\lambda)'} \cdot
\sp^{(2n)}_{\lambda}(p) \\
&= \sum_{k=0}^u \alpha^k
\sum_{\substack{\lambda \subseteq u^{n} \colon \\ \oddrows(u^n-\lambda)' = k}}
\!\!\!\!\!\!\!\!
\sp^{(2n)}_{\lambda}(p)
= \sum_{k=0}^u \alpha^k \cdot
\schur^{(2n)}_{(u^{n-1} , u-k , 0^{n})} (p,p^{-1}) \, .
\end{split}
\end{equation}
On the other hand, let us recall the well-known branching rule for Schur polynomials (which is an immediate consequence of Definition~\ref{def:schur}):
\[
s^{(m)}_{\mu}(x_1,\dots,x_m) = \sum_{\nu\prec \mu} x_m^{\abs{\mu} - \abs{\nu}} \schur^{(m-1)}_{\nu}(x_1,\dots, x_{m-1}) \, ,
\]
where the sum is taken over all $(m-1)$-partitions $\nu$ that upwards interlace with $\mu$.
If we take $m=2n+1$ and $\mu:= (u^{n}, 0^{n+1})$, then all and only the $(2n)$-partitions that upwards interlace with $\mu$ are of the form $(u^{n-1}, k, 0^{n})$ for any integer $0\leq k\leq u$.
The branching rule and the symmetry of Schur polynomials then tell us that
\begin{equation}
\label{eq:nearlyRectangularA_2}
\begin{split}
\left[\prod_{i=1}^n p_i\right]^{-u}
C^{\doubleSymSq}_{\alpha,0}
&= \schur^{(2n+1)}_{(u^{n} , 0^{n+1})}(p,p^{-1},\alpha)
= \sum_{k=0}^u \alpha^{nu- ((n-1)u+k)} \cdot
\schur^{(2n)}_{(u^{n-1} , k , 0^{n})} (p,p^{-1}) \\
&= \sum_{k=0}^u \alpha^{u-k} \cdot
\schur^{(2n)}_{(u^{n-1} , k , 0^{n})} (p,p^{-1}) \, .
\end{split}
\end{equation}
Comparing~\eqref{eq:nearlyRectangularA_1} and~\eqref{eq:nearlyRectangularA_2}, we obtain $B^{\doubleSymSq}_{\alpha,0} = C^{\doubleSymSq}_{\alpha,0}$.

\section{Duality between Fredholm determinants and Pfaffians in random matrix theory}
\label{sec:duality}

Certain distributions from random matrix theory, which describe the fluctuations of the largest eigenvalue of an $N\times N$ random matrix drawn from a specific ensemble as $N\to\infty$, often possess dual expressions as Fredholm determinants on one hand and Fredholm Pfaffians on the other hand.
In this section we briefly describe how the combinatorial and algebraic structures described in the present work explain, \emph{already at a finite $N$ level}, such a duality for two random matrix distributions: the Tracy-Widom GOE and GSE laws.

\vskip 1mm

Let us start by briefly recalling the notions of a Fredholm determinant and a Fredholm Pfaffian.
Given a measure space $(\mathcal{X},\mu)$, any linear operator $K\colon L^2(\mathcal{X}) \to L^2(\mathcal{X})$ can be given in terms of its integral kernel $K(x,y)$ by
\[
(Kf)(x) := \int_{\mathcal{X}} K(x,y)f(y) \mu(\diff y) \, ,\qquad f\in L^2(\mathcal{X}) \, .
\]
The \emph{Fredholm determinant} of $K$ can then be defined through its series expansion:
\[
\det(I + K)_{L^2(\mathcal{X})}
:= 1+\sum_{n=1}^{\infty} \frac{1}{n!} \int_{\mathcal{X}^n}
\underset{1\leq i,j\leq n}{\det}(K(x_i,x_j)) \,
\mu(\diff x_1)\cdots \mu(\diff x_n) \, ,
\]
assuming the series converges.

The Pfaffian of a skew-symmetric matrix $A=(a_{i,j})_{1\leq i,j\leq 2n}$ is defined via the symmetric group expansion
\begin{align*}
\Pf(A) =\frac{1}{2^n n!} \sum_{\sigma \in S_{2n}} \sign(\sigma) \prod_{k=1}^n a_{\sigma(2k-1),\sigma(2k)}
\end{align*}
and it can be shown to be the square root of $\det(A)$.
Let now
\begin{align*}
J:= \delta(x-y)
\begin{pmatrix}
0 & 1 \\ -1 & 0
\end{pmatrix} \qquad \text{for } x,y\in\mathcal{X} \, ,
\end{align*}
where $\delta(\cdot)$ is the Dirac delta function.
Then the Fredholm Pfaffian of a skew-symmetric\footnote{This means that $K_{j,i}(y,x) = - K_{i,j}(x,y)$ for all $1\leq i,j\leq 2$ and $x,y\in\mathcal{X}$.} matrix-valued kernel 
\begin{align*}
K(x,y)=
\begin{pmatrix}
K_{11}(x,y) & K_{12}(x,y) \\
K_{21}(x,y) & K_{22}(x,y)
\end{pmatrix} \, , \qquad
\text{with } K_{i,j}\colon L^2(\mathcal{X}) \to L^2(\mathcal{X}) \, ,
\end{align*}
is defined as 
\[
\Pf(J + K)_{L^2(\mathcal{X})}
:= 1+\sum_{n=1}^{\infty} \frac{1}{n!} \int_{\mathcal{X}^n}
\underset{1\leq i,j\leq n}{\Pf}(K(x_i,x_j)) \,
\mu(\diff x_1)\cdots \mu(\diff x_n) \, ,
\]
assuming that the series converge.

We also recall two crucial identities that express integrals of determinantal functions as Pfaffians or determinants, and that will be useful in the following of this section.
Andr\'eief's identity (a generalization of the Cauchy-Binet identity, see~\cite{And86}) states that
\begin{equation}
\label{eq:CauchyBinet}
\int\limits_{x_1\leq\dots\leq x_n}
\underset{1\leq i,j\leq n}{\det}\!\left(f_j(x_i)\right) \underset{1\leq i,j\leq n}{\det}\!\left(g_j(x_i)\right) \,
\prod_{i=1}^n \nu(\diff x_i)
= \underset{1\leq i,j\leq n}{\det}\!\left( \int_{\R} f_i(x) g_j(x) \, \nu(\diff x) \right) ,
\end{equation}
where $\nu$ is a Borel measure on $\R$ and $f_1,\dots,f_n, g_1, \dots, g_n$ are integrable functions.
On the other hand, assuming for simplicity that $n$ is even, the de Bruijn identity~\cite{Bru55} states that
\begin{equation}
\label{eq:deBruijn}
\int\limits_{x_1\leq\dots\leq x_n}
\underset{1\leq i,j\leq n}{\det}\!\left(\phi_j(x_i)\right) \, \prod_{i=1}^n \nu(\diff x_i)
= \underset{1\leq i,j\leq n}{\Pf}\!\left( \int_{\R^2} \sign(y-x) \, \phi_i(x) \, \phi_j(y) \, \nu(\diff x) \, \nu(\diff y) \right) ,
\end{equation}
where $\nu$ is a Borel measure on $\R$ and $\phi_1,\dots,\phi_n$ are integrable functions.

\subsection{The Tracy-Widom GOE distribution}
\label{subsec:GOE}

The \emph{Gaussian Orthogonal Ensemble} (GOE) is the space of $N\times N$ real symmetric matrices $H$ endowed with the Gaussian probability density proportional to $\exp\{-\frac{N}{4} \Tr H^2\}$, which turns out to be invariant under conjugation with the orthogonal group.
The law of the largest eigenvalue of an $N\times N$ GOE matrix converges as $N\to\infty$, after suitable rescaling, to the so-called {\bf\emph{Tracy-Widom GOE distribution}}.
This random matrix model has been first studied in~\cite{TW96}.

The cumulative function of the Tracy-Widom GOE distribution admits the Fredholm Pfaffian expression~\cite{Fer04}:
\begin{equation}
\label{eq:GOE_Pfaffian}
F_1(s) = {\rm Pf}(J + K^{{\rm GOE}})_{L^2[s,\infty)} \, ,
\end{equation}
where, denoting by $\Ai(\cdot)$ the Airy function and $\Ai'(\cdot)$ its derivative, $K^{{\rm GOE}}$ is the $2\times 2$ matrix-valued kernel defined by
\begin{align*}
K_{11}^{\rm GOE}(x,y) &= \int_0^\infty \Ai(x+\gl) \Ai'(y+\gl)\diff \gl - \int_0^\infty \Ai(y+\gl) \Ai'(x+\gl) \diff \gl \, , \\
K_{12}^{\rm GOE}(x,y) &= -K_{21}^{\rm GOE}(x,y)= 
\int_0^\infty \Ai(x+\gl) \Ai(y+\gl) \diff\gl  +\frac{1}{2} \Ai(x) \int_0^\infty \Ai(y-\gl)\dd\gl \, ,
 \\
K_{22}^{\rm GOE}(x,y)&=
\frac{1}{4} \int_0^\infty \!\! \diff \gl \int_\gl^\infty \!\! \diff \mu \Ai(y-\mu) \Ai(x-\gl)
- \frac{1}{4}\int_0^\infty \!\! \diff \gl \int_\gl^\infty \!\! \diff \mu \Ai(x-\mu) \Ai(y-\gl) \, .
\end{align*}
Equivalent but slightly different Pfaffian expressions for the Tracy-Widom GOE distribution, as well as formulas in terms of Painlev\'e functions, also exist -- see e.g.~\cite{TW05,DG09}.

Remarkably, the Tracy-Widom GOE distribution also admits the following simpler Fredholm determinant expression:
\begin{align}
\label{eq:GOE_Sasamoto}
F_1(s) =\det(I-B)_{L^2[s,\infty)}
\qquad \text{with} \qquad
B(x,y)=\frac{1}{2} \Ai\Big(\frac{x+y}{2}\Big) \, .
\end{align}
Expression~\eqref{eq:GOE_Sasamoto} was originally discovered by Sasamoto~\cite{Sas05} via analysis of the Totally Asymmetric Simple Exclusion Process (TASEP).
A confirmation that~\eqref{eq:GOE_Sasamoto} agrees with previously known formulas for the Tracy-Widom GOE distribution was provided by
Ferrari and Spohn in~\cite{FS05} via a series of linear operator tricks.

\vskip 1mm

It is possible to recover both the Pfaffian and the determinantal expressions for the Tracy-Widom GOE distribution by rescaling the formulas provided in Corollary~\ref{coro:antisymLPP_0} for the distribution of the antidiagonally symmetric LPP time.
For convenience and in analogy with the asymptotic analysis carried out in~\cite{BZ19b}, we will work with exponentially distributed weights, instead of geometrically.
Namely, we will consider an array $\{\tilde{W}_{i,j}\colon 1\leq i,j\leq 2n\}$ symmetric about the antidiagonal and such that the weights $\tilde{W}_{i,j}$ with $i+j\leq 2n+1$ are independent and distributed as
\begin{equation}
\label{eq:antisymExp}
\P(\tilde{W}_{i,j} \in \diff x) =
(\rho_{2n-i+1} +\rho_j) \e^{-(\rho_{2n-i+1} +\rho_j) x} \, .
\end{equation}

Consider $A^{\antisymSq}_0 = B^{\antisymSq}_0$ in Corollary~\ref{coro:antisymLPP_0}, with $N=2n$.
After replacing the Schur polynomial in $B^{\antisymSq}_0$ with its Weyl character formula~\eqref{eq:schurWeyl}, expressing the denominator's Vandermonde determinant in its closed form and setting $\mu_i := 2 \lambda_i$ for $1\leq i\leq 2n$, the identity reads as
\begin{equation}
\label{eq:antisymLPP_Geom}
 \P\left(L^{\antisymSq}_{0}(2n,2n) \leq 2u\right) 
= \frac{\prod_{1\leq i\leq j\leq 2n} (1-p_i p_j)}
{\prod_{1\leq i< j\leq 2n} (p_i - p_j)}
\sum_{\substack{ \lambda \in \Z^{2n} \colon \\ 0 \leq \lambda_{2n} \leq \dots \leq \lambda_1 \leq u}}
\underset{1\leq i,j\leq 2n}{\det}\big( p_j^{2\lambda_i + 2n-i} \big)
\, .
\end{equation}
The LPP time $L^{\antisymSq}_{0}(2n,2n)$ is taken on geometric weights $W_{i,j}$'s with distribution given by~\eqref{eq:antisymWeights} (for $N=2n$ and $\beta=0$).
Scaling the parameters as $p_i:= \e^{- \epsilon \rho_i}$ for $\epsilon>0$, the variables $\epsilon W_{i,j}$ will converge in law, as $\epsilon\downarrow 0$, to the exponential weights in~\eqref{eq:antisymExp}.
To obtain the analog of~\eqref{eq:antisymLPP_Geom} for the LPP model with exponential weights, it then suffices to set also $v:= u/\epsilon$ and take the limit as $\epsilon \downarrow 0$.
By Riemann sum approximation, the sum on the right-hand side of~\eqref{eq:antisymLPP_Geom} will then converge to a continuous integral, yielding:
\begin{align*}
\P\left(\tilde{L}^{\antisymSq}_{0}(2n,2n) \leq 2v\right) 
&= \frac{\prod_{1\leq i\leq j\leq 2n} (\rho_i + \rho_j)}
{\prod_{1\leq i< j\leq 2n} (\rho_j - \rho_i)}
\int\limits_{0 \leq x_{2n} \leq \dots \leq x_{1} \leq v}
\underset{1\leq i,j\leq 2n}{\det}\big( \e^{-2\rho_j x_i} \big) \prod_{i=1}^{2n} \diff x_i \, ,
\end{align*}
where $\tilde{L}^{\antisymSq}_{0}(2n,2n)$ is the antidiagonally symmetric LPP time from $(1,1)$ to $(2n,2n)$ with weights as in~\eqref{eq:antisymExp}.
Recognizing in the latter expression the \emph{Schur Pfaffian}
\[
\underset{1\leq i,j\leq 2n}{\Pf}\left( \frac{\rho_i-\rho_j}{\rho_i + \rho_j} \right)
= \prod_{1\leq i < j \leq 2n} \frac{\rho_i-\rho_j}{\rho_i + \rho_j}
\]
and applying the de Bruijn identity~\eqref{eq:deBruijn} and a basic property of Pfaffians, we obtain:
\begin{align}
\label{eq:antisymLPP_Pf}
\P\left(\tilde{L}^{\antisymSq}_{0}(2n,2n) \leq 2v\right)
&= \frac{ \displaystyle \underset{1\leq i,j\leq 2n}{\Pf}\left( 4 \rho_i \rho_j \int_0^v \diff x \int_0^v \diff y \sign(y-x) \e^{-2\rho_i x - 2\rho_j y} \right)}
{\displaystyle \underset{1\leq i,j\leq 2n}{\Pf}\left( \frac{\rho_i-\rho_j}{\rho_i + \rho_j} \right)}
\, .
\end{align}
The Fredholm Pfaffian expression given in~\eqref{eq:GOE_Pfaffian} for the Tracy-Widom GOE distribution can be derived as a scaling limit of the latter identity, after taking the weights to be exponential i.i.d.\ variables and setting $v:=v_n:= \sff n + \sigma n^{1/3}$ for suitable constants $\sff$ and $\sigma$.
The asymptotic analysis of a Poissonized version of the antidiagonally symmetric LPP model, recovering~\eqref{eq:GOE_Pfaffian}, was carried out by Ferrari~\cite{Fer04}.
A previous asymptotic analysis via orthogonal polynomials and Riemann-Hilbert problems, recovering the expression of the Tracy-Widom GOE distribution in terms of Painlev\'e functions, was performed by Baik and Rains~\cite{BR01b}.

On the other hand, consider identity $A^{\antisymSq}_0 = D^{\antisymSq}_0$ in Corollary~\ref{coro:antisymLPP_0}, for $N=2n$.
Using the Weyl character formula~\eqref{eq:spWeyl} for symplectic characters and taking -- in the same fashion as before -- the exponential limit, one obtains:
\begin{align*}
&\P\left( \tilde{L}^{\antisymSq}_{0}(2n,2n) \leq 2v \right)
= \frac{\prod_{1\leq i,j\leq n} (\rho_i + \rho_{n+j})}
{\prod_{1\leq i<j\leq n} (\rho_i - \rho_j) (\rho_{n+i} - \rho_{n+j})} 
\prod_{i=1}^{2n} \e^{-v \rho_i } \\*
&\qquad \times \int\limits_{0\leq x_n\leq \cdots \leq x_1\leq v}  
\underset{1\leq i,j\leq n}{\det}\big(e^{\rho_j x_i} - e^{-\rho_j x_i}\big)
\underset{1\leq i,j\leq n}{\det}\big(e^{\rho_{n+j} x_i} - e^{-\rho_{n+j} x_i}\big)
\prod_{i=1}^n \diff x_i \, .
\end{align*}
Recognizing in the latter expression the \emph{Cauchy determinant}
\begin{equation}
\label{eq:CauchyDet}
\underset{1\leq i,j\leq n}{\det}\left( \frac{1}{\rho_i + \rho_{n+j}} \right)
= \frac{\prod_{1\leq i<j\leq n} (\rho_i - \rho_j) (\rho_{n+i} - \rho_{n+j})}
{\prod_{1\leq i,j\leq n} (\rho_i + \rho_{n+j})}
\end{equation}
and applying Andr\'eief's identity~\eqref{eq:CauchyBinet} and the multilinearity of determinants, we then obtain:
\begin{align}
\label{eq:antisymLPP_det}
\P\left( \tilde{L}^{\antisymSq}_{0}(2n,2n) \leq 2v \right)
= \frac{ \displaystyle \underset{1\leq i,j\leq n}{\det} \left( e^{-v (\rho_i + \rho_{n+j})}
\int_0^v
\big(e^{\rho_i x} - e^{-\rho_i x}\big)
\big(e^{\rho_{n+j} x} - e^{-\rho_{n+j} x}\big) \diff x \right)}
{\displaystyle \underset{1\leq i,j\leq n}{\det} \left( \frac{1}{\rho_i + \rho_{n+j}} \right)} \, .
\end{align}
The latter identity was shown in~\cite{BZ19b} to directly lead, in the scaling limit, to the Fredholm determinant formula~\eqref{eq:GOE_Sasamoto} for the Tracy-Widom GOE distribution.
This was possible by means of a fairly standard procedure to turn a ratio of determinants like~\eqref{eq:antisymLPP_det} into a Fredholm determinant (see e.g.~\cite[Theorem~2.1]{BZ19b} or~\cite{BG16}) and a suitable asymptotic analysis via steepest descent.

From the discussion above we may conclude that comparing the Pfaffian identity~\eqref{eq:antisymLPP_Pf} and the determinantal identity~\eqref{eq:antisymLPP_det} for the LPP model provides an explanation, at a finite $n$ level, of the duality between the Fredholm Pfaffian and Fredholm determinant expressions of the Tracy-Widom GOE distribution.

\subsection{The Tracy-Widom GSE distribution}
\label{subsec:GSE}

The \emph{Gaussian Symplectic Ensemble} (GSE) is the space of $N\times N$ Hermitian quaternionic matrices $H$ endowed with the Gaussian probability density proportional to $\exp\{-N \Tr H^2\}$, which is invariant under conjugation with the symplectic group.
The law of the largest eigenvalue of an $N\times N$ GSE matrix converges as $N\to\infty$, after suitable rescaling, to the so-called {\bf\emph{Tracy-Widom GSE distribution}}~\cite{TW96}.

The cumulative function of the Tracy-Widom GSE distribution has the Fredholm Pfaffian expression~\cite{BBCS18}\footnote{Other expressions in terms of the square root of a Fredholm determinant with $2\times 2$ matrix-valued kernel~\cite{TW05, DG09} or with scalar kernel~\cite{KD18} also exist.}
\begin{align}
\label{eq:GSE_Pfaffian}
F_4(s) = \Pf(J - K^{{\rm GSE}})_{L^2[s,\infty)} \, ,
\end{align}
where the kernel $K^{{\rm GSE}}$ is given as a $2\times 2$ matrix kernel with entries
\begin{align*}
K_{11}^{\rm GSE}(x,y)
&=-\frac{1}{2}\int_x^\infty \diff \gl \int_0^\infty \diff\mu \Ai(\gl+\mu) \Ai(y+\mu) +\frac{1}{4} \int_x^\infty \Ai(\gl) \dd\gl \int_y^\infty \Ai(\mu)\diff \mu \, ,  \\
K_{12}^{\rm GSE}(x,y)&= -K_{21}^{\rm GSE}(y,x)= \frac{1}{2} \int_0^\infty \Ai(x+\gl) \Ai(y+\gl) \diff \gl -\frac{1}{4} \Ai(y) \int_x^\infty \Ai(\gl)\diff\gl \, , \\
K_{22}^{\rm GSE}(x,y)&=\frac{1}{2} \frac{\partial}{\partial y} \int_0^\infty \Ai(x+\gl) \Ai(y+\gl) \diff \gl +\frac{1}{4} \Ai(x) \Ai(y) \, .
\end{align*}
On the other hand, $F_4$ also admits the simpler Fredholm determinant expression
\begin{align}
\label{eq:GSE_det}
F_4(s) = \frac{1}{2}\Big[ \det(I-B)_{L^2[\sqrt{2}s,\infty)} + \det(I+B)_{L^2[\sqrt{2}s,\infty)} \Big] \, .
\end{align} 
The latter has been established in~\cite{FS05} as a direct consequence of the Tracy-Widom GOE Fredholm determinant formula~\eqref{eq:GOE_Sasamoto} and certain identities linking all three Tracy-Widom distributions for the Gaussian random matrix ensembles (orthogonal, unitary and symplectic).

\vskip 1mm

The Fredholm Pfaffian expression~\eqref{eq:GSE_Pfaffian} of the Tracy-Widom GSE distribution can be derived as a scaling limit of the bounded Littlewood identity for the distribution of the diagonally symmetric LPP model with zero weights on the diagonal, i.e.\ $A^{\symSq}_{0} = B^{\symSq}_{0}$ in Corollary~\ref{coro:symLPP_0}.

On the other hand, we now wish to sketch  how the bounded Cauchy identity $A^{\symSq}_{0} = D^{\symSq}_{0}$, for $N=2n$, provides a direct route to the Fredholm determinant expression~\eqref{eq:GSE_det}.
We first rewrite this identity using the Weyl formula~\eqref{eq:soEvenWeyl} for even orthogonal characters (with the denominator determinant expressed in its closed form of Vandermonde type, see~\cite{FH91}):
\begin{align*}
\P\left(L^{\symSq}_{0}(2n,2n) \leq 2u\right)
= & \prod_{1 \leq i<j\leq 2n} (1-p_i p_j)
\left[ \prod_{i=1}^{2n} p_i \right]^u \\*
\times \sum_{\substack{ \lambda \in \Z^n \colon \\ 0\leq \abs{\lambda_n} \leq \lambda_{n-1} \leq \dots \leq \lambda_1 \leq u}} 
&\frac{\underset{1\leq i,j\leq n}{\det}\big( p_j^{-(\lambda_i + n-i)} + p_j^{\lambda_i + n-i} \big) + \underset{1\leq i,j\leq n}{\det}\big( p_j^{-(\lambda_i + n-i)} - p_j^{\lambda_i + n-i} \big)}
{2\prod_{1\leq i< j\leq n} (p_i + p_i^{-1} - p_j - p_j^{-1})} \\*
&\times \frac{\underset{1\leq i,j\leq n}{\det}\big( p_{n+j}^{-(\lambda_i + n-i)} + p_{n+j}^{\lambda_i + n-i} \big) + \underset{1\leq i,j\leq n}{\det}\big( p_{n+j}^{-(\lambda_i + n-i)} - p_{n+j}^{\lambda_i + n-i} \big)}
{2\prod_{1\leq i< j\leq n} (p_{n+i} + p_{n+i}^{-1} - p_{n+j} - p_{n+j}^{-1})} \, .
\end{align*}
For convenience we again consider, as done in Subsection~\ref{subsec:GSE}, the exponential limit of the latter expression.
Let $\tilde{L}^{\symSq}_{0}(2n,2n)$ be the cumulative function of the LPP time on a symmetric $(2n)\times (2n)$ array with exponential weights distributed as
\[
\P(\tilde{W}_{i,j} \in \diff x) =
(\rho_{i} +\rho_j) e^{-(\rho_{i} +\rho_j) x}
\]
for $1\leq i<j\leq 2n$ and $W_{i,i}=0$ for $1\leq i\leq 2n$.
Then, setting $p_i:=e^{-\epsilon \rho_i}$ and $u:=v/\epsilon$ and then passing to the limit as $\epsilon\downarrow 0$ in the formula above for geometric LPP, via a Riemann sum approximation we obtain:
\begin{align*}
&\P\left(\tilde{L}^{\symSq}_{0}(2n,2n) \leq 2v\right) 
= \frac{\prod_{1\leq i,j\leq n} (\rho_i + \rho_{n+j})}
{4\prod_{1\leq i<j\leq n} (\rho_i - \rho_j) (\rho_{n+i} - \rho_{n+j})}
\prod_{i=1}^{2n} \e^{-v \rho_i}  \\*
&\qquad\qquad\times \int\limits_{\abs{x_n}\leq x_{n-1}\leq \cdots \leq x_1\leq v} 
\left[\underset{1\leq i,j\leq n}{\det}( \e^{\rho_j x_i} + \e^{-\rho_j x_i} ) + \underset{1\leq i,j\leq n}{\det}(\e^{\rho_j x_i} - \e^{-\rho_j x_i}) \right] \\*
&\qquad\qquad\times \left[\underset{1\leq i,j\leq n}{\det}( \e^{\rho_{n+j} x_i} + \e^{-\rho_{n+j} x_i} ) + \underset{1\leq i,j\leq n}{\det}( \e^{\rho_{n+j} x_i} - \e^{-\rho_{n+j} x_i}) \right]
\prod_{i=1}^n \diff x_i \, .
\end{align*}
We again recognize the Cauchy determinant~\eqref{eq:CauchyDet} in the prefactor.
Moreover, by standard observations about the even and the odd part, with respect to $x_n$, of the integrand, the integral above can be reduced to a sum of two integrals over the domain $\{ 0\leq x_n \leq \dots \leq x_1\leq v\}$.
Applying Andr\'eief's identity~\eqref{eq:CauchyBinet} to such two integrals, we finally obtain:
\[
\begin{split}
\P\left(\tilde{L}^{\symSq}_{0}(2n,2n) \leq 2v\right) 
= & \frac{1}{2} \left[ \frac{ \displaystyle \underset{1\leq i,j\leq n}{\det} \left( e^{-v (\rho_i + \rho_{n+j})}
\int_0^v
\big(e^{\rho_i x} - e^{-\rho_i x}\big)
\big(e^{\rho_{n+j} x} - e^{-\rho_{n+j} x}\big) \diff x \right)}
{\displaystyle \underset{1\leq i,j\leq n}{\det} \left( \frac{1}{\rho_i + \rho_{n+j}} \right)} \right. \\
&\left. + \frac{ \displaystyle \underset{1\leq i,j\leq n}{\det} \left( e^{-v (\rho_i + \rho_{n+j})}
\int_0^v
\big(e^{\rho_i x} + e^{-\rho_i x}\big)
\big(e^{\rho_{n+j} x} + e^{-\rho_{n+j} x}\big) \diff x \right)}
{\displaystyle \underset{1\leq i,j\leq n}{\det} \left( \frac{1}{\rho_i + \rho_{n+j}} \right)} \right]
 \, .
 \end{split}
\]
The first summand in the above formula is exactly what appears in the formula
\eqref{eq:antisymLPP_det}, thus giving (in the limit $n\to\infty$ after the appropriate scaling of $v$) the first Fredholm determinant in~\eqref{eq:GSE_det}.
The second summand only differs by a sign from the first one: following exactly the same procedure in the asymptotic analysis as the one carried out in~\cite{BZ19b} leads to the second term in~\eqref{eq:GSE_det}.

\vskip 4mm

{\bf Acknowledgments.} 
We would like to thank D.~Betea, A.~Borodin, V. Gorin, S.~Okada, R.~Proctor, E.~Rains and S.~O.~Warnaar for helpful discussions or comments on this work.
We are very grateful to the anonymous referees for carefully reading our manuscript and making constructive comments.


\vskip 4mm

\begin{thebibliography}{BBCS18}

\bibitem[And86]{And86}
C. Andr\'eief.
\newblock Note sur une relation entre les int\'egrales d\'efinies des produits des fonctions.
\newblock {\em M\'em. Soc. Sci. Phys. Nat. Bordeaux}, (3)2:1--14, 1886.

\bibitem[BBCS18]{BBCS18}
J. Baik, G. Barraquand, I. Corwin, T. Suidan.
\newblock Pfaffian Schur processes and last passage percolation in a half-quadrant.
\newblock {\em Ann. Prob.}, 46(6):3015--3089, 2018.

\bibitem[BDJ99]{BDJ99}
J. Baik, P. Deift, K. Johansson.
\newblock On the distribution of the length of the longest increasing subsequence of random permutations.
\newblock {\em J. Amer. Math. Soc.}, 12(4):1119--1178, 1999.

\bibitem[BDS16]{BDS16}
J. Baik, P. Deift, T. Suidan.
\newblock {\em Combinatorics and random matrix theory}, Graduate Studies in Mathematics, vol. 172, Amer. Math. Soc., 2016.

\bibitem[BR01a]{BR01a}
J. Baik and E. M. Rains.
\newblock {A}lgebraic aspects of increasing subsequences.
\newblock {\em Duke Math. J.}, 109(1):1--65, 2001.

\bibitem[BR01b]{BR01b}
J. Baik and E. M. Rains. 
\newblock The asymptotics of monotone subsequences of involutions.
\newblock {\em Duke Math. J.}, 109(2):205--281, 2001.

\bibitem[Bis18]{Bis18}
E. Bisi.
\newblock {\em Random polymers via orthogonal {W}hittaker and symplectic {S}chur functions}.
\newblock PhD thesis, University of Warwick, 2018, \url{http://wrap.warwick.ac.uk/121448/}.

\bibitem[BZ19a]{BZ19a}
E. Bisi and N. Zygouras.
\newblock Point-to-line polymers and orthogonal {W}hittaker functions.
\newblock {\em Trans. Amer. Math. Soc.}, 371(12):8339–8379, 2019.

\bibitem[BZ19b]{BZ19b}
E. Bisi and N. Zygouras.
\newblock {GOE} and {A}iry{$_{2\to 1}$} marginal distribution via symplectic {S}chur functions.
\newblock In {\em Probability and Analysis in Interacting Physical Systems}, Springer Proceedings in Mathematics \& Statistics, vol. 283, Springer, 2019.

\bibitem[BG16]{BG16}
A. Borodin and V. Gorin.
\newblock Lectures on integrable probability.
\newblock In {\em Probability and statistical physics in St. Petersburg} (V. Sidoravicius and S. Smirnov, eds.), Proceedings of Symposia in Pure Mathematics, vol. 91, Amer. Math. Soc., 2016, pp. 155--214.

\bibitem[Bru55]{Bru55}
N. G. de Bruijn.
\newblock On some multiple integrals involving determinants.
\newblock {\em J. Indian Math. Soc. New Series}, 19:133--151, 1955.

\bibitem[CS13]{CS13}
P. S. Campbell and A. Stokke.
\newblock On the orthogonal tableaux of Koike and Terada.
\newblock {\em Ann. Comb.}, 17:443--453, 2013.

\bibitem[DG09]{DG09}
P. Deift and D. Gioev.
{\em Random matrix theory: invariant ensembles and universality}, Courant Lecture notes, vol. 18, Amer. Math. Soc., 2009.

\bibitem[Fer04]{Fer04}
P. Ferrari.
\newblock {P}olynuclear growth on a flat substrate and edge scaling of {GOE} eigenvalues.
\newblock {\em Comm. Math. Phys.}, 252(1):77--109, 2004.

\bibitem[FS05]{FS05}
P. Ferrari and H. Spohn.  
\newblock A determinantal formula for the GOE Tracy-Widom distribution.
\newblock {\em J. Phys. A}, 38(33):L557, 2005.

\bibitem[Fom95]{Fom95}
S. Fomin.
\newblock Schur operators and {K}nuth correspondences.
\newblock {\em J. Combin. Theory Ser. A}, 72(2):277--292, 1995.

\bibitem[FR07]{FR07}
P. J. Forrester and E. M. Rains.
\newblock Symmetrized models of last passage percolation and non-intersecting lattice paths.
\newblock {\em J. Stat. Phys.}, 129(5-6):833--855, 2007.

\bibitem[FK97]{FK97}
M. Fulmek and C. Krattenthaler.
\newblock Lattice path proofs for determinantal formulas for symplectic and orthogonal characters.
\newblock {\em J. Combin. Theory Ser. A}, 77(1):3--50, 1997.

\bibitem[FH91]{FH91}
W. Fulton and J. Harris.
\newblock {\em Representation theory. A first course}, Graduate Texts in Mathematics, vol. 129, Springer-Verlag, 1991.

\bibitem[GT50]{GT50}
I. M. Gelfand and M. L. Tsetlin.
\newblock Finite-dimensional representations of groups of orthogonal matrices.
\newblock {\em Dokl. Akad. Nauk SSSR}, 71:1017--1020, 1950 (in Russian). English transl. in: I. M. Gelfand, {\em Collected papers}, Vol. II, Springer-Verlag, Berlin, 1988.

\bibitem[Gre74]{Gre74}
C.~Greene.
\newblock An extension of {S}chensted's theorem.
\newblock {\em Adv. Math.}, 14(2):254--265, 1974.

\bibitem[Joh00]{Joh00}
K. Johansson.
\newblock Shape fluctuations and random matrices.
\newblock {\em Comm. Math. Phys.}, 209(2):437--476, 2000.

\bibitem[Kin76]{Kin76}
R. C. King.
\newblock Weight multiplicities for the classical groups.
\newblock In {\em Group Theoretical Methods in Physics} (A. Janner, T. Janssen and M. Boon, eds.), Lecture Notes in Physics, vol. 50, Springer Berlin Heidelberg, 1976, pp. 490--499.

\bibitem[KE83]{KE83}
R. C. King and N. G. I. El-Sharkaway.
\newblock Standard young tableaux and weight multiplicities of the classical Lie groups.
\newblock {\em J. Phys. A}, 16(14):3153--3177, 1983.

\bibitem[Kir01]{Kir01}
A. N. Kirillov.
\newblock Introduction to tropical combinatorics.
\newblock In {\em Physics and Combinatorics. Proc. Nagoya 2000 2nd Internat. Workshop} (A. N. Kirillov and N. Liskova, eds.), World Scientific, Singapore, 2001, pp. 82--150.

\bibitem[Knu70]{Knu70}
D. E. Knuth.
\newblock {P}ermutations, matrices, and generalized {Y}oung tableaux.
\newblock {\em Pacific J. Math.}, 34:709--727, 1970.

\bibitem[KT87]{KT87}
K. Koike and I. Terada.
\newblock Young-diagrammatic methods for the representation theory of the classical groups of type $B_n$, $C_n$, $D_n$.
\newblock {\em J. Algebra}, 107(2):466--511, 1987.

\bibitem[KT90]{KT90}
K. Koike and I. Terada.
\newblock Young diagrammatic methods for the restriction of representations of complex classical Lie groups to reductive subgroups of maximal rank.
\newblock {\em Adv. Math.}, 79(1):104--135, 1990.

\bibitem[Koo92]{Koo92}
T. H. Koornwinder.
\newblock Askey–Wilson polynomials for root systems of type BC.
\newblock In: {\em Hypergeometric functions on domains of positivity, Jack polynomials, and applications} (D. S. P. Richards, ed.), Contemp. Math., vol. 138, American Mathematical Society, Providence, 1992, pp. 189--204. 

\bibitem[KD18]{KD18}
A. Krajenbrink and P. Le Doussal.
\newblock Large fluctuations of the KPZ equation in a half-space. 
\newblock {\em SciPost Physics}, 5(4), 032, 2018.

\bibitem[Kra98]{Kra98}
C. Krattenthaler.
\newblock Identities for classical group characters of nearly rectangular shape.
\newblock {\em J. Algebra}, 209(1):1--64, 1998.

\bibitem[Kra06]{Kra06}
C. Krattenthaler.
\newblock {G}rowth diagrams, and increasing and decreasing chains in fillings of {F}errers shapes.
\newblock {\em Adv. in Appl. Math.}, 37(3):404--431, 2006.

\bibitem[LRW20]{LRW20}
C.-h. Lee, E. M. Rains, S. O. Warnaar.
\newblock An elliptic hypergeometric function approach to branching rules.
\newblock {\em SIGMA}, 16, Art. No. 142, 2020.

\bibitem[Lit50]{Lit50}

D. E. Littlewood.
\newblock {\em The theory of group characters and matrix representations of groups}, Second Edition,
Oxford University Press, 1950.

\bibitem[Mac95]{Mac95}
I. G. Macdonald.
\newblock {\em {S}ymmetric functions and {H}all polynomials}, Second Edition, Oxford Mathematical Monographs, Oxford University Press, 1995.

\bibitem[MY05]{MY05}
H. Mizukawa, H.-F. Yamada.
\newblock Rectangular Schur functions and the basic representation of affine Lie algebras.
\newblock {\em Discrete Math.},
298(1–3):285--300, 2005.

\bibitem[NZ17]{NZ17}
V. L. Nguyen and N. Zygouras. 
\newblock Variants of geometric RSK, geometric PNG, and the multipoint distribution of the log-gamma polymer.
\newblock {\em Int. Math. Res. Notices}, 2017(15):4732--4795, 2017.
 
\bibitem[Oka98]{Oka98}
S. Okada,
\newblock Applications of minor summation formulas to rectangular-shaped representations of classical groups.
\newblock {\em J. Algebra}, 205(2):337--367, 1998.

\bibitem[Oka19]{Oka19}
S. Okada.
\newblock A bialternant formula for odd symplectic characters and its application.
\newblock {\em Josai Mathematical Monographs}, 12:99--116, 2020.

\bibitem[Pro88]{Pro88}
R. A. Proctor.
\newblock Odd symplectic groups.
\newblock {\em Invent. Math.}, 92(2):307--332, 1988.

\bibitem[Pro89a]{Pro89a}
R. A. Proctor.
\newblock Equivalence of the combinatorial and the classical definitions of {Schur} functions.
\newblock {\em J. Combin. Theory Ser. A}, 51(1):135--137, 1989.

\bibitem[Pro89b]{Pro89b}
R. A. Proctor.
\newblock Interconnections between orthogonal and symplectic characters.
\newblock In {\em Invariant theory} (R. Fossum, W. Haboush, M. Hochster, V. Lakshmibai, eds.), Contemp. Math., vol. 88, Amer. Math. Soc., Providence, RI, 1989, pp. 145--162.

\bibitem[Pro93]{Pro93}
R. A. Proctor.
\newblock A bideterminant proof of a product identity for plane partitions with symmetries.
\newblock {\em J. Stat. Plan. Inference}, 34(2):239--250, 1993.

\bibitem[Pro94]{Pro94}
R. A. Proctor.
\newblock {Y}oung tableaux, {G}elfand patterns, and branching rules for classical groups.
\newblock {\em J. Algebra}, 164(2):299--360, 1994.

\bibitem[Rai00]{Rai00}
E. M. Rains.
\newblock {C}orrelation functions for symmetrized increasing subsequences. arXiv:0006097, 2000.

\bibitem[Rai05]{Rai05}
E. M. Rains.
\newblock $BC_n$-symmetric polynomials.
{\em Transform. groups}, 10(1):63--132, 2005.

\bibitem[Rai10]{Rai10}
E. M.  Rains.
Transformations of elliptic hypergeometric integrals. 
{\em Ann. Math.}, 171(1):169--243, 2010.

\bibitem[Rai12]{Rai12}
E. M. Rains.
\newblock Elliptic {L}ittlewood identities.
\newblock {\em J. Combin. Theory Ser. A}, 119(7):1558--1609, 2012.

\bibitem[RW21]{RW21}
E. M. Rains and S. O. Warnaar.
\newblock Bounded {L}ittlewood identities.
\newblock {\em Mem. Amer. Math. Soc.}, 270, No 1317, 2021.

\bibitem[Sas05]{Sas05}
T. Sasamoto.
\newblock Spatial correlations of the 1{D} {KPZ} surface on a flat substrate.
\newblock {\em J. Phys. A}, 38(33):L549--L556, 2005.

\bibitem[Ste90a]{Ste90a}
J. R. Stembridge.
\newblock Nonintersecting paths, {P}faffians, and plane partitions.
\newblock {\em Adv. Math.}, 83(1):96--131, 1990.

\bibitem[Ste90b]{Ste90b}
J. R. Stembridge.
\newblock Hall-{L}ittlewood functions, plane partitions, and the {R}ogers-{R}amanujan identities.
\newblock {\em Trans. Amer. Math. Soc.}, 319(2):469--498, 1990.

\bibitem[Ste01]{Ste01}
J. R. Stembridge.
\newblock Multiplicity-free products of Schur functions.
\newblock {\em Ann. Comb.}, 5(2):113-121, 2001.

\bibitem[Sun90a]{Sun90a}
S. Sundaram.
\newblock Orthogonal tableaux and an insertion algorithm for $\SO(2n + 1)$.
\newblock {\em J. Combin. Theory Ser. A}, 53(2):239--256, 1990.

\bibitem[Sun90b]{Sun90b}
S. Sundaram.
\newblock Tableaux in the representation theory of the classical {L}ie groups.
\newblock In {\em Invariant theory and tableaux} (D. Stanton, ed.), IMA Vol. Math. Appl., vol. 19, Springer-Verlag, New York, 1990, pp. 191--225.

\bibitem[TW94]{TW94}
C. A. Tracy and H. Widom.
\newblock Level-spacing distributions and the {A}iry kernel.
\newblock {\em Comm. Math. Phys.}, 159(1):151--174, 1994.

\bibitem[TW96]{TW96}
C. A. Tracy and H. Widom.
\newblock On orthogonal and symplectic matrix ensembles.
\newblock {\em Comm. Math. Phys.}, 177:727--754, 1996.

\bibitem[TW05]{TW05}
C. A. Tracy and H. Widom.
\newblock {M}atrix kernels for the Gaussian orthogonal and symplectic ensembles.
\newblock {\em Ann. Inst. Fourier}, 55(6):2197--2007, 2005.

\bibitem[Ven15]{Ven15}
V. Venkateswaran.
\newblock Symmetric and nonsymmetric Koornwinder polynomials in the $q\to 0$ limit.
\newblock {\em J. Algebraic Combin.}, 42(2):331--364, 2015.

\bibitem[Zyg18]{Zyg18}
N. Zygouras.
\newblock Some algebraic structures in the KPZ universality.
\newblock arXiv:1812.07204, 2018.

\end{thebibliography}
\end{document}